\documentclass[12pt,a4paper]{amsart}
\usepackage{amsthm,amsfonts,url,amssymb}
\usepackage{mathrsfs}
\usepackage{amssymb,amscd}
\usepackage{color}
\usepackage{chemarrow}
\usepackage{pictexwd,dcpic}
\usepackage{extarrows}
\usepackage[colorlinks,linkcolor=red,anchorcolor=green,citecolor=blue]{hyperref}
\usepackage{enumitem}
\usepackage{lipsum}
\usepackage[all,cmtip]{xy}
\usepackage{leftidx}
\usepackage{tikz-cd}
\usepackage{imakeidx}
\makeatletter
\renewcommand\normalsize{%
    \@setfontsize\normalsize{11.7}{14pt plus .3pt minus .3pt}%
    \abovedisplayskip 10\p@ \@plus4\p@ \@minus4\p@
    \abovedisplayshortskip 6\p@ \@plus2\p@
    \belowdisplayshortskip 6\p@ \@plus2\p@
    \belowdisplayskip \abovedisplayskip}
\renewcommand\small{%
    \@setfontsize\small{9.5}{12\p@ plus .2\p@ minus .2\p@}%
    \abovedisplayskip 8.5\p@ \@plus4\p@ \@minus1\p@
    \belowdisplayskip \abovedisplayskip
    \abovedisplayshortskip \abovedisplayskip
    \belowdisplayshortskip \abovedisplayskip}
\renewcommand\footnotesize{%
    \@setfontsize\footnotesize{8.5}{9.25\p@ plus .1pt minus .1pt}
    \abovedisplayskip 6\p@ \@plus4\p@ \@minus1\p@
    \belowdisplayskip \abovedisplayskip
    \abovedisplayshortskip \abovedisplayskip
    \belowdisplayshortskip \abovedisplayskip}
\ifdefined\pdfpagewidth
\else
\fi
\calclayout
\makeatother

\numberwithin{equation}{section}

\newtheorem{theorem}{Theorem}[section]
\newtheorem{corollary}[theorem]{Corollary}
\newtheorem{lemma}[theorem]{Lemma}
\newtheorem{proposition}[theorem]{Proposition}

\newtheorem{hypothesis}[theorem]{Hypothesis}
\newtheorem{remark}[theorem]{Remark}

\newtheorem{example}[theorem]{Example}

\newcommand{\hooklongleftarrow}{\longleftarrow\joinrel\rhook}
\newcommand{\hooklongrightarrow}{\lhook\joinrel\longrightarrow}
\newcommand{\twoheadlongrightarrow}{\relbar\joinrel\twoheadrightarrow}

\newcommand{\ra}{\rightarrow}
\newcommand{\lra}{\longrightarrow}

\newcommand{\ul}{\underline}

\newcommand{\bA}{\mathbb A}

\newcommand{\bG}{\mathbb G}

\newcommand{\Q}{\mathbb Q}
\newcommand{\bR}{\mathbb R}

\newcommand{\bT}{\mathbb T}
\newcommand{\Z}{\mathbb Z}

\newcommand{\cL}{\mathcal L}
\newcommand{\co}{\mathcal O}

\newcommand{\cR}{\mathcal R}
\newcommand{\cH}{\mathcal H}
\newcommand{\cC}{\mathcal C}

\newcommand{\cD}{\mathcal D}
\newcommand{\cI}{\mathcal I}

\newcommand{\cT}{\mathcal T}
\newcommand{\cM}{\mathcal M}
\newcommand{\cF}{\mathcal F}

\newcommand{\cE}{\mathcal E}

\newcommand{\cJ}{\mathcal J}

\newcommand{\cZ}{\mathcal Z}

\newcommand{\fn}{\mathfrak n}
\newcommand{\fh}{\mathfrak h}
\newcommand{\fm}{\mathfrak{m}}
\newcommand{\ub}{\mathfrak b}

\newcommand{\fp}{\mathfrak p}

\newcommand{\fT}{\mathfrak T}

\newcommand{\ug}{\mathfrak g}

\newcommand{\ft}{\mathfrak t}
\newcommand{\fa}{\mathfrak a}

\newcommand{\sC}{\mathscr C}

\newcommand{\sW}{\mathscr W}

\newcommand{\sF}{\mathscr F}
\newcommand{\sI}{\mathscr I}
\newcommand{\sG}{\mathscr G}
\newcommand{\sE}{\mathscr E}
\newcommand{\sT}{\mathscr T}

\newcommand{\lin}{\rule[2.5pt]{10pt}{0.5 pt}}

\newcommand{\lan}{\langle}
\newcommand{\ran}{\rangle}

\newcommand{\ol}{\overline}

\DeclareMathOperator{\cind}{c-\mathrm{ind}}
\DeclareMathOperator{\Sen}{\mathrm Sen}

\DeclareMathOperator{\gl}{\mathfrak gl}

\DeclareMathOperator{\tr}{\mathrm tr}
\DeclareMathOperator{\GL}{\mathrm GL}
\DeclareMathOperator{\gr}{\mathrm gr}

\DeclareMathOperator{\Fil}{\mathrm Fil}
\DeclareMathOperator{\Res}{\mathrm Res}

\DeclareMathOperator{\Gal}{\mathrm Gal}
\DeclareMathOperator{\Hom}{\mathrm Hom}

\DeclareMathOperator{\End}{\mathrm End}
\DeclareMathOperator{\cris}{\mathrm cris}
\DeclareMathOperator{\rig}{\mathrm rig}
\DeclareMathOperator{\an}{\mathrm an}
\DeclareMathOperator{\Spec}{\mathrm Spec}

\DeclareMathOperator{\dR}{\mathrm dR}
\DeclareMathOperator{\Frob}{\mathrm Frob}
\DeclareMathOperator{\Ind}{\mathrm Ind}
\DeclareMathOperator{\unr}{\mathrm unr}

\DeclareMathOperator{\Ker}{\mathrm Ker}

\DeclareMathOperator{\rank}{\mathrm rank}

\DeclareMathOperator{\Ext}{\mathrm Ext}

\DeclareMathOperator{\Spf}{\mathrm Spf}
\DeclareMathOperator{\Ima}{\mathrm Im}
\DeclareMathOperator{\SL}{\mathrm SL}
\DeclareMathOperator{\lalg}{\mathrm lalg}

\DeclareMathOperator{\id}{\mathrm id}

\DeclareMathOperator{\dett}{\mathrm det}
\DeclareMathOperator{\alg}{\mathrm alg}

\DeclareMathOperator{\soc}{\mathrm soc}

\DeclareMathOperator{\nc}{\mathrm nc}

\DeclareMathOperator{\diag}{\mathrm diag}

\DeclareMathOperator{\cont}{\mathrm cont}

\DeclareMathOperator{\pst}{\mathrm pst}

\DeclareMathOperator{\PS}{\mathrm PS} 
\DeclareMathOperator{\fss}{\mathrm fs}

\DeclareMathOperator{\sm}{\mathrm sm}
\DeclareMathOperator{\tri}{\mathrm tri}

\DeclareMathOperator{\HC}{\mathrm HC}
\DeclareMathOperator{\univ}{\mathrm univ}

\begin{document}
	\title{$p$-adic Hodge parameters in the crystabelline representations of $\GL_n$}
	\author{Yiwen Ding}
	\date{}
	\maketitle
	\begin{abstract}
		Let $K$ be a finite extension of $\Q_p$, and $\rho$ be an  $n$-dimensional (non-critical generic) crystabelline  representation of  the absolute Galois group of $K$ of regular Hodge-Tate weights. We associate to $\rho$ an explicit locally $\Q_p$-analytic representation $\pi_1(\rho)$ of $\GL_n(K)$, which encodes some $p$-adic Hodge parameters of $\rho$. When $K=\Q_p$, it encodes the full information hence reciprocally determines $\rho$. When $\rho$ is associated to $p$-adic automorphic representations, we show under mild hypotheses that $\pi_1(\rho)$ is a subrepresentation of the $\GL_n(K)$-representation globally associated to $\rho$. 
	\end{abstract}
	\tableofcontents
	\section{Introduction}
	
	The locally analytic $p$-adic Langlands program for $\GL_n(\Q_p)$ aims at building a correspondence between $n$-dimensional $p$-adic continuous representations of the abosulte Galois group $\Gal_{\Q_p}$ of $\Q_p$ and certain locally analytic representations of $\GL_n(\Q_p)$.  In particular, it is expected  to match the parameters on both sides via the conjectural correspondence.
	
	On the Galois side, the $p$-adic $\Gal_{\Q_p}$-representations are central objects in the $p$-adic Hodge theory, and are classified by Fontaine's theory. Among these representations, the de Rham ones are particularly important, as they  include those arising from geometry (\cite{Fon94}). The $p$-adic Langlands program for de Rham representations  is expected to be compatible with the classical local Langlands correspondence (e.g. see \cite{BS07}). More precisely, by Fontaine's theory, for a de Rham representation $\rho$ over a $p$-adic field $E$, one can associate an $n$-dimensional Weil-Deligne representation $\mathtt{r}$, which furthermore corresponds, via the classical local Langlands correspondence, to an irreducible smooth representation $\pi_{\sm}(\mathtt{r})$ of $\GL_n(\Q_p)$ over $E$. If $\rho$ has regular Hodge-Tate weights $\textbf{h}=(h_1, \cdots, h_n)$, then the locally algebraic representation 
	\begin{equation*}
		\pi_{\alg}(\mathtt{r}, \textbf{h}):=\pi_{\sm}(\mathtt{r}) \otimes_E L(\textbf{h}-\theta)
	\end{equation*}
	is expected to be the locally algebraic subrepresentation of the conjectural locally analytic representation $\pi^{?}(\rho)$ associated to $\rho$, where $\theta=(0,-1,\cdots, 1-n)$ and $L(\textbf{h}-\theta)$ is the algebraic representation of $\GL_n(\Q_p)$ of highest weight $\textbf{h}-\theta$. One can clearly recover $\mathtt{r}$ (up to $F$-semi-simplification) and $\textbf{h}$ from the representation of $\pi_{\alg}(\mathtt{r},\textbf{h})$. However, passing from $\rho$ to $(\mathtt{r}, \textbf{h})$, one loses the information of Hodge filtration of $\rho$. A fundamental question in the $p$-adic Langlands program is to find the missing information on Hodge filtration on the automorphic side, say, in the conjectural locally analytic representation $\pi^?(\rho)$. After the pioneer work of Breuil (\cite{Br04}\cite{Br10}), the question was settled for $\GL_2(\Q_p)$ by Colmez, establishing the $p$-adic Langlands correspondence  (\cite{Colm10a}). It remains quite mysterious for general $\GL_n(\Q_p)$. In this paper, we address the question for (non-critical generic) crystabelline $\Gal_{\Q_p}$-representations $\rho$, those that become crystalline when restricted to the absolute Galois group of a certain abelian extension of $\Q_p$. 
	
	For simplicity, we assume in the introduction that  $\rho$ itself is crystalline. Then by Fontaine's theory, $\rho$ is equivalent to the associated filtered $\varphi$-module $D_{\cris}(\rho)$. We assume the $\varphi$-action is \textit{generic} (and we simply call such $\rho$  generic), which means the $\varphi$-eigenvalues $\ul{\alpha}=(\alpha_i)$ on $D_{\cris}(\rho)$ are distinct, and $\alpha_i\alpha_j^{-1}\neq p$ for $i\neq j$. In this case, $\mathtt{r}\cong \oplus_{i=1}^n \unr(\alpha_i)$ and we denote $\mathtt{r}$ by $\ul{\alpha}$. The classical local Langlands correspondence in this case is simply given by $$\pi_{\sm}(\ul{\alpha})\cong (\Ind_{B^-(\Q_p)}^{\GL_n(\Q_p)} \unr(\ul{\alpha}) \eta)^{\infty}$$ where $\unr(\ul{\alpha})=\unr(\alpha_1) \boxtimes \cdots \unr(\alpha_n)$,  $\eta=1 \boxtimes |\cdot|^1 \boxtimes \cdots |\cdot|^{n-1}$ are unramified characters of $T(\Q_p)$, and $B^-$ is the Borel subgroup of lower triangular matrices. Let $\Fil_{H}^{\bullet}$ denote the Hodge filtration, which is a complete flag in $D_{\cris}(\rho)$ as $\textbf{h}$ is regular. Let $e_{i}\in D_{\cris}(\rho)$ be an eigenvector for $\alpha_i$. Under the basis $\{e_i\}$, $\Fil_H^{\bullet}$ is parametrized by an element in $T\backslash \GL_n/B$, which we call the  \textit{$p$-adic Hodge parameter} of $\rho$. Recall that $\rho$ is called \textit{non-critical} if $\Fil_H^{\bullet}$ is in a relative general position with respect to  all the $n!$ $\varphi$-stable (complete) flags. When $n=2$, $T\backslash \GL_2/B$ is a finite set of cardinality  $3$. So there are at most $3$ isomorphism classes\footnote{The \'etaleness of $\rho$ will imply that some of these classes may not occur. In most general cases, there is typically a unique isomorphism class. But note if we relax the \'etaleness condition,  and consider crystalline $(\varphi, \Gamma)$-modules instead of $\rho$, all these classes can appear.} of $\rho$, distinguished by the relative position of $\Fil_H^{\bullet}$ with the two $\varphi$-stable flags. The information  is reflected by the extra socle phenomenon on the $\GL_2(\Q_p)$-side. In this context, Breuil formulated a conjecture concerning the locally analytic socle of $\GL_n$, which characterizes  the relative positions of $\Fil^{\bullet}_H$ with the $\varphi$-stable flags. The conjecture was subsequently proved  (under Taylor-Wiles hypotheses) by Breuil-Hellmann-Schraen (\cite{BHS3}). However, a significant  difference between the cases $n=2$ and $n\geq 3$ lies in the extra parameters for non-critical $\rho$ (with fixed $(\ul{\alpha}, \textbf{h})$): when $n=2$, the non-critical $\rho$ is unique, whereas for $n\geq 3$, there are additional (new) parameters for non-critical $\rho$ (as  $T\backslash \GL_n/B$ is now an infinite set). We refer to Example \ref{Egl3} for a concrete example of $n=3$.

	In the paper, we reveal these  $p$-adic Hodge parameters on the $\GL_n(\Q_p)$-side. It turns out  it is convenient to work with  $(\varphi, \Gamma)$-modules over the Robba ring instead of Galois representations. Denote by $\Phi\Gamma_{\nc}(\ul{\alpha}, \textbf{h})$ the set of isomorphism classes of non-critical crystalline $(\varphi, \Gamma)$-modules overlying $\ul{\alpha}$ of regular Hodge-Tate weights $\textbf{h}$. Under the basis of $\varphi$-eigenvectors  $\{e_i\}$ in the precedent paragraph (noting that $D_{\cris}(D)\cong \oplus_{i=1}^n E e_i$, as $\varphi$-module, for all $D\in \Phi\Gamma_{\nc}(\ul{\alpha},\textbf{h})$), the set $\Phi\Gamma_{\nc}(\ul{\alpha}, \textbf{h})$ can be identified with a Zariski open subset of $T\backslash \GL_n/B$. For each $D\in \Phi\Gamma_{\nc}(\ul{\alpha}, \textbf{h})$, we associate an explicit locally analytic $\GL_n(\Q_p)$-representation $\pi_1(D)$ (see Theorem \ref{Tintro2} below for the construction). We have:

	\begin{theorem}\label{Tintro1}
		(1) (Local correspondence) For $D\in \Phi\Gamma_{\nc}(\ul{\alpha}, \textbf{h})$, $\soc_{\GL_n(\Q_p)} \pi_1(D)\cong \pi_{\alg}(\ul{\alpha},\textbf{h})$, and $\pi_1(D)\twoheadrightarrow  \pi_{\alg}(\ul{\alpha}, \textbf{h})^{\oplus (2^n-\frac{n(n+1)}{2}-1)}$. Moreover, for $D'\in  \Phi\Gamma_{\nc}(\ul{\alpha}, \textbf{h})$, $\pi_{1}(D)\cong \pi_{1}(D')$ if and only if $D'\cong D$.
		
		(2) (Local-global compatibility) Suppose $\rho$ is automorphic for the setting of \cite{CEGGPS1} (or the setting in \S~\ref{S422}), and let $\widehat{\pi}(\rho)$ be the unitary Banach representation of $\GL_n(\Q_p)$ (globally) associated to $\rho$. Assume $D_{\rig}(\rho) \in \Phi\Gamma_{\nc}(\ul{\alpha},\textbf{h})$. Then for $D\in \Phi\Gamma_{\nc}(\ul{\alpha}, \textbf{h})$,
		\begin{equation*}
			\pi_{1}(D)\hookrightarrow \widehat{\pi}(\rho)^{\an} \text{ if and only if } D\cong D_{\rig}(\rho).
		\end{equation*}
	In particular, $\widehat{\pi}(\rho)^{\an}$ determines $\rho$.\footnote{Note that the information that $D_{\rig}(\rho)$ is non-critical is determined by $\widehat{\pi}(\rho)^{\an}$ by \cite[Thm.~9.3]{Br13II}.}
	\end{theorem}
	The quotient $\pi_{\alg}(\ul{\alpha}, \textbf{h})^{\oplus (2^n-\frac{n(n+1)}{2}-1)}$ of $\pi_1(D)$ appears in the ``third" layer in its  socle filtration. Let $\pi_{\min}(D)$ be the minimal subrepresentation of $\pi_1(D)$ such that the composition $\pi_{\min}(D) \hookrightarrow \pi_1(D) \twoheadrightarrow \pi_{\alg}(\ul{\alpha}, \textbf{h})^{\oplus (2^n-\frac{n(n+1)}{2}-1)}$ is surjective. The representation $\pi_{\min}(D)$ has a much cleaner structure. For example, its  socle filtration has only  three grades (see \S~\ref{S32}). Note that one can replace everywhere $\pi_1(D)$ in the statements by $\pi_{\min}(D)$. The extra locally algebraic constituents in the cosocle of $\pi_1(D)$ were unexpected, not to mention their huge multiplicity. It is one of the reasons why it took a long time to find the Hodge parameters. In fact, the work grows out from the finding of such extra constituents while excluding such constituents for $\GL_2$ in \cite{Ding15} . We remark that the existence of the extra locally algebraic constituent was first proved by Hellmann-Hernandez-Schraen  in the split case for $\GL_3(\Q_p)$ (\cite{HHS}). 	
	
	For a finite extension $K$ of $\Q_p$, we also construct a locally $\Q_p$-analytic representation $\pi_{1}(D)$ of $\GL_n(K)$ such that $\soc_{\GL_n(K)} \pi_{1}(D)\cong \pi_{\alg}(\ul{\alpha}, \textbf{h})$ and $\pi_{1}(D)\twoheadrightarrow \pi_{\alg}(\ul{\alpha}, \textbf{h})^{\oplus (2^n-\frac{n(n+1)}{2}-1) [K:\Q_p]}$. The local-global compatibility result still holds. But a major difference is that when $K\neq \Q_p$, $\pi_{1}(D)$ just determines the filtered $\varphi^f$-module $D_{\cris}(D)_{\sigma}$ (where $f$ is the unramified degree of $K$ over $\Q_p$) for each embedding $\sigma: K\hookrightarrow E$ rather than $D$ itself. For example, when $n=2$, $\pi_{1}(D)$ are all isomorphic (for  different $D\in \Phi\Gamma_{\nc}(\ul{\alpha},\textbf{h})$) but there are still extra parameters, see for example \cite[\S~3]{Br} \cite[Conj.~1.7]{Ding15}.

	We make a few additional remarks on Theorem \ref{Tintro1}.
	
	\begin{remark}\label{Rintro1}
		(1) Very little was known about such a local correspondence when $n\geq 3$.  We highlight some related results. When $n=3$, in \cite{BD1}, we showed  how to recover the Hodge parameters in the semi-stable non-crystalline case (given by the Fontaine-Mazur $\cL$-invariants) in the locally analytic $\GL_3(\Q_p)$-representations and proved a local-global compatibility result in the ordinary case. When the Weil-Deligne representation $\mathtt{r}$ associated to $\rho$ is indecomposable, the (largely open) conjecture on $\Ext^1$ in \cite{Br16} (see also \cite{BD2}) suggests a way to recover the $p$-adic Hodge parameters on the automorphic side. In contrast, the (non-critical) crystalline case was somewhat more mysterious, as such parameters are entirely new for $n\geq 3$.  We finally mention that the results for $\GL_3(\Q_p)$ were presented in the note \cite{Ding16} (not intended for publication), which may help  readers quickly understand the story.

		
		(2) The phenomenon where the Hodge parameters lie in the extension group of certain locally algebraic representation by certain locally analytic representation traces back to Breuil's initializing work in \cite{Br04}.
		

		(3) Similar results are also obtained in the  patched setting. Let $\Pi_{\infty}$ be the patched Banach representation  over the patched Galois deformation ring $R_{\infty}$  of \cite{CEGGPS1}.  We  show that if there is a maximal ideal $\fm_{\rho}$ of $R_{\infty}[1/p]$ associated to $\rho$ such that 
		$\Pi_{\infty}[\fm_{\rho}]^{\lalg} \neq 0$, then for $D\in \Phi\Gamma_{\nc}(\ul{\alpha}, \textbf{h})$,
	$\pi_{1}(D) \hookrightarrow \Pi_{\infty}[\fm_{\rho}]$ if and only if $D\cong D_{\rig}(\rho)$.		

		(4) We finally remark the representation $\pi_{1}(D)$ should still be  far from the final complete locally analytic $\GL_n(\Q_p)$-representation associated to $D$ (so we choose not to use the notation $\pi(D)$).
	\end{remark}
	We now give the construction of $\pi_{1}(D)$. We first look at the Galois side. For each $w\in S_n$, let $\Ext^1_w(D, D)$ be the extension group of trianguline deformations of $D$ with respect to the refinement $w(\ul{\alpha})$ 
	(see the discussion above (\ref{Ekappaw})). Recall there is a natural (weight) map
$\kappa_w: \Ext^1_w(D, D) \ra \Hom(T(\Q_p),E)$, 
	sending $\widetilde{D}$ to $\psi$ such that $\widetilde{D}$ is trianguline with  parameter $\unr(w(\ul{\alpha}))z^{\textbf{h}} (1+\psi \epsilon)$ (that is a character of $T(\Q_p)$ over $E[\epsilon]/\epsilon^2$).  The map $\kappa_w$ is surjective (e.g. see \cite[Prop.~2.3.10]{BCh}). One can show that $\Ker \kappa_w$, as a subspace of $\Ext^1_{(\varphi, \Gamma)}(D, D)$, is independent of the choice of $w$, denoted by $\Ext^1_0(D, D)$ (cf. Lemma \ref{Lint1}). For a subspace $\Ext^1_?(D, D)\subset \Ext^1_{(\varphi, \Gamma)}(D, D)$ containing $\Ext^1_0(D, D)$, set $$\ol{\Ext}^1_?(D, D):=\Ext^1_?(D, D)/\Ext^1_0(D,D).$$ We have hence a bijection 
	\begin{equation*}
		\kappa_w: \ol{\Ext}^1_w(D, D) \xlongrightarrow{\sim} \Hom(T(\Q_p), E).
	\end{equation*}
	By \cite{Che11},  the following ``amalgamating" map is surjective (see also \cite{Kis10} \cite{GoMa}, noting it is already surjective before quotienting by $\Ext^1_0(D,D)$ on both sides)
	\begin{equation}\label{Eintro1}
		\oplus_{w\in S_n} \ol{\Ext}^1_w(D, D) \twoheadlongrightarrow \ol{\Ext}^1_{(\varphi, \Gamma)}(D, D).
	\end{equation}
Remark that here we use that all the refinements of $D$ are non-critical.

	Now we look at the $\GL_n(\Q_p)$-side. For each $w$, consider the locally analytic principal series
	$\PS(w,\ul{\alpha}, \textbf{h}):=(\Ind_{B^-(\Q_p)}^{\GL_n(\Q_p)} \unr(w(\ul{\alpha})) z^{\textbf{h}} \varepsilon^{-1} \circ \theta)^{\an}$, where $\varepsilon$ denotes the cyclotomic character. 
	The explicit structure of $\PS(w,\ul{\alpha}, \textbf{h})$ is well-understood by Orlik-Strauch (\cite{OS}). For example,  $\soc_{\GL_n(\Q_p)} \PS(w,\ul{\alpha}, \textbf{h})\cong \pi_{\alg}(\ul{\alpha}, \textbf{h})$, which has multiplicity one as irreducible constituent of $\PS(w,\ul{\alpha}, \textbf{h})$. 
	For $w\in S_n$, consider the composition
	\begin{multline*}
		\zeta_w: \Hom(T(\Q_p),E) \lra \Ext^1_{\GL_n(\Q_p)}(\PS(w,\ul{\alpha}, \textbf{h}), \PS(w,\ul{\alpha}, \textbf{h}))\\  \lra \Ext^1_{\GL_n(\Q_p)}(\pi_{\alg}(\ul{\alpha}, \textbf{h}), \PS(w,\ul{\alpha},\textbf{h})),
	\end{multline*}
	where the first map sends $\psi$ to $(\Ind_{B^-(\Q_p)}^{\GL_n(\Q_p)} \unr(w(\ul{\alpha})) z^{\textbf{h}} (\varepsilon^{-1} \circ \theta)(1+\psi \epsilon))^{\an}$, and the second map is the natural pull-back map. Using Schraen's spectral sequence (\cite[(4.37)]{Sch11}), one can show that  $\zeta_w$ is in fact bijective.  Now we amalgamate these principal series: let $\pi(\ul{\alpha}, \textbf{h})$ be the unique quotient of the amalgamation $\oplus_{\pi_{\alg}(\ul{\alpha}, \lambda)}^{w\in S_n} \PS(w,\ul{\alpha}, \textbf{h})$ of socle $\pi_{\alg}(\ul{\alpha}, \lambda)$ (which was introduced and denoted by $\pi(D)^{\fss}$ in \cite[Def.~5.7]{BH2}). For each $w\in S_n$, there is a natural injection $\PS(w,\ul{\alpha}, \textbf{h})\hookrightarrow \pi(\ul{\alpha}, \textbf{h})$ which induces an injection
	\begin{equation*}
		\Ext^1_{\GL_n(\Q_p)}(\pi_{\alg}(\ul{\alpha}, \textbf{h}), \PS(w,\ul{\alpha}, \textbf{h})) \hooklongrightarrow \Ext^1_{\GL_n(\Q_p)}(\pi_{\alg}(\ul{\alpha}, \textbf{h}), \pi(\ul{\alpha},\textbf{h})).
	\end{equation*}
	We denote by $\Ext^1_w(\pi_{\alg}(\ul{\alpha}, \textbf{h}), \pi(\ul{\alpha}, \textbf{h}))$ its image. The following ``amalgamating" map is also surjective (see Proposition \ref{Pextpi1} (2) and compare with (\ref{Eintro1})):
	\begin{equation}\label{Eintro3}
		\oplus_{w\in S_n}  \Ext^1_w(\pi_{\alg}(\ul{\alpha}, \textbf{h}), \pi(\ul{\alpha}, \textbf{h})) \twoheadlongrightarrow \Ext^1_{\GL_n(\Q_p)}(\pi_{\alg}(\ul{\alpha}, \textbf{h}), \pi(\ul{\alpha}, \textbf{h})).
	\end{equation}
The following theorem is crucial in the paper:
	\begin{theorem}[cf. Theorem \ref{TtD}, Theorem \ref{Thodgeauto}]\label{Tintro2}(1) For $D\in \Phi\Gamma_{\nc}(\ul{\alpha}, \textbf{h})$, there is a unique (surjective) map 
		$	t_D:	\Ext^1_{\GL_n(\Q_p)}(\pi_{\alg}(\ul{\alpha}, \textbf{h}), \pi(\ul{\alpha}, \textbf{h})) \twoheadrightarrow  \ol{\Ext}^1_{(\varphi, \Gamma)}(D,D)$
		such that the following diagram commutes:
		\begin{equation*}
			\begin{CD}
				\oplus_{w\in S_n}	\ol{\Ext}^1_w(D, D) @> (\zeta_w \circ \kappa_w) > \sim > \oplus_{w\in S_n} \Ext^1_w(\pi_{\alg}(\ul{\alpha}, \textbf{h}), \pi(\ul{\alpha}, \textbf{h})) \\
				@V (\ref{Eintro1}) VV @V (\ref{Eintro3}) VV \\
				\ol{\Ext}^1_{(\varphi, \Gamma)}(D,D) @< t_{D} << 	\Ext^1_{\GL_n(\Q_p)}(\pi_{\alg}(\ul{\alpha}, \textbf{h}), \pi(\ul{\alpha}, \textbf{h})).
			\end{CD}
		\end{equation*}
		Moreover, $\dim_E \ol{\Ext}^1_{(\varphi, \Gamma)}(D,D)=\frac{n(n+1)}{2}+n$, $\dim_E\Ext^1_{\GL_n(\Q_p)}(\pi_{\alg}(\ul{\alpha}, \textbf{h}), \pi(\ul{\alpha}, \textbf{h}))=2^n+n-1$ hence $\dim_E \Ker(t_{D})=2^n-\frac{n(n+1)}{2}-1$.
		
		(2) For $D, D'\in \Phi\Gamma_{\nc}(\ul{\alpha}, \textbf{h})$, $\Ker(t_D)=\Ker(t_{D'})$ if and only if $D\cong D'$.
	\end{theorem}
\begin{remark}
	Consider the composition 
	\begin{equation}\label{Ekappaw-1}
		\oplus_{w\in S_n} \Hom(T(\Q_p),E) \xlongrightarrow[\sim]{(\kappa_w^{-1})}  \oplus_{w\in S_n} \ol{\Ext}^1_w(D,D) \xlongrightarrow{(\ref{Eintro1})} \ol{\Ext}^1_{(\varphi, \Gamma)}(D,D).
	\end{equation}
By Theorem \ref{Tintro2} (1), the map (\ref{Eintro3}) induces an exact sequence
\begin{equation*}
	0 \lra \Ker (\ref{Eintro3}) \lra(\zeta_w)_{w\in S_n}(\Ker(\ref{Ekappaw-1})) \lra \Ker (t_D)\lra 0.
\end{equation*}As  the maps $(\ref{Eintro3})$ and $\zeta_w$'s are all independent of $D\in \Phi\Gamma_{\nc}(\ul{\alpha},\textbf{h})$,  Theorem \ref{Tintro2} (2) implies that   $\Ker (\ref{Ekappaw-1})$ 
also determines $D$. This fact (purely on Galois side) is of interest on its own right.
\end{remark}
	The representation $\pi_1(D)$ is then defined to  be the (tautological) extension of $\pi_{\alg}(\ul{\alpha}, \textbf{h}) \otimes_E \Ker(t_{D})\cong \pi_{\alg}(\ul{\alpha},\textbf{h})^{\oplus (2^n-\frac{n(n+1)}{2}-1)}$ by $\pi(\ul{\alpha},\textbf{h})$. More precisely, choosing a basis $\{v_i\}$ of $\Ker(t_D)$ with $\sE(v_i)$ the associated extension of  $\pi_{\alg}(\ul{\alpha},\textbf{h})$ by $\pi(\ul{\alpha},\textbf{h})$, $\pi_1(D)$ is the amalgamated sum of these $\sE(v_i)$ along $\pi(\ul{\alpha},\textbf{h})$, which is clearly independent of the choice of $\{v_i\}$. The structure of $\pi(\ul{\alpha}, \textbf{h})$ is  complicated (see for example \cite[\S~5.3]{BH2}). However, the theorem actually holds  with $\pi(\ul{\alpha}, \textbf{h})$ replaced by its subrepresentation given by the first two layers in its socle filtration, which  has a much easier and cleaner structure, see Theorem \ref{TtD} and \S~\ref{Sps}.  The extension of $\pi_{\alg}(\ul{\alpha}, \textbf{h}) \otimes_E \Ker(t_{D})$ by this subrepresentation actually gives $\pi_{\min}(D)$ in the discussion below Theorem \ref{Tintro1}. Theorem \ref{Tintro1} (1) is then a direct consequence of Theorem \ref{Tintro2}.
	
	One can deduce from Theorem \ref{Tintro2} (1):
	\begin{corollary}[cf. Corollary \ref{CtD}]\label{Cintro1}
		The map $t_{D}$ induces a bijection
		\begin{equation}\label{EtDintro}
			t_{D}: \Ext^1_{\GL_n(\Q_p)}(\pi_{\alg}(\ul{\alpha}, \textbf{h}), \pi_1(D)) \xlongrightarrow{\sim} \ol{\Ext}^1_{(\varphi, \Gamma)}(D,D).
		\end{equation}
	\end{corollary}

	Before discussing the proof of Theorem \ref{Tintro2}, we first explain the proof of the  local-global compatibility (Theorem \ref{Tintro1} (2)). For this, we will use an alternative formulation of Theorem \ref{Tintro2} given as follows. Let $\pi^{\univ}$ (resp. $\pi_w^{\univ}$) be the (universal) extension of $\pi_{\alg}(\ul{\alpha}, \textbf{h}) \otimes_E \Ext^1_{\GL_n(\Q_p)}(\pi_{\alg}(\ul{\alpha}, \textbf{h}), \pi(\ul{\alpha}, \textbf{h}))$ \big(resp. of $\pi_{\alg}(\ul{\alpha}, \textbf{h}) \otimes_E \Ext^1_{w}(\pi_{\alg}(\ul{\alpha}, \textbf{h}), \pi(\ul{\alpha}, \textbf{h}))$\big) by $\pi(\ul{\alpha}, \textbf{h})$ (defined in a similar way as in the discussion below Theorem \ref{Tintro2}). By (\ref{Eintro3}), $\pi^{\univ}$ is generated by all the subrepresentations $\pi^{\univ}_w$ for $w\in S_n$. On the Galois side, let $R_D$ be the universal deformation ring of deformations of $D$ over Artinian local $E$-algebras and $\fm$ be its maximal ideal. The quotient  $\ol{\Ext}^1_{(\varphi, \Gamma)}(D, D)$ corresponds to a local Artinian $E$-subalgebra $A_{D}$ of $R_{D}/\fm^2$, and $\ol{\Ext}^1_{w}(D, D)$ corresponds to a quotient $A_{D,w}$ of $A_{D}$. Using the isomorphism $\zeta_w \circ \kappa_w$, there exists a natural action of $A_{D,w}$ on $\pi_w^{\univ}$ such that $x\in \fm_{A_{D,w}}/\fm_{A_{D,w}}^2\cong \ol{\Ext}^1_{w}(D, D)^{\vee}\cong \Ext^1_w(\pi_{\alg}(\ul{\alpha},\textbf{h}), \pi(\ul{\alpha},\textbf{h}))^{\vee}$ acts via 
	\begin{equation*}
		\pi_w^{\univ} \twoheadlongrightarrow \pi_{\alg}(\ul{\alpha}, \textbf{h}) \otimes_E \Ext^1_{w}(\pi_{\alg}(\ul{\alpha}, \textbf{h}), \pi(\ul{\alpha}, \textbf{h})) \xlongrightarrow{x} \pi_{\alg}(\ul{\alpha}, \textbf{h}) \hooklongrightarrow \pi_w^{\univ}.
	\end{equation*}The following corollary gives a reformulation of Theorem \ref{Tintro2} (1).
	\begin{corollary}[cf. Theorem \ref{Tuniv}, Corollary \ref{Cpimin}]\label{Cintro2}
		There exists a unique action of $A_{D}$ on $\pi^{\univ}$ such that for each $w\in S_n$, the $A_{D}$-action on its subrepresentation $\pi^{\univ}_w$ factors through the natural  $A_{D,w}$-action. Moreover, we have $\pi_1(D)\cong \pi^{\univ}[\fm_{A_D}]$.
	\end{corollary}
	Suppose we are in the patched setting as in Remark \ref{Rintro1} (4), and let $D=D_{\rig}(\rho)$. Let $\fa$ be an ideal of $R_{\infty}[1/p]$ with $\fa\supset \fm_{\rho}^2$ (cf. Remark \ref{Rintro1} (3)) such that the composition $A_D \ra R_D/\fm^2 \ra R_{\infty}[1/p]/\fa$ is an isomorphism (see the discussion below (\ref{Ecartiwg1})). 
	Working with the patched eigenvariety of \cite{BHS1}, and using Emerton's adjunction formula \cite{Em2},  we can  obtain  $A_{D} \times \GL_n(\Q_p)$-equivariant injections $\pi^{\univ}_w\hookrightarrow \Pi_{\infty}[\fa]$ for all $w\in S_n$,  where the $A_{D}$-action on the  right hand side comes from the $R_{\infty}$-action (noting $R_D$ is isomorphic to the universal Galois deformation ring of $\rho$). These injections ``amalgamate" to an $A_{D} \times \GL_n(\Q_p)$-equivariant injection 
	\begin{equation}\label{Eintro4}
		\pi^{\univ} \hooklongrightarrow \Pi_{\infty}[\fa].
	\end{equation} 
	By Corollary \ref{Cintro2}, it induces an injection $\iota: \pi_1(D)\cong \pi^{\univ}[\fm_{A_D}] \hookrightarrow \Pi_{\infty}[\fa+ \fm_{A_D}]\cong \Pi_{\infty}[\fm_{\rho}]$. Now for $D'\in \Phi\Gamma_{\nc}(\ul{\alpha}, \textbf{h})$, if $\pi_1(D')\hookrightarrow \Pi_{\infty}[\fm_{\rho}]$, one can prove (cf. the proof of Corollary \ref{Cmaxmin}) that  it factors through the injection (\ref{Eintro4}), i.e. we have 
	$\pi_1(D') \hookrightarrow \pi^{\univ} \hookrightarrow \Pi_{\infty}[\fa]$. 
	As $A_{D}$($\hookrightarrow R_D/\fm^2$) acts on $\Pi_{\infty}[\fm_{\rho}]$ hence on its sub $\pi_1(D')$ via $A_D/\fm_{A_D}$ and (\ref{Eintro4}) is $A_D$-equivariant, $\pi_1(D')\hookrightarrow \pi^{\univ}$ has image contained in $\pi^{\univ}[\fm_{A_D}]\cong \pi_1(D)$. Since $\pi_1(D')$ and $\pi_1(D)$ have the same irreducible constituents with the same multiplicities, this implies $\pi_1(D')\xrightarrow{\sim} \pi_1(D)$. 
	
	We now  discuss the proof of Theorem \ref{Tintro2}. First,  the case of $n=2$ is clear, as now  $\# \Phi\Gamma_{\nc}(\ul{\alpha}, \textbf{h})=1$,  $t_D$ is bijective, and $\pi_1(D)\cong \pi(\ul{\alpha}, \textbf{h})$ (which is the locally analytic $\GL_2(\Q_p)$-representation associated to $D$, see \cite{Liu12} \cite{Colm14}).  For general $n\geq 3$, we use an induction argument. For simplicity, in the rest of the introduction, we restrict to the case of $n=3$. This case already presents the key arguments. Let $D_1$ (resp. $C_1$) be the (unique) non-critical $(\varphi, \Gamma)$-module of rank $2$ over $\cR_{E}$ of refinement $\ul{\alpha}^1:=(\alpha_1,\alpha_2)$ and of Hodge-Tate weights $\textbf{h}^1:=(h_1> h_2)$ (resp. $\textbf{h}^2:=(h_2>h_3)$). Then for any $D\in \Phi\Gamma_{\nc}(\ul{\alpha},\textbf{h})$, $D$ admits  two filtrations:
	\begin{eqnarray*}
		&\sF:&  0 \lra D_1 \lra D  \lra \cR_{E}(\unr(\alpha_3)z^{h_3}) \lra  0, \\
		&\sG: & 0 \lra  \cR_E(\unr(\alpha_3) z^{h_1}) \lra  D \lra C_1 \lra 0.
	\end{eqnarray*}
	 Similarly as in (\ref{Eintro1}) by considering the paraboline deformations with respect to $\sF$ and $\sG$, we have  a natural map
	\begin{equation}\label{Eintro2}
	f_D=(f_{\sF}, f_{\sG}):	\ol{\Ext}^1_{(\varphi, \Gamma)}(D_1,D_1) \oplus \ol{\Ext}^1_{(\varphi, \Gamma)}(C_1,C_1) \lra \ol{\Ext}^1_{(\varphi, \Gamma)}(D,D),
	\end{equation}
	sending $\widetilde{D}_1$ (resp. $\widetilde{C}_1$) to  a (or any) deformation $\widetilde{D}$ of $D$ of the form (whose image in $\ol{\Ext}^1_{(\varphi, \Gamma)}(D,D)$ does not depend on the choice):
$0 \ra \widetilde{D}_1 \ra \widetilde{D} \ra\cR_{E[\epsilon]/\epsilon^2}(\unr(\alpha_3)z^{h_3}) \ra 0$ (resp. $0 \ra \cR_{E[\epsilon]/\epsilon^2}(\unr(\alpha_3) z^{h_1}) \ra \widetilde{D} \ra \widetilde{C}_1 \ra 0$).
	The kernel of (\ref{Eintro2}) is particularly important for our application. For $(\widetilde{D}_1, \widetilde{C}_1)\in \Ker(\ref{Eintro2})$, let  $\widetilde{D}$ be a deformation of $D$ whose image in $\ol{\Ext}^1_{(\varphi,\Gamma)}(D,D)$ is equal to $f_{\sF}(\widetilde{D}_1)=-f_{\sG}(\widetilde{C}_1)$. Then $\widetilde{D}$ admits two different parabolic filtrations (of saturated $(\varphi, \Gamma)$-submodules over $\cR_{K, E[\epsilon]/\epsilon^2}$). We refer to this as a \textit{higher intertwining} property (see \S~\ref{S24}). The following theorem is purely on Galois side, and follows from  an explicit description of $\Ker(\ref{Eintro2})$ together with a reinterpretation of the $p$-adic Hodge parameters of $D$ given in \S~\ref{S22}.
		\begin{theorem}[cf.  Corollary \ref{ChIw}] \label{Tintro3} For $D,D'\in \Phi\Gamma_{\nc}(\ul{\alpha},\textbf{h})$, $D\cong D'$ if and only if $\Ker(f_D)=\Ker(f_{D'})$.
		%
	\end{theorem}
We move to the automorphic side.  Using parabolic inductions, one can show there is a natural map \begin{multline}\label{Ezeta00}
	\zeta: \Ext^1_{\GL_2}(\pi_{\alg}(\ul{\alpha}^1, \textbf{h}^1), \pi_1(D_1)) \oplus \Ext^1_{\GL_2}(\pi_{\alg}(\ul{\alpha}^1, \textbf{h}^2), \pi_1(C_1))\\  \xlongrightarrow{(\zeta_{\sF}, \zeta_{\sG})}\Ext^1_{\GL_3}(\pi_{\alg}(\ul{\alpha}, \textbf{h}), \pi(\ul{\alpha}, \textbf{h})).
\end{multline}
For example, $\zeta_{\sF}$ is constructed using $(\Ind_{P^-}^{\GL_3} - \boxtimes \unr(\alpha_3) \varepsilon^2)^{\an}$, for $P^-=\big(\begin{smallmatrix} \GL_2 & 0 \\ * & \GL_1\end{smallmatrix}\big)$, and $\zeta_{\sG}$ uses $\big(\begin{smallmatrix} \GL_1& 0 \\ * & \GL_2\end{smallmatrix}\big)$. Moreover,  (\ref{Ezeta00}) is surjective (roughly because $\pi(\ul{\alpha}, \textbf{h})$ can be ``amalgamated" from the two corresponding parabolic inductions). We refer to Proposition \ref{PFGfil} for details. 

Now a key fact is that for any $D\in \Phi\Gamma_{\nc}(\ul{\alpha},\textbf{h})$, $\Ker(\zeta)$ is sent to $\Ker f_D$ (cf. (\ref{Eintro2})) via the isomorphism for $n=2$ (cf. (\ref{EtDintro})):
\begin{multline*}
	t_{D_1, C_1}:\Ext^1_{\GL_2}(\pi_{\alg}(\ul{\alpha}^1, \textbf{h}^1), \pi_1(D_1)) \oplus \Ext^1_{\GL_2}(\pi_{\alg}(\ul{\alpha}^1, \textbf{h}^2), \pi_1(C_1))\\
	\xlongrightarrow[\sim]{(t_{D_1}, t_{C_1})} \ol{\Ext}^1_{(\varphi, \Gamma)}(D_1,D_1) \oplus \ol{\Ext}^1_{(\varphi, \Gamma)}(C_1,C_1).
\end{multline*}
The map $t_D$ in Theorem \ref{Tintro2} (1) can now be easily constructed: there is a unique map $t_D$ such that the following diagram commutes
\begin{equation*}
	\begin{tikzcd}
	&\substack{\Ext^1_{\GL_2}(\pi_{\alg}(\ul{\alpha}^1, \textbf{h}^1), \pi_1(D_1)) \\
		\oplus  \Ext^1_{\GL_2}(\pi_{\alg}(\ul{\alpha}^1, \textbf{h}^2), \pi_1(C_1))} \arrow[r, two heads, "\zeta"] \arrow[d, "t_{D_1,C_1}"', "\sim"] &  \Ext^1_{\GL_3}(\pi_{\alg}(\ul{\alpha}, \textbf{h}), \pi(\ul{\alpha}, \textbf{h})) \arrow[d, two heads, "t_D"] \\
	& \substack{\ol{\Ext}^1_{(\varphi, \Gamma)}(D_1,D_1) \\ \oplus \ol{\Ext}^1_{(\varphi, \Gamma)}(C_1,C_1)} \arrow[r, two heads, "f_D"] & \ol{\Ext}^1(D,D).
	\end{tikzcd}
\end{equation*}
It is not very difficult to check $t_D$ satisfies the properties in Theorem \ref{Tintro2} (1), and we refer to the proof of Theorem \ref{TtD} for details. Theorem \ref{Tintro2} (2) is then a consequence of  Theorem \ref{Tintro3}, as $\Ker(t_D)=\zeta(t_{D_1,C_1}^{-1} (\Ker(f_D)))$  (noting $t_{D_1,C_1}(\Ker(\zeta))\subset \Ker(f_{D'})$ for all $D'\in \Phi\Gamma_{\nc}(\ul{\alpha}, \textbf{h})$). Remark the existence of the extra one copy of $\pi_{\alg}(\ul{\alpha}, \textbf{h})$ in $\pi_1(D)$ (for $n=3$) then comes from the fact $\dim_E \Ker(f_D)=\dim_E \Ker(\ref{Ezeta00})+1$.

	We refer to the the context for  the more precise and detailed statements. One main difference from what's discussed in the introduction is that we mainly work with $\pi_{\min}(D)$ instead of $\pi_1(D)$ in the introduction, which has a cleaner structure but requires a bit more on Orlik-Strauch representations. 
	\subsection*{Acknowledgement}
	I thank Xiaozheng Han, Zicheng Qian, Zhixiang Wu for helpful discussions. I especially thank Christophe Breuil for the discussions and his interest, which  substantially accelerates the work, and for his comments on a preliminary version of the paper. I heartily thank the anonymous referee for the  careful reading, and valuable suggestions, which greatly improve the clarity of the paper. This work is supported by   the	 NSFC Grant No. 8200800065, No. 8200907289 and No. 8200908310, and partially supported by the New Cornerstone Science Foundation.
	\section{Hodge filtration and higher intertwining}
	
	\subsection{Notation and preliminaries}\label{S2.1}
	
	Let $K$ be a finite extension of $\Q_p$, $E$ be a finite extension of $\Q_p$ containing all the embeddings of $K$ in $\overline{\Q_p}$. Let $\Sigma_K:=\{\sigma: K \hookrightarrow E\}$, and $d_K:=[K:\Q_p]$.
	For $\textbf{k}=(k_{\sigma})_{\sigma\in \Sigma_K}\in \Z^{\Sigma_K}$, denote by $z^{\textbf{k}}:=\prod_{\sigma\in \Sigma_K} \sigma(z)^{k_{\sigma}}$ \ the \ ($\Q_p$-)algebraic \ character \ of \ $K^{\times}$ \ of \ weight \ $\textbf{k}$. Let $|\cdot|_K: K^{\times} \ra E^{\times}$ be the unramified character such that $|\varpi_K|_K=p^{-[K_0:\Q_p]}$ for  a uniformizer $\varpi_K$ of $K$, where $K_0$ is the maximal unramified subextension of $K$ over $\Q_p$. Let $\Gal_K$ be the absolute Galois group of $K$, and $\varepsilon: \Gal_K \ra \Z_p^{\times} \ra E^{\times}$ be the cyclotomic character. We normalize the local class field theory by sending a uniformizer to a (lift of the) geometric
	Frobenius. In this way, we view $\varepsilon$ as a character of $K^{\times}$, which is equal to $N_{K/\Q_p}(\cdot)|\cdot|_K$.
	
	For a locally $K$-analytic group $H$ (e.g. $H=K^{\times}$), set $\Hom(H,E)$ to be the $E$-vector space of locally $\Q_p$-analytic $E$-valued characters on $H$, $\Hom_{\sm}(H,E)$ the subspace of smooth (i.e. locally constant) $E$-valued characters on $H$. Let $\fh$ be the Lie algebra of $H$ (over $K$). For $\chi\in \Hom(H,E)$, by derivation, it induces a \textit{$\Q_p$-linear} map $\fh \ra E$ hence an $E$-linear map $d\chi: \fh \otimes_{\Q_p} E \ra E$. It is clear that $\chi\in \Hom_{\sm}(H,E)$ if and only if $d\chi=0$. For $\sigma\in \Sigma_K$, we call $\chi$ locally $\sigma$-analytic if $d\chi$ factors through $\fh \otimes_{K,\sigma} E \ra E$. Set $\Hom_{\sigma}(H,E)\subset \Hom(H,E)$ to be the subspace  of locally $\sigma$-analytic characters. Note we have $\dim_E \Hom_{\sm}(K^{\times},E)=1$, $\dim_E \Hom_{\sigma}(K^{\times},E)=2$ and $\dim_E \Hom(K^{\times},E)=1+d_K$.
		
	Let $\cR_{K,E}$ be the $E$-coefficient Robba ring for $K$. For a continuous character $\chi: K^{\times} \ra E^{\times}$, denote by $\cR_{K,E}(\chi)$ the associated rank one $(\varphi, \Gamma)$-module over $\cR_{K,E}$ (see for example \cite[\S~6.2]{KPX}). Note $\cR_{K,E}(\chi)$ is de Rham if and only if $\chi$ is locally algebraic. We write $\Ext^i$ (and $\Hom=\Ext^0$)  without ``$(\varphi, \Gamma)$" in the subscript for the $i$-th extension group of $(\varphi, \Gamma)$-modules (cf. \cite{Liu07}). For de Rham $(\varphi, \Gamma)$-modules $M$ and $N$, denote by $\Ext^1_g(M,N)\subset \Ext^1(M,N)$ the subspace of de Rham extensions.  For a $(\varphi, \Gamma)$-module $M$,  we identify \textit{elements} in $\Ext^1(M,M)$ with deformations of $M$ over $\cR_{K, E[\epsilon]/\epsilon^2}$. Indeed, the $E[\epsilon]/\epsilon^2$-structure on $\widetilde{M}\in \Ext^1(M,M)$ is given by letting $\epsilon$ act via
$\epsilon: \widetilde{M} \twoheadrightarrow M \xrightarrow{\id} M \hookrightarrow \widetilde{M}$.

	We denote by $W_{\dR}^+(M)$ the (semi-linear) $\Gal_K$-representation over $B_{\dR}^+ \otimes_{\Q_p} E$ associated to $M$ (cf. \cite[Prop.~2.2.6 (2)]{Ber08II}). There is a natural decomposition $W_{\dR}^+(M)\cong \oplus_{\sigma\in \Sigma_K} W_{\dR,\sigma}^+(M)$ with respect to $B_{\dR}^+ \otimes_{\Q_p} E\cong \oplus_{\sigma \in \Sigma_K} B_{\dR}^+ \otimes_{K, \sigma} E$. Denote by $D_{\dR}^+(M):=W_{\dR}^+(M)^{\Gal_K}\cong \oplus_{\sigma\in \Sigma_K} W_{\dR,\sigma}^+(M)^{\Gal_K}=:\oplus_{\sigma\in \Sigma_K} D_{\dR}^+(M)_{\sigma}$. We will frequently use the following lemma.
	\begin{lemma}\label{Lemtorsion}
		Let $M$ be a $(\varphi, \Gamma)$-module over $\cR_{K,E}$, $N$ be a $(\varphi, \Gamma)$-submodule of $M$ such that  $\rank_{\cR_{K,E}}(N)=\rank_{\cR_{K,E}}(M)$. Then there is a natural isomorphism of $E$-vector spaces:
		$H^0_{(\varphi, \Gamma)}(M/N)  \xrightarrow{\sim} H^0(\Gal_K, W_{\dR}^+(M)/W_{\dR}^+(N))$. Moreover, when $M$ is de Rham, this isomorphism identifies $H^0_{(\varphi, \Gamma)}(M/N) $ with  $D_{\dR}^+(M)/D_{\dR}^+(N)$.
	\end{lemma}
	\begin{proof} The first part follows from a straightforward generalization of the proof of \cite[Lem.~5.1]{BD2} to finite extensions $K$ of $\Q_p$. For the second part, applying $(-)^{\Gal_K}$ to the exact sequence of $B_{\dR}^+$-representations $0 \ra W_{\dR}^+(N) \ra W_{\dR}^+(M) \ra W_{\dR}^+(M)/W_{\dR}^+(N) \ra 0$, it suffices to show the natural map $H^1(\Gal_K, W_{\dR}^+(N)) \ra H^1(\Gal_K, W_{\dR}^+(M))$ is injective. But this follows from \cite[Lem.~2.6]{Na}.
		\end{proof}
	Let $M$ be a crystabelline $(\varphi, \Gamma)$-module of rank $d$ over $\cR_{K,E}$. We can associate to $M$ a filtered Deligne-Fontaine module $(D_{\pst}(M), D_{\dR}(M))$ such that 
	\begin{itemize}\item $D_{\pst}(M)=(W_e(M) \otimes_{B_e} B_{\cris})^{\Gal_{K'}}$ which is free of rank $d$ over $K_0' \otimes_{\Q_p} E$ equipped with a commuting $K_0'$-semi-linear action of $\varphi$ and $\Gal(K'/K)$, $K'$ is an \textit{abelian} extension of $K$, and $K_0'$ is the maximal unramified extension of $K'$ (over $\Q_p$), and where $W_e(M)$ is the $B_e=B_{\cris}^{\varphi=1}$-representation associated to $M$ (\cite[Prop.~2.2.6 (1)]{Ber08II}),
		\item $D_{\dR}(M):=(W_{\dR}^+(M)[1/t])^{\Gal_K}\cong (D_{\pst}(M)\otimes_{K_0'} K')^{\Gal(K'/K)}$ is free of rank $d$ over $K\otimes_{\Q_p} E$, equipped with a Hodge filtration $\Fil_{H}$ of $K \otimes_{\Q_p} E$-submodules (not necessarily free).
	\end{itemize}	
	By \cite[Prop.~4.1]{BS07}, to $D_{\pst}(M)$, one can associate a Weil-Deligne representation $\mathtt{r}(M)$ over $E$. We call $M$ \textit{generic} if $\mathtt{r}(M)$  generic, which means $\mathtt{r}(M)$ is  semi-simple and isomorphic to $\oplus_{i=1}^d \phi_i$ with $\phi_i \phi_j^{-1} \neq 1, |\cdot|_K$ for $i\neq j$.
	In fact,  $M$ being generic crystabelline is equivalent to the existence of smooth characters $\phi_i$ for $i=1, \cdots, d$ such that $M[1/t]\cong \oplus_{i=1}^d \cR_{K,E}(\phi_i)[1/t]$, and $\phi_i\phi_j^{-1} \neq 1, |\cdot|_K$ for $i\neq j$. An ordering of $(\phi_1, \cdots, \phi_d)$ is refereed to as a \textit{refinement} of $M$. Indeed,  an ordering $w(\phi)=(\phi_{w^{-1} (1)}, \cdots, \phi_{w^{-1}(d)})$ for $w\in S_d$, corresponds uniquely to a filtration $\sT_w=\{\sT_w^i\}$, increasing with $i$, of saturated $(\varphi, \Gamma)$-submodules on $M$ such that $(\gr_{\sT_w}^i M)[1/t]\cong \cR_{K,E}(\phi_{w^{-1}(i)})[1/t]$. We frequently view $w(\phi)$ as a (smooth) character of $T(K)$ (the torus subgroup of $\GL_d(K)$) for any $w\in S_d$. We also call these characters of $T(K)$ refinements of $M$.
	
	Let $\textbf{h}:=(\textbf{h}_i)_{i=1,\cdots, d}=(\textbf{h}_{\sigma})_{\sigma\in \Sigma_K}=(h_{\sigma,1}\geq \cdots \geq h_{\sigma,d})_{\sigma \in \Sigma_K}$ be the Hodge-Tate-Sen weights of $M$ (normalized such that the weight of the cyclotomic character is $1$). Let $w\in S_d$, we call the refinement $w(\phi)$ (or $\sT_w$) \textit{non-critical} if the Hodge-Tate-Sen weights of $\gr_{\sT_w}^i M$ are exactly $\textbf{h}_i$ (which are hence decreasing with growth of $i$). We call $M$  non-critical, if all the  refinements  of $M$ are non-critical. We denote by $\Phi\Gamma_{\nc}(\phi, \textbf{h})$ the set of isomorphism classes of non-critical crystabelline $(\varphi, \Gamma)$-modules of refinement $\phi$ and of Hodge-Tate-Sen weights $\textbf{h}$. Finally, we say $M$ has \textit{regular} Hodge-Tate-Sen weights if $\textbf{h}$ is strictly dominant, i.e. $h_{i,\sigma}>h_{i+1,\sigma}$ for all $\sigma\in \Sigma_K$.

	Suppose $M$ is generic crystabelline with refinement $\phi$. For  a subset $\textbf{r}=\{r_1, \cdots, r_k\}\subset \{1, \cdots, d\}$, denote by $M_{\textbf{r}}$ (resp. $M^{\textbf{r}}$) the saturated $(\varphi,\Gamma)$-submodule of $M$ (resp. the quotient of $M$) which has a refinement given by $(\phi_{r_1}, \cdots, \phi_{r_k})$. So $M^{\textbf{r}}=M / M_{\textbf{r}^c}$ with $\textbf{r}^c=\{1,\cdots, d\} \setminus \textbf{r}$. While $M_{\textbf{r}}$ and $M^{\textbf{r}}$ depend on the chosen refinement, this will not cause ambiguity: in all instances where they appear, the context will specify the refinement in use. Assuming $M$ is non-critical, $M_{\textbf{r}}$ and $M^{\textbf{r}}$ are non-critical as well for any $\textbf{r}$ (noting any triangulation of $M_{\textbf{r}}$ or of $M^{\textbf{r}}$ extends to a triangulation of $M$). In this case,  the Hodge-Tate-Sen weights of $M_{\textbf{r}}$ (resp. $M^{\textbf{r}}$) are $(\textbf{h}_1, \cdots, \textbf{h}_k)$ (resp. $(\textbf{h}_{d-k+1}, \cdots, \textbf{h}_d)$). 
	
	Throughout the paper, we will use $\bullet \lin \bullet$ to denote an extension of two objects (such as $(\varphi, \Gamma)$-modules, or $\GL_n(K)$-representations etc.), where the left object is the sub and the right  the quotient.
	\subsection{A reinterpretation of Hodge parameters}\label{S22}
	In this section, we give a reinterpretation of (some) $p$-adic Hodge parameters of a generic non-critical crystabelline $(\varphi,\Gamma)$-module.
	
	Let $\phi=(\phi_i)_{i=1,\cdots, n}$ be generic, and $\textbf{h}=(\textbf{h}_{\sigma})_{\sigma\in \Sigma_K}=(\textbf{h}_i)_{i=1, \cdots, n}=(h_{\sigma,1}>h_{\sigma, 2}> \cdots > h_{\sigma,n})$. Let $D\in \Phi\Gamma_{\nc}(\phi, \textbf{h})$. Let $D_1:=D_{\{1, \cdots, n-1\}}$ and $C_1:=D^{\{1, \cdots, n-1\}}$, we have two exact sequences:
	\begin{eqnarray}
		&& 0 \lra D_1 \lra D \lra \cR_{K,E}(\phi_n z^{\textbf{h}_n}) \lra 0, \label{ED1}\\
		&&	0 \lra \cR_{K,E}(\phi_n z^{\textbf{h}_1}) \lra D \lra C_1 \lra 0.\label{EC1}
	\end{eqnarray}
	Let  $\iota_D$ be the composition $D_1 \hookrightarrow D \twoheadrightarrow C_1$. As $\Hom(D_1, \cR_{K,E}(\phi_n z^{\textbf{h}_1}))=0$, $\iota_D$ is  injective. 
\begin{proposition}\label{Pdim2}
	(1) We have $\dim_E \Hom(D_1,C_1)\leq 2$.
	
	(2) We have $\dim_E \Hom(D_1,C_1)=2$ if and only if $n\geq 3$, and for any $i\in \{1, \cdots, n-1\}$, $\textbf{r}:=\{1,\cdots, n-1\} \setminus \{i\}$, we have $(D_1)^{\textbf{r}}\cong (C_1)_{\textbf{r}}$ (for the refinement $(\phi_1, \cdots, \phi_{n-1})$). Moreover, if these hold,  for $i\in \{1,\cdots, n-1\}$, the composition
	\begin{equation}\label{Ealphai}
		\alpha_i: D_1 \twoheadlongrightarrow (D_1)^{\textbf{r}} \cong (C_1)_{\textbf{r}} \hooklongrightarrow C_1.
	\end{equation}
	are pair-wisely linearly independent as elements in $\Hom(D_1, C_1)$.
\end{proposition}
\begin{proof}
(1) If $n=2$, it is clear that $\dim_E \Hom(D_1,C_1)=1$. Assume $n\geq 3$, and let $\textbf{r}=\{1,\cdots, n-3\},$ and consider $(C_1)_{\textbf{r}}$, which is the saturated submodule of $C_1$ of rank $n-3$ over $\cR_{K,E}$ with a refinement  $(\phi_1, \cdots, \phi_{n-3})$. As $C_1$ is non-critical of Hodge-Tate weights $(\textbf{h}_2, \cdots \textbf{h}_n)$, $(C_1)_{\textbf{r}}$ is non-critical of  Hodge-Tate weights $(\textbf{h}_2, \cdots, \textbf{h}_{n-2})$. Thus $(C_1)_{\textbf{r}}$ is isomorphic to a (non-split) successive extension of $\cR_{K,E}(\phi_i z^{\textbf{h}_{i+1}})$ for $i=1, \cdots, n-3$. Consider
	\begin{equation*}
		0 \lra \Hom(D_1, (C_1)_{\textbf{r}}) \lra \Hom(D_1, C_1) \lra \Hom(D_1, C_1/(C_1)_{\textbf{r}}).
	\end{equation*}
	Any map in  $\Hom(D_1, (C_1)_{\textbf{r}})$ clearly factors through $(D_1)^{\textbf{r}}$, the latter being isomorphic to a (non-split) successive extension of $\cR_{K,E}(\phi_i z^{\textbf{h}_{i+2}})$ for $i=1, \cdots, n-3$. By an easy d\'evissage, using $h_{\sigma, i+2}<h_{\sigma, i+1}$ and the fact 
	\begin{equation}\label{Ehomchara}\Hom(\cR_{K,E}(\phi_1' z^{k_1}), \cR_{K,E}(\phi'_2z^{k_2}))=0 \text{ if $\phi_1'\neq \phi_2'$  or $k_1<k_2$,}\end{equation}
	we deduce $\Hom((D_1)^{\textbf{r}}, (C_1)_{\textbf{r}})=0$ hence $\Hom(D_1, (C_1)_{\textbf{r}})=0$. Again by an easy d\'evissage, we have $\dim_E \Hom(D_1, C_1/(C_1)_{\textbf{r}})=\dim_E \Hom(D_1, (C_1)^{\{n-2,n-3\}})\leq 2$. Hence $\dim_E\Hom(D_1,C_1)\leq 2$.

	(2) We first prove ``if". As $(D_1)^{\textbf{r}}\cong (C_1)_{\textbf{r}}$, it is clear that $\alpha_i$  are well defined (as in  (\ref{Ealphai})) and pair-wisely linearly independent. Together with (1), we deduce $\dim_E \Hom(D_1, C_1)=2$. Conversely, assume $\Hom(D_1, C_1)=2$, and let $\iota_1, \iota_2$ be a basis of $\Hom(D_1, C_1)$. Let $i$, $\textbf{r}$ be as in (2). Consider the induced map $f_i: \cR_{K,E}(\phi_i z^{\textbf{h}_1}) \hookrightarrow D_1 \xrightarrow{\iota_i} C_1$. As $\dim_E \Hom(\cR_{K,E}(\phi_i z^{\textbf{h}_1}),C_1)=1$, there exists a  non-zero linear combination $\iota=a_1 \iota_1+a_2 \iota_2$ such that $a_1f_1+a_2f_2=0$. So (the non-zero) $\iota$ factors through a non-zero map $(D_1)^{\textbf{r}}\ra C_1$. As both $(D_1)^{\textbf{r}}$ and $C_1$ are non-critical, by comparing the weights and using (\ref{Ehomchara}), we deduce the map has  to factor through an isomorphism $(D_1)^{\textbf{r}} \xrightarrow{\sim} (C_1)_{\textbf{r}}$.  
\end{proof}
\begin{remark}
	By Proposition \ref{Pdim2} (2),  $ \Hom(D_1,C_1)$ is always two dimensional when $n=3$, or $n=4$ and $K=\Q_p$. In general, its dimension may be one or two  depending on the specific $D_1$ and $C_1$. 
\end{remark}
	Consider the cup-product
	\begin{equation}\label{Ecup1}
		\Ext^1(\cR_{K,E}(\phi_n z^{\textbf{h}_n}),D_1) \times \Hom(D_1, C_1) \lra \Ext^1(\cR_{K,E}(\phi_n z^{\textbf{h}_n}),C_1).
	\end{equation}
	\begin{proposition} \label{Phodge}
		Under the cup-product,   $E[D]\subset [\iota_D]^{\perp}$, and we have an equality if $K=\Q_p$. In particular, when $K=\Q_p$, $D$ is determined by $D_1$, $C_1$, $\phi_n$ and $\iota_D$.
	\end{proposition}
The last statement in Proposition \ref{Phodge} can be formulated precisely as follows: for a crystabelline $(\varphi, \Gamma)$-module $D'$ of rank $n$ over $\cR_{\Q_p,E}$, suppose $D_1$ is a saturated submodule of $D'$, $C_1$ is a quotient of $D'$ and $(\phi_1, \cdots, \phi_n)$ is a refinement of $D'$ (noting $(\phi_1, \cdots, \phi_{n-1})$ is already determined by $D_1$).  If the composition 
$\iota_{D'}: D_1 \hookrightarrow D' \twoheadrightarrow C_1$ is equal to $\iota_D$ up to a non-zero scalar, then $D'\cong D$.
	\begin{proof}
		As $\iota_D$ factors through $D$, the map induced by the pairing $\langle -, \iota_D\rangle$ (in (\ref{Ecup1})) is equal to the following  composition
		\begin{equation}\label{equpair}
			\Ext^1(\cR(\phi_n z^{\textbf{h}_n}), D_1) \lra 	\Ext^1(\cR(\phi_n z^{\textbf{h}_n}), D) \lra 	\Ext^1(\cR(\phi_n z^{\textbf{h}_n}), C_1).\end{equation}
		The first map sends  $[D]$ to zero, hence $\langle D, \iota_D\rangle=0$. In fact, by d\'evissage,  the kernel of the composition is isomorphic to $\Hom(\cR_{K,E}(\phi_n z^{\textbf{h}_n}), C_1/D_1)$, which, by Lemma \ref{Lemtorsion}, is furthermore isomorphic to $D_{\dR}^+(C_1')/D_{\dR}^+(D_1')$,
		where $C_1'=C_1 \otimes_{\cR_{K,E}} \cR_{K,E}(\phi_n^{-1} z^{-\textbf{h}_n})$ and $D_1'=D_1 \otimes_{\cR_{K,E}} \cR_{K,E}(\phi_n^{-1} z^{-\textbf{h}_n})$. As $C_1'$ (resp. $D_1'$) has Hodge-Tate-Sen weights $\{\textbf{h}_i-\textbf{h}_n\}_{i=2, \cdots, n}$ (resp. $\{\textbf{h}_i-\textbf{h}_n\}_{i=1,\cdots, n-1}$), we have $\dim_ED_{\dR}^+(C_1')/D_{\dR}^+(D_1')=d_K$. In particular, when $K=\Q_p$, the kernel of (\ref{equpair}) is exactly generated by $[D]$. This finishes the proof.
	\end{proof}
	
	In the rest of the section, we discuss what  information of  $D$ can be detected by $\iota_D$ for general $K$.  The reader who is mainly interested in the $\Q_p$-case can skip to the next section. Fix $\sigma\in \Sigma_K$, and define $\fT_{\sigma}(\textbf{h})$ to be the weight such that $\fT_{\sigma}(\textbf{h})_{\tau,i}=\begin{cases}
		h_{\tau,i} &\tau=\sigma \\
		h_{\tau,n} & \tau\neq \sigma
	\end{cases}$ which is in particular constant for $\tau\neq \sigma$.  The following proposition is a direct consequence of \cite[Thm.~A]{Ber08a}. We include a proof (of (1)) using similar arguments as in \cite[Lem.~2.1]{Ding15}.
	\begin{proposition} \label{Pcow}
	(1)	Let $D\in \Phi\Gamma_{\nc}(\phi, \textbf{h})$, and 	let $\sigma \in \Sigma_K$. There exists a unique $(\varphi, \Gamma)$-module (up to isomorphism) $D_{\sigma}$  over $\cR_{K,E}$ such that  $D_{\sigma}[1/t]\cong D[1/t]$,  $D\subset D_{\sigma}$,
		and  the Hodge-Tate weights of $D_{\sigma}$ are $\fT_{\sigma}(\textbf{h})$. 
		
	(2) Let $D$, $D_{\sigma}$ be as in (1).  The injection $D\hookrightarrow D_{\sigma}$ induces  a natural isomorphism of Deligne-Fontaine modules $D_{\pst}(D)\xrightarrow{\sim} D_{\pst}(D_{\sigma})$, such that the induced map $D_{\dR}(D) \ra D_{\dR}(D_{\sigma})$ is a morphism of filtered $K \otimes_{\Q_p} E$-modules,  satisfying $D_{\dR}(D)_{\sigma}\xrightarrow{\sim} D_{\dR}(D_{\sigma})$ (as filtered $E$-vector space). 
	\end{proposition}
\begin{proof}
Let $(W_e(D), W_{\dR}^+(D))$ be the $B$-pair associated to $D$ (cf. \cite[Thm.~A]{Ber08II}). By   Fontaine's classification of $B_{\dR}$-representations \cite[Thm.~3.19]{Fo04}, there is a unique $B_{\dR}^+\otimes_{\Q_p} E$-representation $\Lambda \cong \oplus_{\tau\in \Sigma_K} \Lambda_{\tau}$ such that $W_{\dR}^+(D)\subset \Lambda\subset W_{\dR}^+(D)[\frac{1}{t}]$ and $\Lambda_{\tau}\cong\begin{cases}
		W_{\dR}^+(D)_{\tau} & \tau=\sigma \\
		(t^{h_{\tau,n}} B_{\dR}^+ \otimes_{K,\tau} E)^{\oplus n} & \tau\neq \sigma
	\end{cases}$. Let $D_{\sigma}$ be the $(\varphi, \Gamma)$-module associated to the $B$-pair $(W_e(D), \Lambda)$. This construction satisfies all the claimed properties in (1). (2) follows from (1) and \cite[Thm.~A]{Ber08a}.  
\end{proof}
	\begin{lemma}\label{Lcow}
		Let $D$, $D_{\sigma}$ be as in Proposition \ref{Pcow}. For each $w\in S_n$, $w(\phi)z^{\fT_{\sigma}(\textbf{h})}$ is a trianguline parameter of  $D_{\sigma}$.
	\end{lemma}
	\begin{proof}
		Consider the composition $\cR_{K,E}(\phi_{w^{-1} (1)} z^{\textbf{h}_1}) \hookrightarrow D \hookrightarrow D_{\sigma}$. It is not difficult  to see the saturation of the image in $D_{\sigma}$ is just $\cR_{K,E}(\phi_{w^{-1} (1)}z^{\fT_{\sigma}(\textbf{h})_1})$,  
		and  we have $D/\cR_{K,E}(\phi_{w^{-1} (1)} z^{\textbf{h}_1}) \hookrightarrow D_{\sigma}/\cR_{K,E}(\phi_{w^{-1} (1)} z^{\fT_{\sigma}(\textbf{h})_1})$. Continuing with the argument, the lemma follows.
	\end{proof}
	We have hence a (surjective) map 
	\begin{equation}\label{Ecow}\fT_{\sigma}: \Phi\Gamma_{\nc}(\phi, \textbf{h}) \lra \Phi\Gamma_{\nc}(\phi, \fT_{\sigma}(\textbf{h})), \ D \mapsto D_{\sigma}.
	\end{equation}
	Let $D_{1,\sigma}:=(D_{\sigma})_{\{1,\cdots, n-1\}}$ and $C_{1,\sigma}:=(D_{\sigma})^{\{1,\cdots, n-1\}}$ (for the refinement $\phi$). By Lemma \ref{Lcow}, it is not difficult to see $D_{1,\sigma}$ (resp. $C_{1,\sigma}$) has Hodge-Tate-Sen weights $(\fT_{\sigma}(\textbf{h})_1, \cdots, \fT_{\sigma}(\textbf{h})_{n-1})$ \big(resp. $(\fT_{\sigma}(\textbf{h})_2, \cdots, \fT_{\sigma}(\textbf{h})_n)$\big). In fact, we have $D_{1,\sigma}=\fT_{\sigma}(D_1)$ and $C_{1,\sigma}=\fT_{\sigma}(C_1)$ (where $\fT_{\sigma}$ is defined in a similar way as (\ref{Ecow})). 
Consider $\Hom(D_{1,\sigma}, C_{1,\sigma})$. Note it is non-zero as it contains the composition $\iota_{D_{\sigma}}: D_{1,\sigma} \hookrightarrow D_{\sigma} \twoheadrightarrow C_{1,\sigma}$. By similar arguments as in Proposition \ref{Pdim2}, we have:
	\begin{proposition}
	(1) $\dim_E \Hom(D_{1,\sigma},C_{1,\sigma})\leq 2$.
	
	(2) We have $\dim_E \Hom(D_{1,\sigma},C_{1,\sigma})=2$ if and only if $n\geq 3$, and for any $i\in \{1,\cdots, n-1\}$, $\textbf{r}:=\{1,\cdots, n-1\} \setminus \{i\}$, we have $(D_{1,\sigma})^{\textbf{r}}\cong (C_{1,\sigma})_{\textbf{r}}$ (for the refinement $(\phi_1, \cdots, \phi_{n-1})$). Moreover, if these hold, for $i \in \{1,\cdots, n-1\}$, the  composition
	\begin{equation}\label{Ealphaisig}
		\alpha_{i,\sigma}: D_{1,\sigma} \twoheadlongrightarrow (D_{1,\sigma})^{\textbf{r}} \cong (C_{1,\sigma})_{\textbf{r}} \hooklongrightarrow C_{1,\sigma}.
	\end{equation}
	 are pair-wisely linearly independent as elements in $\Hom(D_{1,\sigma},C_{1,\sigma})$.
	\end{proposition}
	\begin{proposition}\label{Piotasigma0}
		For the cup-product $$\Ext^1(\cR_{K,E}(\phi_n z^{\textbf{h}_n}), D_{1,\sigma}) \times \Hom(D_{1,\sigma}, C_{1,\sigma}) \ra \Ext^1(\cR_{K,E}(\phi_n z^{\textbf{h}_n}), C_{1,\sigma}),$$ we have $[\iota_{D_{\sigma}}]^{\perp}=E[D_{\sigma}]$. In particular, $D_{\sigma}$ is determined by $D_{1,\sigma}$, $C_{1,\sigma}$, $\phi_n$ and $\iota_{D_{\sigma}}$ in a similar sense to that discussed following Proposition \ref{Phodge}.
	\end{proposition}
	\begin{proof} Taking the cup-product with $\iota_{D_{\sigma}}$  is equal to the following composition
		\begin{equation}\label{equpair2}
			\Ext^1(\cR(\phi_n z^{\textbf{h}_n}), D_{1,\sigma}) \ra \Ext^1(\cR(\phi_n z^{\textbf{h}_n}), D_{\sigma})\ra \Ext^1(\cR(\phi_n z^{\textbf{h}_n}), C_{1,\sigma}),
		\end{equation}
		which is the push-forward map via $\iota_{D_{\sigma}}$. We see $\lan D_{\sigma}, \iota_{D_{\sigma}}\ran =0$. On the other hand, by d\'evissage and Lemma \ref{Lemtorsion}, $\Ker(\ref{equpair2})$ is isomorphic to \begin{equation}\label{EdimDr}D_{\dR}^+\big(C_{1,\sigma} \otimes_{\cR_{K,E}} \cR_{K,E}(\phi_n^{-1} z^{-\textbf{h}_n})\big)/D_{\dR}^+\big(D_{1,\sigma} \otimes_{\cR_{K,E}} \cR_{K,E}(\phi_n^{-1} z^{-\textbf{h}_n})\big).\end{equation}
		By comparing the Hodge-Tate-Sen weights (and noting the weights of $D_{1,\sigma}$ and $C_{1,\sigma}$ for embeddings different from $\sigma$ are the same), we easily see that (\ref{EdimDr}) is one dimensional. Hence $\Ker(\ref{equpair2})$ is generated by $[D_{\sigma}]$.
	\end{proof}
	\begin{example}\label{Egl3}
		We give an example to illustrate how $\iota_{D_{\sigma}}$ determines $D_{\sigma}$ (or equivalently the Hodge $\sigma$-filtration of $D$).  Suppose $n=3$, $K$ unramified and  $D$ is crystalline (generic non-critical) of regular Hodge-Tate-Sen  weights $\textbf{h}$. In this case we have $D_{\cris}(D)\cong D_{\dR}(D)\cong \oplus_{\tau\in \Sigma_K} D_{\cris}(D)_{\tau}$, where each $D_{\cris}(D)_{\tau}$ is a filtered $\varphi^{d_K}$-module. Fix $\sigma\in \Sigma_K$. Note that we have an isomorphism of filtered $\varphi^{d_K}$-module $D_{\cris}(D_{\sigma})_{\sigma}\cong D_{\cris}(D)_{\sigma}$. 
		
		Let $\alpha_1$, $\alpha_2$, $\alpha_3$ be the three distinct eigenvalues of $\varphi^{d_K}$ on $D_{\cris}(D_{\sigma})_{\tau}$ (for any $\tau$). Let $e_{i,\sigma}$ be an $\alpha_i$-eigenvector in $D_{\cris}(D_{\sigma})_{\sigma}$, hence $D_{\cris}(D_{\sigma})_{\sigma}\cong E e_{1,\sigma} \oplus E e_{2,\sigma} \oplus E e_{3,\sigma}$. For $j=0, \cdots, d_K-1$, we have $D_{\cris}(D_{\sigma})_{\sigma \circ \Frob^{-j}}\cong E \varphi^j(e_{1,\sigma}) \oplus E \varphi^j (e_{2,\sigma}) \oplus E \varphi^j(e_{3,\sigma})$ (where $\Frob$ denotes the absolute Frobenius), and $D_{\cris}(D_{1,\sigma})_{\sigma \circ \Frob^{-j}} \cong  E \varphi^j(e_{1,\sigma}) \oplus E \varphi^j (e_{2,\sigma})$ for $j=0, \cdots, d_K-1$, which is equipped with the induced Hodge filtration.  As $D_{1,\sigma}$ is non-critical, multiplying $e_{1,\sigma}$, $e_{2,\sigma}$ by non-zero scalars, we can and do assume $\Fil^{\max} D_{\cris}(D_{1,\sigma})_{\sigma}=\Fil^j D_{\cris}(D_{1,\sigma})_{\sigma}$, $-h_{1,\sigma}<j\leq -h_{2,\sigma}$, is generated by $e_{1,\sigma}+e_{2,\sigma}$. As $D_{\sigma}$ is non-critical for all the refinements, multiplying $e_{3,\sigma}$ by a non-zero scalar, we can and do assume $\Fil^{\max} D_{\cris}(D_{\sigma})_{\sigma}=\Fil^j D_{\cris}(D_{\sigma})_{\sigma}$, $-h_{2,\sigma} < j \leq -h_{3,\sigma}$, is generated by $e_1+a_{D_{\sigma}} e_2+e_3$. The filtered $\varphi^{d_K}$-module $D_{\cris}(D_{\sigma})_{\sigma}$ is in fact parametrized (and determined) by $a_{D_{\sigma}}\in E \setminus \{0,1\}$: we have 
		\begin{equation*}
			\Fil^j D_{\cris}(D_{\sigma})_{\sigma}=\begin{cases}
				D_{\cris}(D_{\sigma})_{\sigma} & j\leq -h_{1,\sigma} \\
				E(e_{1,\sigma}+e_{2,\sigma}) \oplus E(e_{1,\sigma}+a_{D_{\sigma}} e_{2,\sigma}+e_{3,\sigma}) & -h_{1,\sigma}<j \leq -h_{2,\sigma} \\
				E(e_{1,\sigma}+a_{D_{\sigma}} e_{2,\sigma}+e_{3,\sigma}) & -h_{2,\sigma}<j \leq -h_{3,\sigma} \\
				0 & j>-h_{3,\sigma}
			\end{cases}
		\end{equation*}
		For $\tau \neq \sigma$, we have $\Fil^j D_{\cris}(D_{\sigma})_{\tau}=\begin{cases} D_{\cris}(D_{\sigma})_{\tau} & j\leq -h_{n,\tau} \\ 0 & j> -h_{n,\tau}\end{cases}$. So $D_{\sigma}$ is indeed determined by the single parameter $a_{D_{\sigma}}$ (in contrast, $D$ itself has many more parameters, when $K\neq \Q_p$). Note that for $-h_{2,\sigma}<j \leq -h_{3,\sigma}$,
		$\Fil^{\max} D_{\cris}(C_{1,\sigma})_{\sigma}=\Fil^j D_{\cris}(C_{1,\sigma})_{\sigma}$		is generated by $e_{1,\sigma}+a_{D_{\sigma}}e_{2,\sigma}$ (as it is equipped with the quotient filtration). The map $\iota_{D_{\sigma}}$ uniquely corresponds to the morphism of filtered $\varphi^{d_K}$-modules  $\iota_{D_{\sigma}}: D_{\cris}(D_{1,\sigma})_{\sigma} \ra D_{\cris}(C_{1,\sigma})_{\sigma}$  sending $e_{i,\sigma}$ to $e_{i,\sigma}$ for $i=1,2$. We see $a_{D_{\sigma}}$ can be read out from the relative position of the two lines $\Fil^{\max} D_{\cris}(C_{1,\sigma})_{\sigma}$ and $\iota_{D_{\sigma}}\big(\Fil^{\max} D_{\cris}(D_{1,\sigma})_{\sigma}\big)$  in $D_{\cris}(C_{1,\sigma})_{\sigma}$. Thus $a_{D_{\sigma}}$ (hence $D_{\sigma}$) is determined by $\iota_{D_{\sigma}}$.
\end{example}

\subsection{Deformations of crystabelline $(\varphi, \Gamma)$-modules}
Let $D\in \Phi\Gamma_{\nc}(\phi, \textbf{h})$. In this section, we collect some facts on certain deformations of $D$. 

\subsubsection{Trianguline and paraboline deformations, I}

We first consider  trianguline deformations. For a character $\chi: K^{\times} \ra E^{\times}$, recall we have natural isomorphisms
\begin{equation}\label{Edefchara}
\Hom(K^{\times},E) \xlongrightarrow{\sim}	\Ext^1_{K^{\times}}(\chi, \chi) \xlongrightarrow{\sim} \Ext^1(\cR_{K,E}(\chi), \cR_{K,E}(\chi)),
\end{equation}
sending $\psi$ to $\chi(1+\psi\epsilon)$ then to $\cR_{K, E[\epsilon]/\epsilon^2}(\chi(1+\psi \epsilon))$.

For $w\in S_n$, denote by $\Ext^1_w(D,D)\subset \Ext^1(D,D)$ the subspace of trianguline deformations with respect to the refinement $w(\phi)$. More precisely, for $\widetilde{D}\in \Ext^1(D,D)$ (viewed as a $(\varphi, \Gamma)$-module over $\cR_{K, E[\epsilon]/\epsilon^2})$), $\widetilde{D}\in \Ext^1_w(D,D)$ if and only if $\widetilde{D}$ is isomorphic to a successive extension of $\cR_{K, E[\epsilon]/\epsilon^2}(\phi_{w^{-1}(i)}z^{\textbf{h}_i}(1+\psi_i\epsilon))$ for $\psi_i \in \Hom(K^{\times},E)$.   In this case, we call the character $w(\phi) z^{\textbf{h}} (1+\psi\epsilon)$ (with $\psi:=(\psi_1, \cdots, \psi_n)$) of $T(K)$ over $E[\epsilon]/\epsilon^2$ the \textit{trianguline parameter} of $\widetilde{D}$ with respect to $w(\phi)$.  Let $\kappa_w$ be the following composition:
\begin{equation}\label{Ekappaw}
	\kappa_w:	\Ext^1_w(D,D) \lra \Ext^1_{T(K)}(w(\phi) z^{\textbf{h}}, w(\phi) z^{\textbf{h}})\xlongrightarrow{\sim}	\Hom(T(K),E) ,
\end{equation}
where the first map sends $\widetilde{D}$ to its trianguline parameter with respect to $w(\phi)$, and the second map is induced by (\ref{Edefchara}).
We also denote $\Ext^1_w(D,D)$ by $\Ext^1_{w(\phi)}(D,D)$ or $\Ext^1_{\sT_w}(D,D)$ where $\sT_w$ is the filtration on $D$ associated to $w(\phi)$ whenever it is convenient for the context.  The following proposition is well-known (cf.  \cite[\S~2]{BCh} \cite[\S~2]{Na2}).
\begin{proposition}\label{PDef1}(1) $\dim_E \Ext^1(D,D)=1+n^2 d_K$, $\dim_E \Ext^1_g(D,D)=1+\frac{n(n-1)}{2}d_K$ and $\dim_E \Ext^1_w(D,D)=1+\frac{n(n+1)}{2} d_K$ for all $w\in S_n$.
	
	(2) For $w\in S_n$, $\kappa_w$ is surjective. 
	
	(3) For $w\in S_n$, $\Ext^1_g(D,D)\subset \Ext^1_w(D,D)$ and is equal to the preimage of the subspace $\Hom_{\sm}(T(K),E)$ via $\kappa_w$. 
\end{proposition}
\begin{proof}
	The $K=\Q_p$-case is given in \cite[Prop.~2.3.10, Thm.~2.5.10]{BCh}. We sketch a proof for general $K$. As $D$ is non-critical, $\Hom(D,D)=E$. We also have $\Ext^2(D,D)=0$ since $D$ is generic. By \cite[Thm.~1.2(1)]{Liu07}, $\dim_E \Ext^1(D,D)=1+n^2 d_K$. By \cite[Cor.~2.53]{Na2} (noting any de Rham deformation of $D$ is automatically potentially crystalline), $\dim_E \Ext^1_g(D,D)=1+\frac{n(n-1)}{2}d_K$. By \cite[Prop.~2.41]{Na2} and the proof, $\dim_E \Ext^1_w(D,D)=1+\frac{n(n+1)}{2} d_K$ and $\kappa_w$ is surjective for all $w\in S_n$. Hence $\dim_E \Ker \kappa_w=\frac{n(n-1)}{2} d_K+1-n$. By \cite[Lem.~2.56]{Na2}, $\Ext^1_g(D,D)\subset \Ext^1_w(D,D)$ for all $w$. It is also clear $\kappa_w(\Ext^1_g(D,D))\subset \Hom_{\sm}(T(K),E)$. By comparing dimensions: $\dim_E \Ext^1_g(D,D)=\dim_E \Hom_{\sm}(T(K),E)+\dim_E \Ker \kappa_w,$
	 (3) follows. 
\end{proof}
Recall there is a right action of $S_n$ on $T(K)$: $w(a_1, \cdots, a_n)=(a_{w(1)}, \cdots, a_{w(n)})$ for $w\in S_n$. It induces a left action of $S_n$ on $\Hom(T(K),E)$: $(w\psi)(a_1,\cdots, a_n)=\psi(a_{w(1)}, \cdots, a_{w(n)})$. 
It is clear that $\Hom_{\sm}(T(K),E)$ is stabilized by the action. 
\begin{lemma}\label{Lint1}
	Let $w_1, w_2\in S_n$, the following diagram commutes
	\begin{equation}\label{Eint1}
		\begin{CD}
			\Ext^1_g(D,D) @> \kappa_{w_1} >>  \Hom_{\sm}(T(K),E) \\
			@| @V w_2w_1^{-1} V \sim V \\
			\Ext^1_g(D,D) @> \kappa_{w_2} >> \Hom_{\sm}(T(K),E).
	\end{CD}\end{equation}
\end{lemma}
\begin{proof}
	The lemma is well-known, but we include a proof for the convenience of the reader. It suffices to prove the statement for the case where  $w_2w_1^{-1}$ is a simple reflection, say, $s_k$. Let $\widetilde{D}\in \Ext^1_g(D,D)$ and suppose $\kappa_{w_i}(\widetilde{D})=(\psi_{i,1}, \cdots, \psi_{i,n})$. By definition, $\widetilde{D}$ admits  triangulations:
	\begin{equation*}
		\cR_{K, E[\epsilon]/\epsilon^2}(\phi_{w_i^{-1}(1)} z^{\textbf{h}_1} (1+\psi_{i,1}\epsilon)) \lin \cdots \lin 	\cR_{K, E[\epsilon]/\epsilon^2}(\phi_{w_i^{-1}(n)} z^{\textbf{h}_1} (1+\psi_{i,n}\epsilon)).
	\end{equation*}
	Note by assumption $w_1^{-1}(j)=w_2^{-1}(j)$ for $j\neq k, k+1$. Consequently,  for $j<k$ or $j>k+1$, we have  $\Fil^j_{\sT_{w_1}} \widetilde{D} \cong \Fil^j_{\sT_{w_2}} \widetilde{D}$, since $\Hom\big(\Fil^j_{\sT_{w_1}} \widetilde{D}, \widetilde{D}/\Fil^j_{\sT_{w_2}} \widetilde{D}\big)=0$.
	
	As $\Hom\Big(\cR_{K, E[\epsilon]/\epsilon^2}\big(\phi_{w_1^{-1}(1)} z^{\textbf{h}_1} (1+\psi_{1,1}\epsilon)\big), \widetilde{D}\Big)\cong E[\epsilon]/\epsilon^2$, 
	using d\'evissage for $\sT_{w_2}$, we easily deduce that if $k>1$, 
	\begin{equation*}
		\Hom\Big(\cR_{K, E[\epsilon]/\epsilon^2}\big(\phi_{w_1^{-1}(1)} z^{\textbf{h}_1} (1+\psi_{1,1}\epsilon)\big), \cR_{K, E[\epsilon]/\epsilon^2}\big(\phi_{w_2^{-1}(1)} z^{\textbf{h}_1} (1+\psi_{2,1}\epsilon)\big)\Big)\cong E[\epsilon]/\epsilon^2,
	\end{equation*} 
	hence  $H^0_{(\varphi, \Gamma)}\big(\cR_{K, E[\epsilon]/\epsilon^2}(1+(\psi_{1,1}-\psi_{2,1})\epsilon)\big)\cong E[\epsilon]/\epsilon^2$ (noting $w_1^{-1}(1)=w_2^{-1}(1)$). So $\psi_{1,1}=\psi_{2,1}$. We can then consider the $\cR_{K, E[\epsilon]/\epsilon^2}$-module $\widetilde{D}/\cR_{K, E[\epsilon]/\epsilon^2}(\phi_{w_1^{-1}(1)} z^{\textbf{h}_1} (1+\psi_{1,1}\epsilon))$ equipped with the filtrations induced by $\sT_{w_1}$ and $\sT_{w_2}$. Continuing with the above argument, we have $\psi_{1,j}=\psi_{2,j}$ for $j<k$. 
	
	For $j=k$, we have (noting $\Fil^{k-1}_{\sT_{w_1}} \widetilde{D}=\Fil^{k-1}_{\sT_{w_2}} \widetilde{D}$)
	\begin{equation*}
		\Hom\Big(\cR_{K, E[\epsilon]/\epsilon^2}\big(\phi_{w_1^{-1}(k)} z^{\textbf{h}_k}(1+\psi_{1,k\epsilon})\big), \widetilde{D}/\Fil^{k-1}_{\sT_{w_2}} \widetilde{D}\Big)\cong E[\epsilon]/\epsilon^2.
	\end{equation*}
	Using d\'evissage for $\sT_{w_2}$ (and the fact $w_2w_1^{-1}=s_k$), we get
	\begin{equation*}
		\Hom\big(\cR\big(\phi_{w_1^{-1}(k)} z^{\textbf{h}_k} (1+\psi_{1,k}\epsilon)\big), \cR\big(\phi_{w_2^{-1}(k+1)} z^{\textbf{h}_{k+1}} (1+\psi_{2,{k+1}}\epsilon)\big)\big)\cong E[\epsilon]/\epsilon^2,
	\end{equation*} 
	hence $\psi_{1,k}=\psi_{2,k+1}$. Exchanging $\sT_{w_1}$ and $\sT_{w_2}$, we get $\psi_{2,k}=\psi_{1,k+1}$.  
	
	For $j>k+1$, using the same argument as in the case of $j<k$ with $\widetilde{D}$ replaced by $\widetilde{D}/\Fil^{k+1}_{\sT_{w_1}} \widetilde{D}$, we see $\psi_{1,j}=\psi_{2,j}$. This concludes the proof. 
\end{proof}
Let $\Ext^1_0(D,D):=\Ker \kappa_w$ (for some $w\in S_n$ \textit{a priori}). By Proposition \ref{PDef1} (3), $\Ext^1_0(D,D) \subset \Ext^1_g(D,D)$. Using Lemma \ref{Lint1}, we see $\Ext^1_0(D,D)=\Ker \kappa_w$ for all $w\in S_n$. Moreover, by Proposition \ref{PDef1}  (1) (2), we have 
\begin{equation}\label{Ext0dim}
	\dim_E \Ext^1_0(D,D)=\frac{n(n-1)}{2}d_K+1-n.
\end{equation}
For $\Ext_*^1(D,D)\subset \Ext^1(D,D)$ (with $*=g, w, ...$), if $\Ext^1_*(D,D) \supset \Ext^1_0(D,D)$, we set $$\overline{\Ext}^1_*(D,D):=\Ext^1_*(D,D)/\Ext^1_0(D,D).$$ 
We have hence isomorphisms
\begin{equation}\label{Ekappaw000}
	\overline{\Ext}^1_w(D,D) \xlongrightarrow[\sim]{\kappa_w} \Hom(T(K),E), \ \ 	\overline{\Ext}^1_g(D,D) \xlongrightarrow[\sim]{\kappa_w} \Hom_{\sm}(T(K),E).
\end{equation}
Note also
\begin{equation}\label{bardim} \dim_E \ol{\Ext}^1(D,D)=\frac{n(n+1)}{2}d_K +n.
\end{equation}
Let $\Ext^1_{g'}(D,D)\subset \Ext^1(D,D)$ be the subspace of de Rham deformations up to twist by characters of $K^{\times}$ over $(E[\epsilon]/\epsilon^2)^{\times}$. Similarly, set
\begin{multline}\label{EHomg'T}
	\Hom_{g'}(T(K),E):=\{\psi\in \Hom(T(K),E)\ |\ \exists \psi_0: K^{\times} \ra E
\\ \text{ such that  } \psi-\psi_0 \circ \dett\in \Hom_{\sm}(T(K),E)\}.
	\end{multline}
One easily deduces from Proposition \ref{PDef1} (3) that for all $w\in S_n$, $\Ext^1_{g'}(D,D)\subset \Ext^1_w(D,D)$ and is equal to the preimage of $\Hom_{g'}(T(K),E)$ under $\kappa_w$. Thus
$\dim_E \Ext^1_{g'}(D,D)=1+(\frac{n(n-1)}{2}+1) d_K$.
Moreover, (\ref{Eint1}) holds with ``$g$" and ``$\sm$" replaced by ``$g'$". 
Using the fact that $D$ is non-critical, by \cite[Thm.~3.19]{Che11} (for $K=\Q_p$) and \cite[Thm.~2.62]{Na2} (for general $K$) (see also \cite{Kis10} for the ($n=2$)-case, noting the proposition also follows from Corollary \ref{Csur0} below and an easy induction argument), we have
\begin{proposition}\label{Pinffern}
	The natural map $\oplus_{w\in S_n} \Ext^1_w(D,D) \ra \Ext^1(D,D)$ is surjective and induces a surjective map
	$	\oplus_{w\in S_n} \ol{\Ext}^1_w(D,D) \twoheadrightarrow \ol{\Ext}^1(D,D)$.
\end{proposition}

Now we consider general paraboline deformations of $D$. Let $B$ the Borel subgroup of $\GL_n$ of upper triangular matrices, $P\supset B$ be a standard parabolic subgroup of $\GL_n$ with the standard Levi subgroup $L_P\supset T$ equal to $\diag(\GL_{n_1}, \cdots, \GL_{n_r})$. A filtration $$\sF_P: 0=\Fil^0_{\sF_P} D \subsetneq \Fil^1_{\sF_P} D \subsetneq \cdots \subsetneq \Fil^r_{\sF_P} D=D$$ of saturated $(\varphi, \Gamma)$-submodules of $D$ is called a \textit{$P$-filtration} if $M_i:=\rank \gr^i_{\sF_P} D=n_i$. A deformation $\widetilde{D}$ of $D$ over $E[\epsilon/\epsilon^2]$ is called an \textit{$\sF_P$-deformation}, if $\widetilde{D}$ admits a filtration $\Fil^i_{\sF_P} \widetilde{D}$ of saturated $(\varphi, \Gamma)$-submodules of $D$ over $\cR_{K, E[\epsilon]/\epsilon^2}$ (which means $\Fil^i_{\sF_P} \widetilde{D}$ is free over $\cR_{K, E[\epsilon]/\epsilon^2}$) such that $\gr^i_{\sF_P} \widetilde{D}$ is a deformation of $M_i$ over $\cR_{K,E[\epsilon]/\epsilon^2}$. Denote by $\Ext^1_{\sF_P}(D,D) \subset \Ext^1(D,D)$ the subspace of $\sF_P$-deformations. By \cite[Prop.~3.6, Prop.~3.7]{Che11} (which is for $K=\Q_p$, but all the arguments generalize directly to general $K$, see also the proof of Proposition \ref{Pparasigma} below), we have
\begin{proposition}\label{Ppara1}
	$\dim_E \Ext^1_{\sF_P}(D,D)=1+d_K \dim P=1+d_K \sum_{1\leq i \leq j \leq r} n_i n_j$. The natural map \begin{equation}\label{EkappaFp0}
		\kappa_{\sF_P}: \Ext^1_{\sF_P}(D,D) \lra \prod_{i=1}^r \Ext^1(M_i,M_i),
	\end{equation}sending $\widetilde{D}$ to $(\gr^i_{\sF_P} \widetilde{D})_{i=1,\cdots, r}$, is surjective. 
\end{proposition}
For  $w\in S_n$, we call the $B$-filtration $\mathscr{T}_w$ (associated to $w(\phi)$) compatible with $\sF_P$, if $\mathscr{T}_w$ induces a complete flag on $\Fil^i_{\sF_P} D$ for all $i$. In this case, we have $\Ext^1_{\sT_w}(D,D) \subset \Ext^1_{\sF_P}(D,D)$. For $i=1, \cdots, r$, we let $\mathscr{T}_{w,i}$ be the induced filtration on $M_i$ ($=\gr^i_{\sF_P} D$). 
\begin{corollary}\label{CkappaFp}
	Keep the above situation. 
	
	(1) $ \Ext^1_{w}(D,D) $ is the preimage of $ \prod_{i=1}^r \Ext^1_{\sT_{w,i}}(M_i,M_i)$ via $\kappa_{\sF_P}$. In particular, $\kappa_{\sF_P}$ induces a surjective map
	$	\kappa_{\sF_P}: \Ext^1_{w}(D,D) \twoheadrightarrow \prod_{i=1}^r \Ext^1_{\sT_{w,i}}(M_i,M_i)$.
	
	(2) The map $\kappa_{\sF_P}$ sends $\Ext^1_{0}(D,D)$ to $\prod_{i=1}^r \Ext^1_0(M_i,M_i)$ and induces  isomorphisms 
	\begin{equation}\label{EkappaFp}
		\kappa_{\sF_P}: \ol{\Ext}^1_{\sF_P}(D,D) \xlongrightarrow{\sim} \prod_{i=1}^r \ol{\Ext}^1(M_i,M_i),
	\end{equation}
and $\ol{\Ext}^1_{w}(D,D)\xrightarrow{\sim} \prod_{i=1}^r \ol{\Ext}^1_{\sT_{w,i}}(M_i,M_i)$.
\end{corollary}
\begin{proof}
	The first part of (1) is by definition, and the second part follows from Proposition \ref{Ppara1}. It is clear that the following diagram commutes
	\begin{equation}\label{Eparatri}
		\begin{tikzcd}
		&\Ext^1_w(D,D) \arrow[r, two heads] \arrow[d, "\kappa_w", two heads]& \prod_{i=1}^r \Ext^1_{\sT_{w,i}}(M_i,M_i)  \arrow[d,	"(\kappa_{\sT_{w,i}})", two heads] \\
			&\Hom(T(K),E) \arrow[r, "\sim"] &\prod_{i=1}^r \Hom(T_i(K),E)
		\end{tikzcd}
	\end{equation}
	where $T_i$ is the torus subgroup of $\GL_{n_i}$. The first part of (2) follows. By (1) and Proposition \ref{Pinffern}, (\ref{EkappaFp}) is surjective.  However, by Proposition \ref{Ppara1}, (\ref{Ext0dim}) and (\ref{bardim}) (applied to the $M_i$'s), we have  $\dim_E\overline{\Ext}^1_{\sF_P}(D,D)=d_K \dim (B \cap L_P)-n=\sum_{i=1}^r \dim_E \ol{\Ext}^1(M_i,M_i) $. Hence (\ref{EkappaFp}) is bijective. The final isomorphism follows by similar arguments.
\end{proof}
Let  $\Ext^1_{\sF_P,g'}(D,D)$ be the preimage of $\prod_{i=1}^r \Ext^1_{g'}(M_i,M_i)$ via (\ref{EkappaFp0}). Set 
\begin{multline}\label{EPg'T}\Hom_{P,g'}(T(K),E):=\{\psi \in \Hom(T(K),E)\ |\ \exists \psi_P: Z_{L_P}(K)\ra E \\
	\text{ such that }  \psi-\psi_P \circ \dett_{L_P} \in \Hom_{\sm}(T(K),E)\}.\end{multline}  It is straightforward to see $\dim_E \Hom_{P,g'}(T(K),E)=n+rd_K$. 
The following corollary  generalizes (\ref{Eint1}). 
\begin{corollary}\label{Cint}
	(1)	Let $w\in S_n$ such that $\sT_w$ is compatible with $\sF_P$, then $\Ext^1_{\sF_P,g'}(D,D)\subset \Ext^1_w(D,D)$.
	
	(2) Let $w_1,w_2\in S_n$ such that $\sT_{w_1}$, $\sT_{w_2}$  are compatible with $\sF_P$ (so $w_2w_1^{-1}$ lies in the Weyl group  $\sW_P$ of $L_P$), we have a commutative diagram
	\begin{equation*}
		\begin{CD}
			\ol{\Ext}^1_{\sF_P, g'}(D,D) @>  \kappa_{w_1} >  \sim >  \Hom_{P,g'}(T(K),E) \\
			@| @V w_2w_1^{-1} V \sim V \\ 
			\ol{\Ext}^1_{\sF_P,g'}(D,D) @> \kappa_{w_2}> \sim > \Hom_{P,g'}(T(K),E).
		\end{CD}
	\end{equation*}.
\end{corollary}
\begin{proof}
	(1) follows from the fact $\Ext^1_{g'}(M_i,M_i)\subset \Ext^1_{\sT_{w,i}}(M_i,M_i)$ and Corollary \ref{CkappaFp} (1) . By Corollary \ref{CkappaFp} (2), we have $\ol{\Ext}^1_{\sF_P, g'}(D,D)\xrightarrow{\sim} \prod_{i=1}^r \ol{\Ext}^1_{g'}(M_i,M_i)$. (2) then follows from the commutative diagram (\ref{Eparatri}) and Lemma \ref{Lint1} (applied to each $M_i$, and with ``$g$", ``$\sm$" replaced by ``$g'$").
\end{proof}
\subsubsection{Trianguline and paraboline deformations, II}

Let $D\in \Phi\Gamma_{\nc}(\phi, \textbf{h})$.
We consider some partially de Rham deformations of $D$. The reader who is mainly interested in the $\Q_p$-case can skip this section. Recall for $J\subset \Sigma_K$, and a $(\varphi, \Gamma)$-module $M$  over $\cR_{K,E}$, $M$ is called \textit{$J$-de Rham}, if $\dim_E D_{\dR}(M)_{\tau}=\rank_{\cR_{K,E}} M$ for all $\tau\in J$, where $D_{\dR}(M)_{\tau}=H^0(\Gal_K, W_{\dR}^+(M)_{\tau}[1/t])$. Note the property is clearly inherited by taking subquotients. For a $(\varphi, \Gamma)$-module $M$ over $\cR_{K,E}$, denote by $W(M)=(W_e(M), W_{\dR}^+(M))$ its associated $B$-pair (\cite{Ber08II}). By \cite[Thm.~5.11]{Na2}, there are natural isomorphisms for $i=0, 1, 2$,
\begin{equation}
	\label{EcohophiGamma}	H^i_{(\varphi, \Gamma)}(M) \xlongrightarrow{\sim} H^i(\Gal_K, W(M))
\end{equation}
where $H^i(\Gal_K,W(M))$ denotes the $i$-th Galois cohomology of the $B$-pair $M$, see \cite[\S~2.1]{Na}.

 Throughout the section, we fix $\sigma\in \Sigma_K$. For an extension group $\Ext^1_?(D,D)$, we denote by $\Ext^1_{\sigma, ?}(D,D)\subset \Ext^1_?(D,D)$ the subspace consisting of  $\widetilde{D}$ that are $\Sigma_K \setminus \{\sigma\}$-de Rham.
If $\Ext^1_?(D,D)\supset \Ext^1_{0}(D,D)$, then it is clear that $\Ext^1_{\sigma,?}(D,D)\supset \Ext^1_{0}(D,D)$ and we set $$\ol{\Ext}^1_{\sigma,?}(D,D):=\Ext^1_{\sigma, ?}(D,D)/\Ext^1_0(D,D) \subset \ol{\Ext}^1_?(D,D). $$
\begin{lemma}\label{Lsigma1}
	We have $\dim_E \Ext^1_{\sigma}(D,D)=1+\frac{n(n-1)}{2} (d_K-1)+n^2$.
\end{lemma}
\begin{proof}
	Using the notation of \cite[\S~A]{Ding6}, the isomorphism (\ref{EcohophiGamma}) (for $i=1$, $M=D\otimes_{\cR_K,E} D^{\vee}$) induces an isomorphism $\Ext^1_{\sigma}(D,D)\cong H^1_{g,\Sigma_K\setminus \{\sigma\}}(\Gal_K, W(D \otimes_{\cR_{K,E}} D^{\vee}))$ where $D^{\vee}:=\Hom_{\cR_{K,E}}(D,\cR_{K,E})$. The lemma follows then from \cite[Cor.~A.4]{Ding6} (noting $D \otimes_{\cR_{K,E}} D^{\vee}$ has Hodge-Tate-Sen weights $\{h_{\tau,i}-h_{\tau,j}\}_{\substack{\tau\in \Sigma_K \\ i,j=1,\cdots, r}})$. The required assumption holds because $D$ is generic.	
\end{proof}
Let $P$ be a standard parabolic subgroup, and $\sF_P$ be a $P$-filtration on $D$ with $\gr^i_{\sF_P} D=:M_i$. The  surjection $\kappa_{\sF_P}$ (\ref{EkappaFp0}) induces a map
\begin{equation}\label{EkappaFp2}
\kappa_{\sF_P}: \Ext^1_{\sigma, \sF_P}(D,D) \lra \prod_{i=1}^r \Ext^1_{\sigma}(M_i,M_i).
\end{equation}
\begin{proposition}\label{Pparasigma}
(1) We have $\dim_E \Ext^1_{\sigma, \sF_P}(D,D)=1+(d_K-1)\frac{n(n-1)}{2}+\dim P$.

(2)	The map (\ref{EkappaFp2})  is surjective and induces an isomorphism
\begin{equation}\label{EkappaFP11}
	\ol{\Ext}_{\sigma, \sF_P}(D,D) \xlongrightarrow{\sim} \prod_{i=1}^r \ol{\Ext}^1_{\sigma}(M_i,M_i).
\end{equation}
\end{proposition}
\begin{proof}Let $\Hom_{\sF_P}(D,D)$ be the $(\varphi, \Gamma)$-submodule of $\Hom_{\cR_{K,E}}(D,D)\cong D\otimes_{\cR_{K,E}} D^{\vee}$ consisting of the maps $f$ such that $f(\Fil_{\sF_P}^i) \subset \Fil_{\sF_P}^i$ for $i=1, \cdots, r$. Similarly as in \cite[Prop.~3.6 (ii)]{Che11} and using the notation of \cite[\S~A]{Ding6}, we have $\Ext^1_{\sigma, \sF_P}(D,D)\cong H^1_{g,\Sigma_K\setminus \{\sigma\}}(\Gal_K, W(\Hom_{\sF_P}(D,D)))$. Since $D$ is non-critical, it is straightforward  to see $\Hom_{\sF_P}(D,D)$ has Hodge-Tate-Sen weights $\{h_{\tau,i}-h_{\tau,j}\}_{\tau \in \Sigma_K}$ where the indices $(i,j)$ correspond to entries of the matrix $\gl_n$ lying in $\fp$, the Lie algebra of $P$. 
	By \cite[Cor.~A.4]{Ding6} (noting that the $(\varphi, \Gamma)$-module $\Hom_{\sF_P}(D,D)$ satisfies the assumptions in \textit{loc. cit.} as $D$ is generic), we calculate
$\dim_E H^1_{g,\Sigma_K\setminus \{\sigma\}}(\Gal_K, W_{\dR}^+(\Hom_{\sF_P}(D,D)))
=1+d_K \dim P-\sum_{\tau\neq \sigma} \dim (B \cap L_P)=1+(d_K-1) \frac{n(n-1)}{2}+\dim P$.
(1) follows.
For any $\widetilde{D}\in \Ker(\ref{EkappaFp0})$, using Corollary \ref{CkappaFp} (1), (\ref{Eparatri}) and Proposition \ref{PDef1} (3), we see $\widetilde{D}$ is de Rham. Hence $\Ker(\ref{EkappaFp0})\subset \Ext^1_{\sigma, \sF_P}(D,D)$. Let $N$ be the unipotent radical of $B$. As
$\dim_E \Ext^1_{\sigma, \sF_P}(D,D)-\dim_E \Ker (\ref{EkappaFp0})=r+(d_K-1) \dim (N\cap L_P)+\dim L_P
	=\sum_{i=1}^r \dim_E \Ext^1_{\sigma}(M_i,M_i)$, 
(\ref{EkappaFp2}) hence (\ref{EkappaFP11}) are  surjective. Finally we have  equalities
$	\dim_E 	\ol{\Ext}_{\sigma, \sF_P}(D,D)=n+\dim (B \cap L_P)=\sum_{i=1}^r \dim \ol{\Ext}^1_{\sigma}(M_i,M_i)$
which complete the proof of (2).
\end{proof}
Combining Proposition \ref{Pparasigma} (2) with Corollary \ref{CkappaFp} (2), we get:
\begin{corollary}\label{CFpwsigma}
Let $\sT_w$ be a $B$-filtration compatible with $\sF_P$ (see Corollary \ref{CkappaFp}). The map $\kappa_{\sF_P}$ induces a bijection $\ol{\Ext}^1_{\sigma,w}(D,D) \xrightarrow{\sim} \prod_{i=1}^r \ol{\Ext}^1_{\sigma, \sT_{w,i}}(M_i,M_i)$.
\end{corollary}
For a rank one de Rham $(\varphi, \Gamma)$-module $\cR_{K,E}(\chi)$ (implying $\chi$ is locally algebraic), by \cite[Lem.~1.15]{Ding6}, (\ref{Edefchara}) induces  by restriction an isomorphism 
\begin{equation}\label{Echaracsigma}\Ext^1_{\sigma}(\cR_{K,E}(\chi), \cR_{K,E}(\chi))\cong \Hom_{\sigma}(K^{\times},E).\end{equation} By Proposition \ref{Pparasigma} (2) applied to $P=B$, we obtain:
\begin{corollary}\label{Cwtsigma}
For $w\in S_n$, $\kappa_w$ (\ref{Ekappaw}) induces an isomorphism
$	\ol{\Ext}_{\sigma,w}(D,D) \xrightarrow{\sim} \Hom_{\sigma}(T(K),E)$.
\end{corollary}
We will show later (in Corollary \ref{Cinf2} below) the induced map 
\begin{equation}\label{Einf2}\oplus_{w\in S_n} \Ext^1_{\sigma, w}(D,D) \lra \Ext^1_{\sigma}(D,D)
\end{equation} is surjective (and the same holds with $\Ext^1$ replaced by $\ol{\Ext}^1$). Consider now certain extension groups of  $D_{\sigma}:=\fT_{\sigma}(D)$ (cf. (\ref{Ecow})). 
\begin{proposition}\label{Ppara2}
(1) We have $\dim_E\Ext^1(D_{\sigma},D_{\sigma})=1+n^2 d_K$.

(2) We have $\dim_E \Ext^1_g(D_{\sigma}, D_{\sigma})=1+\frac{n(n-1)}{2}$. 

(3)   Let $P$ be a standard parabolic subgroup of $\GL_n$, and $\sF_P$ be a $P$-filtration of $D_{\sigma}$ with $\gr_i \sF_P\cong M_{i,\sigma}$. We have $\dim_E \Ext^1_{\sF_P}(D_{\sigma}, D_{\sigma})=1+d_K \dim P$ and $\Ext^1_g(D_{\sigma}, D_{\sigma}) \subset \Ext^1_{\sF_P}(D_{\sigma},D_{\sigma})$. Moreover, the following natural map  (defined similarly as in (\ref{EkappaFp0})) is surjective
\begin{equation}\label{EkapFpsigma}
	\Ext^1_{\sF_P}(D_{\sigma}, D_{\sigma}) \twoheadlongrightarrow \prod_{i=1}^r \Ext^1(M_{i,\sigma}, M_{i,\sigma}).
\end{equation}
\end{proposition} 
\begin{proof}
(1) follows from \cite[Thm.~1.2 (1)]{Liu07} as $\Hom(D_{\sigma},D_{\sigma})\cong E$, $\Ext^2(D_{\sigma}, D_{\sigma})=0$. (2) follows from (1), \cite[Cor.~A.4]{Ding6} and 
$\dim H^0(\Gal_K, W_{\dR}^+(D_{\sigma} \otimes_{\cR_{K,E}} D_{\sigma}^{\vee})_{\tau})=\begin{cases} n^2 & \tau \neq \sigma \\ \frac{n(n+1)}{2} & \tau=\sigma.\end{cases}$. The statements in (3) except $\Ext^1_g(D_{\sigma}, D_{\sigma})\subset \Ext^1_{\sF_P}(D_{\sigma}, D_{\sigma})$ follow by the same argument as in the proof of \cite[Prop.~3.6, Prop.~3.7]{Che11}.  By  \cite[Cor.~A.4]{Ding6}, $\Ext^1_g\big(\Fil_{\sF_P}^i D_{\sigma}, D_{\sigma}/\Fil_{\sF_P}^i D_{\sigma}\big)=0$ for $i=1, \cdots, r-1$. Hence if  $\widetilde{D}_{\sigma} \in \Ext^1_g(D_{\sigma}, D_{\sigma})$, it must map to zero under the natural map 
\begin{equation*}
	\Ext^1(D_{\sigma}, D_{\sigma}) \lra \Ext^1(\Fil^1_{\sF_P} D_{\sigma}, D_{\sigma}/\Fil^1_{\sF_P} D_{\sigma}).
\end{equation*}
Thus $\widetilde{D}_{\sigma}$ has the form $[\widetilde{M}_1 \lin \widetilde{M}_2]$ where $\widetilde{M}_1$ (resp. $\widetilde{M}_2$) is a deformation of $\Fil^1_{\sF_P} D_{\sigma}$ (resp. of $D_{\sigma}/\Fil^1_{\sF_P} D_{\sigma}$). Iterating the argument for $\widetilde{M}_2$, we inductively deduce $\widetilde{D}_{\sigma}\in \Ext^1_{\sF_P}(D_{\sigma}, D_{\sigma})$.
\end{proof}
\begin{remark}\label{RdeDs}
Recall for each $w\in S_n$, $w(\phi)$ is also a refinement of $D_{\sigma}$ and we still use $\sT_w$ to denote the associated $B$-filtration on $D_{\sigma}$. Applying Proposition \ref{Ppara2} (3) for $P=B$ and $\sF_P=\sT_w$, we have $\dim_E \Ext^1_w(D_{\sigma}, D_{\sigma})=1+d_K \frac{n(n+1)}{2}$ and a natural surjection
\begin{equation}\label{Ekappaw2}
	\kappa_w: \Ext^1_w(D_{\sigma}, D_{\sigma}) \twoheadlongrightarrow \Hom(T(K),E).
\end{equation}
The preimage of $\Hom_{\sm}(T(K),E)$ hence has dimension equal to $(1+d_K \frac{n(n+1)}{2})-nd_K=1+d_K \frac{n(n-1)}{2}$. Together with Proposition \ref{Ppara2} (2), we see  when $K\neq \Q_p$, $\Ext^1_g(D_{\sigma}, D_{\sigma})$ is properly contained in the preimage of $\Hom_{\sm}(T(K),E)$.
\end{remark}
For $\Sigma_K \setminus\{\sigma\}$-de Rham deformations of $D_{\sigma}$, we have:
\begin{proposition}\label{Psigma3}(1) We have $\dim_E \Ext^1_{\sigma}(D_{\sigma}, D_{\sigma})=1+n^2$.

(2) Let $P$ be a standard parabolic subgroup of $\GL_n$, and $\sF_P$ be a $P$-filtration of $D_{\sigma}$ with $\gr^i_{\sF_P} D_{\sigma} \cong M_{i,\sigma}$. Then $\dim_E \Ext^1_{\sigma, \sF_P}(D_{\sigma}, D_{\sigma}) =1+\dim P$.
\end{proposition}
\begin{proof}By \cite[Cor.~A.4]{Ding6}, (1) (resp. (2)) follows from Proposition \ref{Ppara2} (1) (resp. (3)) and the fact that 
for $\tau\neq \sigma$, $\dim_E H^0(\Gal_K, W_{\dR}^+(\Hom_{\cR_{K,E}}(D_{\sigma},D_{\sigma}))_{\tau})=n^2$ (resp. 
$\dim_E H^0(\Gal_K, W_{\dR}^+(\Hom_{\sF_P}(D_{\sigma}, D_{\sigma}))_{\tau})=\dim P$).
Here $\Hom_{\sF_P}(D_{\sigma}, D_{\sigma})$ is defined in a similar way as in the proof of Proposition \ref{Pparasigma}.
\end{proof}
Now we consider the relation between deformations of $D$ and those of $D_{\sigma}$. 	The following proposition follows from the same argument as in the proof of Proposition \ref{Pcow}, accounting for the $E[\epsilon]/\epsilon^2$-structure. We leave the details to the reader. 
\begin{proposition}For any $(\varphi, \Gamma)$-module $\widetilde{D}\in \Ext^1(D,D)$ over $\cR_{K, E[\epsilon]/\epsilon^2}$, there is a unique $(\varphi, \Gamma)$-module $\widetilde{D}_{\sigma}\in \Ext^1(D_{\sigma}, D_{\sigma})$ over $\cR_{K, E[\epsilon]/\epsilon^2}$ satisfying that $\widetilde{D}\subset \widetilde{D}_{\sigma}$, $\widetilde{D}[1/t]\cong \widetilde{D}_{\sigma}[1/t]$, and the Sen $\sigma$-weights of $\widetilde{D}_{\sigma}$ are equal to  those of $\widetilde{D}$, and the Sen $\tau$-weights (over $E$) of $\widetilde{D}_{\sigma}$ are constantly $h_{\tau,n}$ for $\tau \neq\sigma$.  
\end{proposition}

We obtain hence a natural map\begin{equation}\label{Ecow2}
	\fT_{\sigma}: \Ext^1(D,D) \lra \Ext^1(D_{\sigma}, D_{\sigma}), \ \widetilde{D} \mapsto \widetilde{D}_{\sigma}.
\end{equation} 
It is clear that this operation preserves (partial) de Rhamness and filtrations of saturated submodules. In particular, $\fT_{\sigma}$ restricts to a map $\Ext^1_{\sigma}(D,D) \ra \Ext^1_{\sigma}(D_{\sigma}, D_{\sigma})$, and to a map 
$ \Ext^1_{\sF_P}(D,D) \ra \Ext^1_{\sF_P}(D_{\sigma}, D_{\sigma})$, where $\sF_P$ on $D_{\sigma}$ is defined by $\Fil^i_{\sF_P} D_{\sigma}=\fT_{\sigma}(\Fil^i_{\sF_P} D)$.
\begin{proposition}\label{Pcow2}(1) For $*\in \{g, \sigma,  \{\sigma, \sF_P\}\}$, the induced map $\fT_{\sigma}: \Ext^1_*(D,D) \ra \Ext^1_*(D_{\sigma},D_{\sigma})$ is surjective, and has the same kernel as (\ref{Ecow2}).
	
	(2) The following diagram commutes
\begin{equation}\label{Eparacow}
	\begin{tikzcd} 	&	\Ext^1_{\sigma,\sF_P}(D,D) \arrow[r, "(\ref{EkappaFP11})", two heads] \arrow[d,  "\fT_{\sigma}", two heads] &\prod_{i=1}^r \Ext^1_{\sigma}(M_i,M_i) \arrow[d, "\fT_{\sigma}", two heads] \\
	&	\Ext^1_{\sigma, \sF_P}(D_{\sigma}, D_{\sigma}) \arrow[r, "(\ref{EkapFpsigma})"] & \prod_{i=1}^r \Ext^1_{\sigma}(M_{i,\sigma}, M_{i,\sigma}).
	\end{tikzcd}
\end{equation}
Moreover, the map $\Ext^1_{\sigma, \sF_P}(D_{\sigma}, D_{\sigma}) \ra \prod_{i=1}^r \Ext^1_{\sigma}(M_{i,\sigma}, M_{i,\sigma})$ is surjective.
\end{proposition}
\begin{proof} First, any $\widetilde{D}\in \Ker(\ref{Ecow2})$ is de Rham, as it is contained in the de Rham $(\varphi, \Gamma)$-module $D_{\sigma} \oplus D_{\sigma}$. Hence $\Ker(\ref{Ecow2})$ coincides with the kernel of any maps in (1) (see also  Proposition \ref{Ppara2} (3)). Consider the composition  (where the second map is the natural pull-back)
\begin{equation}\label{Ecow3}
	\Ext^1(D,D) \xlongrightarrow{\fT_{\sigma}} \Ext^1(D_{\sigma}, D_{\sigma}) \lra \Ext^1(D, D_{\sigma}).
\end{equation}
As $\Hom(D,D)\cong \Hom(D,D_{\sigma}) \cong \Hom(D_{\sigma}, D_{\sigma}) \cong E$,  the kernel of the second map in (\ref{Ecow3}) is isomorphic to $H^0_{(\varphi,\Gamma)}\big((D_{\sigma}\otimes_{\cR_{K,E}} D^{\vee})/(D_{\sigma} \otimes_{\cR_{K,E}} D_{\sigma}^{\vee})\big)\xrightarrow[\sim]{\text{Lem. \ref{Lemtorsion}}} D_{\dR}^+(D_{\sigma}\otimes_{\cR_{K,E}} D^{\vee})/D_{\dR}^+(D_{\sigma} \otimes_{\cR_{K,E}} D_{\sigma}^{\vee})=0$ where the vanishing follows easily by comparing the weights. Thus $\Ker(\ref{Ecow2})\xrightarrow{\sim} \Ker(\ref{Ecow3})$. The composition (\ref{Ecow3}) coincides with  the natural push-forward map via $D\hookrightarrow D_{\sigma}$. We deduce by d\'evissage that $\Ker(\ref{Ecow3})$ is isomorphic to $H^0_{(\varphi, \Gamma)}\big((D_{\sigma}\otimes_{\cR_{K,E}}D^{\vee})/(D \otimes_{\cR_{K,E}} D^{\vee})\big)$. Using Lemma \ref{Lemtorsion} and the easy fact $\dim_E D_{\dR}^+(D_{\sigma} \otimes_{\cR_{K,E}} D^{\vee})=\frac{n(n+1)}{2}+(d_K-1) n^2$ and $\dim_E D_{\dR}^+(D \otimes_{\cR_{K,E}} D^{\vee})=\frac{n(n+1)}{2}d_K$, we deduce $\dim_E \Ker(\ref{Ecow2})=\dim_E \Ker(\ref{Ecow3})=(d_K-1)\frac{n(n-1)}{2}$. By the dimension results in Proposition \ref{PDef1} (1) (resp. Proposition \ref{Ppara1}, resp. Proposition \ref{Pparasigma} (1)) and Proposition \ref{Ppara2} (2) (resp. Proposition \ref{Psigma3} (1), resp. Proposition \ref{Psigma3} (2)), the difference in dimensions between the source and target spaces in (1)  is exactly $(d_K-1)\frac{n(n-1)}{2}$ for $*=g$ (resp. $*=\sigma$, resp. $*=\{\sigma, \sF_P\}$).  This proves (1). The commutativity of (\ref{Eparacow}) follows directly from the definition of  $\fT_{\sigma}$. The second part of (2)  is then  a consequence of (1) applied to each $M_i$ (with $*=\sigma$)  and of the surjectivity of (\ref{EkappaFp2}) (see the first part of Proposition \ref{Pparasigma} (2)). 
\end{proof}

\begin{corollary}\label{Ccowg}
Let $w_1, w_2\in S_n$, the following diagram commutes
\begin{equation*}
	\begin{CD}
		\Ext^1_g(D_{\sigma}, D_{\sigma}) @> \kappa_{w_1} >> \Hom_{\sm}(T(K),E) \\
		@| @V w_2 w_1^{-1} VV \\
		\Ext^1_g(D_{\sigma}, D_{\sigma}) @> \kappa_{w_2} >> \Hom_{\sm}(T(K),E),
	\end{CD}
\end{equation*}
and the horizontal maps are surjective.
\end{corollary}
\begin{proof}
The commutativity follows by the same argument as in Lemma \ref{Lint1}. For $w\in S_n$, we have a commutative diagram (where the right square   corresponds to (\ref{Eparacow}) for $P=B$ and $\sF_P=\sT_w$)
\begin{equation}\label{Ecowg}
	\begin{tikzcd}&\Ext^1_g(D, D)\arrow[r, hook] \arrow[d, two heads, "\fT_{\sigma}"] &\Ext^1_{\sigma,w}(D,D) \arrow[r,two heads] \arrow[d,  two heads, "\fT_{\sigma}"]  &\Hom_{\sigma}(T(K),E) \arrow[d, equal] \\
		&\Ext^1_g(D_{\sigma}, D_{\sigma}) \arrow[r, hook] &\Ext^1_{\sigma,w}(D_{\sigma}, D_{\sigma}) \arrow[r] &\Hom_{\sigma}(T(K),E).
	\end{tikzcd}
\end{equation}
The surjectivity of $\kappa_{w_i}$ in the corollary follows from Proposition \ref{PDef1} (3). 
\end{proof}
Let $\Ext^1_0(D_{\sigma}, D_{\sigma})\subset \Ext^1_g(D_{\sigma}, D_{\sigma})$ be the kernel of $\kappa_w: \Ext^1_g(D_{\sigma}, D_{\sigma})\ra \Hom_{\sm}(T(K),E)$ (for one or equivalently any $w\in S_n$, by Corollary \ref{Ccowg}).  Note that unlike the case for $D$, this subspace is strictly contained in the kernel of (\ref{Ekappaw2}) when $K\neq \Q_p$ (see the last sentence in Remark \ref{RdeDs}).  
\begin{corollary}\label{Cext10}
We have $\Ext^1_0(D,D)=\fT_{\sigma}^{-1}(\Ext^1_0(D_{\sigma}, D_{\sigma}))$, and $\fT_{\sigma}$ restricts to a surjection $\Ext^1_0(D,D) \twoheadrightarrow \Ext^1_0(D_{\sigma}, D_{\sigma})$.
\end{corollary}
\begin{proof}
By Proposition \ref{Pcow2} (1) (and the proof), $\Ext^1_g(D,D)=\fT_{\sigma}^{-1} (\Ext^1_g(D_{\sigma}, D_{\sigma}))$. The corollary then follows from the definition of $\Ext^1_0$'s and (\ref{Ecowg}).
\end{proof}
For $\Ext^1_?(D_{\sigma}, D_{\sigma}) \supset \Ext^1_0(D_{\sigma}, D_{\sigma})$, set $\ol{\Ext}^1_?(D_{\sigma}, D_{\sigma}):=\frac{\Ext^1_?(D_{\sigma}, D_{\sigma})}{\Ext^1_0(D_{\sigma}, D_{\sigma})}$. By (the first statement of) Corollary \ref{Cext10} and Proposition \ref{Pcow2}, we easily deduce:
\begin{corollary}\label{Ccow2}For $*\in \{\sigma, g, \{\sF_P,\sigma\}\}$, the (surjective) map $\fT_{\sigma}: \Ext^1_{*}(D,D) \ra \Ext^1_{*}(D_{\sigma}, D_{\sigma})$ induces an isomorphism 
$\fT_{\sigma}:	\ol{\Ext}^1_{*}(D,D) \xrightarrow{\sim} \ol{\Ext}^1_{*}(D_{\sigma}, D_{\sigma})$.
Moreover, there is a natural  commutative diagram
\begin{equation*}
	\begin{CD} 		\ol{\Ext}^1_{\sigma,\sF_P}(D,D) @> (\ref{EkappaFP11}) > \sim > \prod_{i=1}^r \ol{\Ext}^1_{\sigma}(M_i,M_i) \\
		@V \fT_{\sigma} V \sim V @V \fT_{\sigma} V \sim V \\
		\ol{\Ext}^1_{\sigma, \sF_P}(D_{\sigma}, D_{\sigma}) @> \sim >> \prod_{i=1}^r \ol{\Ext}^1_{\sigma}(M_{i,\sigma}, M_{i,\sigma}).
	\end{CD}
\end{equation*}
\end{corollary}

\subsection{Hodge filtration and higher intertwining}\label{S24}
Let $D\in \Phi\Gamma_{\nc}(\phi, \textbf{h})$. The existence  of $S_n$-distinct trianguline filtrations of $D$ corresponds to  an intertwining phenomenon on the automorphic side. We adapt the term ``intertwining"  to describe the non-uniqueness of saturated $(\varphi, \Gamma)$-submodules in such modules.  Analogously,  \textit{higher intertwining} in this section refers to the non-uniqueness of filtrations of saturated  $(\varphi, \Gamma)$-submodules over $\cR_{K, E[\epsilon]/\epsilon^2}$ for a $(\varphi, \Gamma)$-module over $\cR_{K, E[\epsilon]/\epsilon^2}$. By Corollary  \ref{Cint} (2), higher intertwining relations exist for $\widetilde{D}\in \Ext^1_{\sF_P,g'}(D,D)$. In this section, we show a special class of paraboline deformations of $D$ admits higher intertwining (cf. Theorem \ref{ThIW1} below). Moreover, the Hodge parameter, reinterpreted as in \S~\ref{S22}, can be revealed in such higher intertwining relations. 

Let $D_1$, $C_1$ be as in \S~\ref{S22}.   Let $\sF$ be the filtration $D_1 \subset D$, and $\sG$ be the filtration $\cR_{K,E}(\phi_nz^{\textbf{h}_n}) \subset D$, which correspond to the exact sequences (\ref{ED1}) (\ref{EC1}) respectively. By Proposition \ref{Ppara1}, we have
$\dim_E\Ext^1_{\sF}(D, D)=\dim_E \Ext^1_{\sG}(D,D)=1+(n^2-n+1) d_K$. 
And there are natural surjections (identifying $\Ext^1_{K^{\times}}(\delta, \delta)$ with $\Hom(K^{\times},E)$):
\begin{eqnarray}\label{EkappaF}
\ \ \kappa_{\sF}&=&(\kappa_{\sF,1}, \kappa_{\sF,2}): \Ext^1_{\sF}(D,D) \twoheadlongrightarrow \Ext^1(D_1,D_1)\times \Hom(K^{\times},E), \\
\kappa_{\sG}&=&(\kappa_{\sG,1}, \kappa_{\sG,2}): \Ext^1_{\sG}(D,D) \twoheadlongrightarrow \Ext^1(C_1,C_1)\times \Hom(K^{\times},E). \nonumber
\end{eqnarray}
We introduce certain subspaces of $\Ext^1(D_1, D_1)$ and $\Ext^1(C_1,C_1)$. For $\iota\in \Hom(D_1,C_1)$. Consider the pull-back and push-forward maps:
\begin{equation}\label{Eiotapm}
\iota^-: \Ext^1(C_1, D_1) \lra \Ext^1(D_1, D_1), \ \iota^+: \Ext^1(C_1, D_1) \lra \Ext^1(C_1, C_1).
\end{equation}
Set
$\Ext^1_{\iota}(D_1,D_1):=\iota^-(\Ext^1(C_1, D_1))$, $\Ext^1_{\iota}(C_1, C_1):=\iota^+(\Ext^1(C_1,D_1))$.
\begin{lemma}\label{Ldalphai} Suppose $\dim_E \Hom(D_1, C_1)=2$, and for $i\in \{1,\cdots, n-1\}$, let $\alpha_i$ be as in (\ref{Ealphai}).  We have $\dim_E \Ext^1_{\alpha_i}(D_1, D_1) =(n-1)(n-2) d_K$. Moreover for $j\in \{1,\cdots, n-1\}$, $j\neq i$, 
\begin{equation*}
	\dim_E \big(\Ext^1_{\alpha_i}(D_1, D_1) \cap \Ext^1_{\alpha_j}(D_1,D_1)\big)=(n-1)(n-3) d_K+d_K-1. 
\end{equation*}
Consequently, $\dim_E \big(\Ext^1_{\alpha_i}(D_1,D_1)+\Ext^1_{\alpha_j}(D_1,D_1)\big)=1+n(n-2)d_K$.
Finally, the same statement holds with $D_1$ replaced by $C_1$.
\end{lemma}
\begin{proof}
We only prove it for $D_1$, $C_1$ being similar. Fix the refinement $(\phi_1, \cdots, \phi_{n-1})$ of $D_1$ and $C_1$. Let $\textbf{r}:=\{1,\cdots, n-1\}\setminus \{i\}$. The map $\alpha_i^-$ factors through
$\Ext^1(C_1, D_1) \twoheadrightarrow \Ext^1((D_1)^{\textbf{r}}, D_1) \hookrightarrow \Ext^1(D_1, D_1)$
where the corresponding  surjectivity and injectivity follow easily by d\'evissage. So $\Ext^1_{\alpha_i}(D_1,D_1)$ is just the image of $\Ext^1((D_1)^{\textbf{r}}, D_1)$ in $\Ext^1(D_1, D_1)$,  and is the kernel of the natural pull-back map $\kappa_i: \Ext^1(D_1, D_1) \ra \Ext^1(\cR_{K,E}(\phi_i z^{\textbf{h}_1}), D_1)$. We directly calculate  $\dim_E \Ext^1((D_1)^{\textbf{r}}, D_1)=(n-1)(n-2) d_K$, and the first part follows.  For $i\neq j$, consider the following composition (of natural pull-back maps)
\begin{multline}\label{EDij}
	 \Ext^1(D_1, D_1) \xlongrightarrow{\kappa_{i,j}} \Ext^1((D_1)_{\{i,j\}}, D_1) \xlongrightarrow{f_{i,j}}
	 \big(\Ext^1(\cR_{K,E}(\phi_i z^{\textbf{h}_1}) \\ \oplus \cR_{K,E}(\phi_j,z^{\textbf{h}_1}),D_1) \big)
	 \cong \Ext^1(\cR_{K,E}(\phi_i z^{\textbf{h}_1}), D_1) \oplus \Ext^1(\cR_{K,E}(\phi_j z^{\textbf{h}_1}), D_1),
\end{multline}
whose kernel is clearly $\Ext^1_{\alpha_i}(D_1,D_1) \cap \Ext^1_{\alpha_j}(D_1,D_1)$.  By d\'evissage, $\kappa_{i,j}$ is surjective and $\Ker(\kappa_{i,j})\cong \Ext^1((D_1)^{i,j}, D_1)$, hence has dimension equal to $(n-3)(n-1) d_K$. 
Let $M_1:=D_1 \otimes_{\cR_{K,E}} (D_1)_{\{i,j\}}^{\vee}$ and $M_2:=D_1 \otimes_{\cR_{K,E}} (\cR_{K,E}(\phi_i^{-1} z^{-\textbf{h}_1}) \oplus \cR_{K,E}(\phi_j^{-1}z^{-\textbf{h}_1}))$. By d\'evissage, we have
\begin{equation*}
	0 \ra H^0_{(\varphi, \Gamma)}(M_1) \ra H^0_{(\varphi, \Gamma)}(M_2)
	\ra H^0_{(\varphi, \Gamma)}(M_2/M_1) \ra H^1_{(\varphi, \Gamma)}(M_1) \ra H^1_{(\varphi, \Gamma)}(M_2)
\end{equation*}
where the last map coincides with $f_{i,j}$ in (\ref{EDij}). We have $\dim_E H^0_{(\varphi, \Gamma)}(M_1)=1$, $\dim_E H^0_{(\varphi, \Gamma)}(M_2)=2$, and by Lemma \ref{Lemtorsion}, 
$\dim_EH^0_{(\varphi, \Gamma)}(M_2/M_1)=\dim_E D_{\dR}^+(M_2)- \dim_ED_{\dR}^+(M_1) =2(n-1)d_K-(n-1+n-2)d_K=d_K$.
So $\dim_E \Ext^1_{\alpha_i}(D_1,D_1) \cap \Ext^1_{\alpha_j}(D_1,D_1)=\dim_E \Ker(\kappa_{i,j})+\dim_E \Ker(f_{i,j})=(n-1)(n-3)d_K+d_K-1$. This proves the second part of the lemma.
\end{proof}
\begin{proposition}\label{Phint1}
Let $\iota\in \Hom(D_1, C_1)$ be an injection.

(1)  $\dim_E \Ext^1_{\iota}(D_1, D_1)=\dim_E \Ext^1_{\iota}(C_1,C_1)=1+(n-1)(n-2) d_K$.

(2) $\Ext_g^1(D_1, D_1) \subset \Ext^1_{\iota}(D_1,D_1)$ and $\Ext_g^1(C_1, C_1) \subset \Ext^1_{\iota}(C_1,C_1)$.

(3) For $\iota'\in \Hom(D_1,C_1)$, $\Ext^1_{\iota'}(D_1, D_1)=\Ext^1_{\iota}(D_1,D_1)$ if and only if $\Ext^1_{\iota'}(C_1, C_1)=\Ext^1_{\iota}(C_1,C_1)$ if and only if $\iota'=a \iota$ for some $a\in E^{\times}$.
\end{proposition}
\begin{proof}
We only prove it for $D_1$ with $C_1$ being similar. 

(1) By d\'evissage, we have
\begin{multline*}
	0 \lra H^0_{(\varphi, \Gamma)}(D_1 \otimes_{\cR_{K,E}} D_1^{\vee}) \lra H^0_{(\varphi, \Gamma)}\big((D_1 \otimes_{\cR_{K,E}}  D_1^{\vee})/(D_1 \otimes_{\cR_{K,E}} C_1^{\vee})\big)\\ \lra H^1_{(\varphi, \Gamma)}(D_1 \otimes_{\cR_{K,E}} C_1^{\vee}) \lra H^1_{(\varphi, \Gamma)}(D_1 \otimes_{\cR_{K,E}} D_1^{\vee}),
\end{multline*}
where the last map can be identified with $\iota^-$. By Lemma \ref{Lemtorsion}, we have
\begin{multline*}
	\dim_E H^0_{(\varphi, \Gamma)}\big((D_1 \otimes_{\cR_{K,E}}  D_1^{\vee})/(D_1 \otimes_{\cR_{K,E}} C_1^{\vee})\big)=\dim_E D_{\dR}^+(D_1 \otimes_{\cR_{K,E}}  D_1^{\vee})\\ -\dim_E D_{\dR}^+(D_1 \otimes_{\cR_{K,E}} C_1^{\vee}) =\frac{n(n-1)}{2}d_K-\frac{(n-1)(n-2)}{2}d_K=(n-1)d_K.
\end{multline*}
Hence $\dim_E \Ima \iota^- =(n-1)^2 d_K-(n-1)d_K+1=1+(n-1)(n-2)d_K$. 

(2) The map $\iota^-$ clearly induces
$\iota_g^-:	\Ext^1_g(C_1, D_1) \ra \Ext^1_g(D_1,D_1)$.
For any $M\in \Ker(\iota^-)$, $D_1\oplus D_1 \subset M$ implies $M$ is de Rham. So 
$\Ker \iota^-\subset \Ext^1_g(C_1,D_1)$ and is equal to $\Ker \iota_g^-$. By \cite[Cor.~A.4]{Ding6} applied to the $B$-pair associated to $D_1 \otimes_{\cR_{K,E}} C_1^{\vee}$ (which satisfies the assumptions of \textit{loc. cit.} by the generic assumption on $D$), we have $\dim_E H^1_g(D_1 \otimes_{\cR_{K,E}} C_1^{\vee})=(n-1)^2 d_K-\frac{(n-1)(n-2)}{2}d_K=\frac{n(n-1)}{2} d_K$. Together with $\dim_E \Ext^1_g(D_1,D_1)=1+\frac{(n-1)(n-2)}{2} d_K$ (cf. Proposition \ref{PDef1} (1)) and (1), we see  $\iota_g^-$ is surjective.

(3) 
The case where $\dim_E \Hom(D_1, C_1)=1$ is trivial. Assume henceforth $\dim_E\Hom(D_1,C_1)=2$ (which implies $n\geq 3$ and Lemma \ref{Ldalphai} can apply).
Suppose $\iota'\notin E[\iota]$, then $\iota'$ and $\iota$ form a basis of $\Hom(D_1, C_1)$. If $\Ext^1_{\iota'}(D_1, D_1)=\Ext^1_{\iota}(D_1,D_1)$, we then easily deduce $\Ext^1_{\alpha_i}(D_1, D_1) \subset \Ext^1_{\iota}(D_1, D_1)$ for all $i=\{1,\cdots, n-1\}$. However, for $i\neq j$, by Lemma \ref{Ldalphai}, $\dim_E (\Ext^1_{\alpha_i}(D_1, D_1)+\Ext^1_{\alpha_j}(D_1, D_1))=1+n(n-2) d_K>\dim_E \Ext^1_{\iota}(D_1, D_1)$, a contradiction.
\end{proof}
Let $T_1$ be the torus subgroup of $\GL_{n-1}$, and $\phi^1:=\phi_1 \boxtimes \cdots \boxtimes \phi_{n-1}$. Let $\textbf{h}^1:=(\textbf{h}_1, \cdots, \textbf{h}_{n-1})$, and $\textbf{h}^2:=(\textbf{h}_2, \cdots, \textbf{h}_{n})$. For the refinement $\phi^1$ (of $D_1$ and $C_1$), we have maps 
\begin{equation*}
\Ext^1_g(D_1,D_1) \xlongrightarrow{\kappa_{\phi^1}}  \Hom_{\sm}(T_1(K),E), \ \Ext^1_g(C_1,C_1) \xlongrightarrow{\kappa_{\phi^1}}  \Hom_{\sm}(T_1(K),E).
\end{equation*}
\begin{lemma}
For $M\in \Ext^1_g(C_1, D_1)$, $\kappa_{\phi^1} \circ \iota_g^- (M)=\kappa_{\phi^1} \circ \iota_g^+(M)$, where $\iota_g^{\pm}$ is the restriction of $\iota^{\pm}$ to $\Ext^1_g(C_1,D_1)$ (see the proof of Proposition \ref{Phint1} (2)).
\end{lemma}
\begin{proof}
By definition, there is a natural injection $\tilde{\iota}: \iota_g^-(M) \hookrightarrow \iota_g^+(M)$ which sits in the following commutative diagram
\begin{equation*}
	\begin{CD}
		0 @>>> D_1 @>>> \iota_g^-(M) @>>> D_1 @>>> 0 \\
		@. @V \iota VV @V \tilde{\iota} VV @V \iota VV \\
		0 @>>> C_1 @>>> \iota_g^+(M) @>>> C_1 @>>> 0.
	\end{CD}
\end{equation*}
Moreover, $\tilde{\iota}$ is $\cR_{K,E[\epsilon]/\epsilon^2}$-linear if $\iota_g^-(M)$ and $\iota_g^+(M)$ are equipped with the natural $\cR_{K, E[\epsilon]/\epsilon^2}$-action. Suppose $$\kappa_{\phi^1}\circ \iota_g^-(M)=(\psi_1, \cdots, \psi_{n-1}),  \ \ \kappa_{\phi^1}\circ \iota_g^+(M)=(\psi_1', \cdots, \psi_{n-1}').$$ Then $\iota_g^-(M)$ (resp. $\iota_g^+(M)$) is isomorphic, as $(\varphi, \Gamma)$-module over $\cR_{K, E[\epsilon]/\epsilon^2}$, to a successive extension of $\cR_{K, E[\epsilon]/\epsilon^2}(\phi_i z^{\textbf{h}_i} (1+\psi_i \epsilon))$ (resp. $\cR_{K, E[\epsilon]/\epsilon^2}(\phi_i z^{\textbf{h}_{i+1}}(1+\psi'_i \epsilon))$) for $i=1,\cdots, n-1$. One sees inductively that  $\tilde{\iota}$ induces injections $\cR_{K, E[\epsilon]/\epsilon^2}(\phi_i z^{\textbf{h}_i} (1+\psi_i \epsilon)) \hookrightarrow \cR_{K, E[\epsilon]/\epsilon^2}(\phi_i z^{\textbf{h}_{i+1}} (1+\psi_i' \epsilon))$ of $(\varphi, \Gamma)$-modules over $\cR_{K, E[\epsilon]/\epsilon^2}$. Hence $\psi_i=\psi_i'$ for all $i$.
\end{proof}
We fix crystabelline $(\varphi, \Gamma)$-modules $D_1$ and $C_1$, where $D_1$  has Hodge-Tate-Sen weights $\textbf{h}^1$ and $C_1$ has weights $\textbf{h}^2$, and both have a generic refinement $\phi^1$. Denote by $\Phi\Gamma_{\nc}(D_1,C_1, \phi_n)\subset \Phi\Gamma_{\nc}(\phi, \textbf{h})$ the subset of isomorphism classes of $(\varphi,\Gamma)$-modules $D$ such that $\Hom(D_1,D)=\Hom(D,C_1)\cong E$.  Assume $\Phi\Gamma_{\nc}(D_1,C_1,\phi_n)$ is non-empty.
For an injection $\iota\in \Hom(D_1,C_1)$, we set $\sI_{\iota}$ to be the following set \begin{multline}\label{EsIiota}
\big\{(\widetilde{D}_1, \widetilde{C}_1)\in \Ext^1_{\iota}(D_1,D_1) \times \Ext^1_{\iota}(C_1,C_1)\ |\ \\ \exists M\in \Ext^1(C_1, D_1) \text{ s.t. } \iota^-(M)=\widetilde{D}_1, -\iota^+(M)=\widetilde{C}_1\big\}. 
\end{multline}
If $\iota=\iota_D$ for some $D\in \Phi\Gamma_{\nc}(D_1, C_1, \phi_n)$, we write $\sI_D:=\sI_{\iota_D}$. The following corollary is a direct consequence of Proposition \ref{Phint1} (3) and Proposition \ref{Phodge}. 
\begin{corollary}
We have $\sI_{\iota}=\sI_{\iota'}$ if and only if $\iota'=a \iota$ for some $a\in E^{\times}$. In particular, for $D$, $D'\in \Phi\Gamma_{\nc}(D_1, C_1, \phi_n)$ we have  $\sI_D=\sI_{D'}$ if and only if $\iota_D=a \iota_{D'}$ for $a\in E^{\times}$. When $K=\Q_p$, this is equivalent to $D\cong D'$.
\end{corollary}

\begin{theorem}[Higher intertwining]\label{ThIW1}Let $D\in \Phi\Gamma_{\nc}(D_1, C_1, \phi_n)$ and $\widetilde{D}\in \Ext^1_{\sF}(D, D)$ with $\kappa_{\sF}(\widetilde{D})=(\widetilde{D}_1, \psi)$ (cf. (\ref{EkappaF})). The followings are equivalent:
\begin{enumerate}
	\item  $\widetilde{D} \in \Ext^1_{\sF}(D,D) \cap \Ext^1_{\sG}(D,D)$.
	\item  $\widetilde{D}_1 \otimes_{\cR_{K, E[\epsilon]/\epsilon^2}} \cR_{E[\epsilon]/\epsilon^2}(1-\psi \epsilon)\in \Ext^1_{\iota_D}(D_1,D_1)$.
\end{enumerate}		
Moreover, if the equivalent conditions hold, then $\kappa_{\sG,2}(\widetilde{D})=\psi$ and	 there exists $M\in \Ext^1(C_1,D_1)$ such that $\widetilde{D}_1=\iota_D^-(M) \otimes_{\cR_{K, E[\epsilon]/\epsilon^2}} \cR_{K, E[\epsilon]/\epsilon^2}(1+\psi \epsilon)$ and $\kappa_{\sG,1}(\widetilde{D})=\iota_D^+(M)\otimes_{\cR_{K, E[\epsilon]/\epsilon^2}}\cR_{K, E[\epsilon]/\epsilon^2}(1+\psi \epsilon)$.
\end{theorem}
\begin{proof}
Twisting $\widetilde{D}$ by $1-\psi \epsilon$, we can and do assume $\kappa_{\sF,2}(\widetilde{D})=0$. By definition, $\widetilde{D}\in \Ext^1_{\sG}(D,D)$  if and only if it lies in the kernel of the composition
\begin{equation}\label{0}
	\Ext^1(D,D) \lra \Ext^1(\cR_{K,E}( \phi_n z^{\textbf{h}_1}), D) \lra \Ext^1(\cR_{K,E}(\phi_n z^{\textbf{h}_1}), C_1).
\end{equation}
Similarly, $\Ext^1_{\sF}(D,D)$ is equal to the kernel of the composition
\begin{equation*}
	\Ext^1(D,D) \lra \Ext^1(D, \cR_{K,E}(\phi_n z^{\textbf{h}_n})) \lra \Ext^1(D_1, \cR_{K,E}(\phi_n z^{\textbf{h}_n})). 
\end{equation*}
By d\'evissage, one can deduce an exact sequence
\begin{equation*}
	0 \ra 	\Ext^1(D,D_1) \lra \Ext^1_{\sF}(D,D) \lra \Ext^1(\cR_{K,E}(\phi_n z^{\textbf{h}_n}), \cR_{K,E}(\phi_n z^{\textbf{h}_n}))\lra 0.
\end{equation*}
As $\kappa_{\sF,2}(\widetilde{D})=0$, $\widetilde{D}$ lies in the image of $\Ext^1(D,D_1) \ra \Ext^1_{\sF}(D,D)$. Let $M_1\in \Ext^1(D,D_1) $ be the  preimage of $\widetilde{D}$. Consider the composition 
\begin{equation*}
	\Ext^1(D,D_1) \hookrightarrow	\Ext^1(D,D) \ra \Ext^1(\cR_{K,E}(\phi_n z^{\textbf{h}_1}), D) \ra \Ext^1(\cR_{K,E}( \phi_n z^{\textbf{h}_1}), C_1).
\end{equation*}
It is straightforward to see it is equal to the composition
\begin{equation}\label{EDtoM}
	\Ext^1(D,D_1) \lra \Ext^1(\cR_{K,E}(\phi_n z^{\textbf{h}_1}), D_1) \xlongrightarrow{\iota_D} \Ext^1(\cR_{K,E}(\phi_n z^{\textbf{h}_1}), C_1).
\end{equation}
So $\widetilde{D}$ lies in the kernel of (\ref{0}) if and only if $M_1$ is sent to zero via (\ref{EDtoM}). However, using d\'evissage, we see the kernel of $\iota_D$ in (\ref{EDtoM}) is isomorphic to $H^0_{(\varphi, \Gamma)}(\cR_{K,E}(\phi_n^{-1} z^{-\textbf{h}_1}) \otimes_{\cR_{K,E}}(C_1/D_1))$, which, by Lemma \ref{Lemtorsion}, has dimension $\dim_E D_{\dR}^+(\cR_{K,E}(\phi_n^{-1} z^{-\textbf{h}_1}) \otimes_{\cR_{K,E}} C_1)-\dim_E D_{\dR}^+(\cR_{K,E}(\phi_n^{-1} z^{-\textbf{h}_1}) \otimes_{\cR_{K,E}} D_1)=0$. So  $\iota_D$ in (\ref{EDtoM}) is injective. We see (under the assumption $\psi=0$) that (1) is equivalent to  that $M_1$ lies in the kernel of the first map of (\ref{EDtoM}), which is equal to $\Ext^1(C_1, D_1)$ by d\'evissage. This is furthermore equivalent to that  $\widetilde{D}_1$ lies in the image of the composition 
$	\Ext^1(C_1,D_1) \hookrightarrow \Ext^1(D,D_1) \ra \Ext^1(D_1,D_1)$,
which is  no other than $\iota_D^-$.  The other parts are straightforward. 
\end{proof}
\begin{corollary}\label{Csur0}
We have $\dim_E (\Ext^1_{\sF}(D,D) \cap \Ext^1_{\sG}(D,D))=1+(n^2-2n+2)d_K$. Consequently, the following natural map is surjective: 
\begin{equation}\label{Epinf}
	\Ext^1_{\sF}(D,D) \oplus \Ext^1_{\sG}(D,D) \twoheadlongrightarrow \Ext^1(D,D).
\end{equation}
\end{corollary}
\begin{proof}
By  Theorem \ref{ThIW1},  $\dim_E (\Ext^1_{\sF}(D,D) \cap \Ext^1_{\sG}(D,D))=\dim_E\Hom(K^{\times},E) +\dim_E \Ext^1_{\iota_D}(D_1,D_1)+\dim_E \Ker(\kappa_{\sF})$ (cf. (\ref{EkappaF})). By Proposition \ref{Phint1} (1) and Proposition \ref{Ppara1}, it is equal to $(1+d_K)+(1+(n-1)(n-2)d_K)+(-1+(n-1)d_K)=1+(n^2-2n+2)d_K$. Together with  Proposition \ref{Ppara1}, we see $\dim_E (\Ext^1_{\sF}(D,D) +\Ext^1_{\sG}(D,D))=2(1+(n(n-1)+1)d_K-1+(n^2-2n+2)d_K=1+n^2 d_K{\buildrel {\text{Prop}.~\ref{PDef1} (1)} \over =} \dim_E \Ext^1(D,D)$. The second part follows.
\end{proof}
\!\!\!\!\!\! \!\!\!Let 
$
V(D_{1}, C_{1})\!\!:= \big(\ol{\Ext}^1(D_1, D_1) \times\Hom(K^{\times},E)\big) \oplus 	\big(\ol{\Ext}^1(C_1, C_1) \times \Hom(K^{\times},E)\big)$
$\big(	\xleftarrow[\sim]{(\kappa_{\sF}, \kappa_{\sG})} \ol{\Ext}^1_{\sF}(D,D) \oplus \ol{\Ext}^1_{\sG}(D,D)\big),
$
and $\cL(D,D_1, C_1)$ be the subspace consisting of those $\big(\big(\overline{\widetilde{D}_1}, \psi\big), \big(\overline{\widetilde{C}_1}, -\psi\big)\big)\in V(D_1,C_1)$ such that  $\big(\widetilde{D}_1 \otimes_{\cR_{K, E[\epsilon]/\epsilon^2}} \cR_{K, E[\epsilon]/\epsilon^2}(1-\psi \epsilon), \widetilde{C}_1\otimes_{\cR_{K, E[\epsilon]/\epsilon^2}} \cR_{K, E[\epsilon]/\epsilon^2}(1+\psi \epsilon)\big)\in \sI_D$ (cf. (\ref{EsIiota})).
\begin{corollary}\label{ChIw}
(1) Let $D, D'\in \Phi\Gamma_{\nc}(D_1, C_1, \phi_n)$, $\cL(D',D_1,C_1)=\cL(D,D_1,C_1)$ if and only if $\iota_{D'}=a\iota_D$ for some $a\in E^{\times}$. When $K=\Q_p$, this is equivalent to $D\cong D'$.

(2) For $D \in \Phi\Gamma_{\nc}(D_1, C_1, \phi_n)$, there is a natural exact sequence
\begin{equation}\label{Eexa1}
	0 \lra \cL(D,D_1,C_1) \lra V(D_1, C_1) \lra \ol{\Ext}^1(D,D) \lra 0.
\end{equation}
\end{corollary}
\begin{proof}
(1): The ``if" part is trivial. Suppose $\cL(D',D_1,C_1)=\cL(D,D_1,C_1)$. Let $\widetilde{D}_1\in \Ext^1_{\iota_D}(D_1,D_1)$ $M\in \Ext^1(C_1,D_1)$ be a preimage of $\widetilde{D}_1$ (via $\iota_D^-$) and $\widetilde{C}_1:=-\iota_D^+(M)\in \Ext^1_{\iota_D}(C_1,C_1)$. We have by definition and assumption $$\big((\overline{\widetilde{D}_1}, 0), (\overline{\widetilde{C}_1},0)\big)\in \cL(D,D_1,C_1)=\cL(D',D_1,C_1).$$ There exists hence $\widetilde{D}_1'\in \Ext^1_{\iota_{D'}}(D_1,D_1)$ such that $[\widetilde{D}_1']-[\widetilde{D}_1]\in \Ext^1_0(D_1,D_1)$. As $\Ext^1_0(D_1,D_1)\subset \Ext^1_{\iota_{D'}}(D_1,D_1)$ (by Proposition \ref{Phint1} (2)), this implies $\widetilde{D}_1\in \Ext^1_{\iota_D'}(D_1,D_1)$. So $\Ext^1_{\iota_D}(D_1,D_1)\subset \Ext^1_{\iota_{D'}}(D_1,D_1)$ hence $\Ext^1_{\iota_{D'}}(D_1,D_1)=\Ext^1_{\iota_D}(D_1,D_1)$ by symmetry and $\iota_{D'}\in E^{\times} \iota_D$ by Proposition \ref{Phint1} (3). 

(2) Quotienting (\ref{Epinf}) by $\Ext^1_0(D,D)$ yields a surjection $V(D_1,C_1)\twoheadrightarrow \ol{\Ext}^1(D,D)$. By Theorem \ref{ThIW1}, the kernel is exactly $\cL(D,D_1,C_1)$.
\end{proof}

Now we consider $\Sigma_K \setminus \{\sigma\}$-de Rham deformations for general $K$.  Let $D_{1,\sigma}=\fT_{\sigma}(D_1)$ and  $C_{1,\sigma}=\fT_{\sigma}(C_1)$ (cf. (\ref{Ecow})). Let $\iota_{\sigma}\in \Hom(D_{1,\sigma}, C_{1,\sigma})$. We have similar maps as in (\ref{Eiotapm}), which induce, by restricting to $\Sigma_K\setminus \{\sigma\}$-de Rham extension groups, 
\begin{equation*}
\iota_{\sigma}^-: \Ext^1_{\sigma}(C_{1,\sigma}, D_{1,\sigma}) \ra \Ext^1_{\sigma}(D_{1,\sigma}, D_{1,\sigma}), \ \iota_{\sigma}^+: \Ext^1_{\sigma}(C_{1,\sigma}, D_{1,\sigma}) \ra \Ext^1_{\sigma}(C_{1,\sigma}, C_{1,\sigma}).
\end{equation*}
Let $\Ext^1_{\iota_{\sigma}}(D_{1,\sigma},D_{1,\sigma}):=\Ima(\iota_{\sigma}^-)$, $\Ext^1_{\iota_{\sigma}}(C_{1,\sigma}, C_{1,\sigma}):=\Ima(\iota_{\sigma}^+)$. Denote by 
\begin{multline}\label{EsIiotasigma}
\sI_{\iota_{\sigma}}:=\{(\widetilde{D}_{1,\sigma}, \widetilde{C}_{1,\sigma})\in \Ext^1_{\iota_{\sigma}}(D_{1,\sigma},D_{1,\sigma}) \times \Ext^1_{\iota_{\sigma}}(C_{1,\sigma},C_{1,\sigma})\ |\ \\
\exists M\in \Ext^1_{\sigma}(C_{1,\sigma}, D_{1,\sigma}) \text{ with } \iota_{\sigma}^-(M)=\widetilde{D}_{1,\sigma}, -\iota^+_{\sigma}(M)=\widetilde{C}_{1,\sigma}\}. 
\end{multline}
Similarly as in Proposition \ref{Phint1}, we have:
\begin{proposition}\label{Piotasigma}
Let $\iota_{\sigma}\in \Hom(D_{1,\sigma}, C_{1,\sigma})$ be an injection.

(1)  $\dim_E \Ext^1_{\iota_{\sigma}}(D_{1,\sigma}, D_{1,\sigma})=\dim_E \Ext^1_{\iota_{\sigma}}(C_{1,\sigma},C_{1,\sigma})=1+(n-1)(n-2) $.

(2) $\Ext_g^1(D_{1,\sigma}, D_{1,\sigma}) \subset \Ext^1_{\iota_{\sigma}}(D_{1,\sigma},D_{1,\sigma})$ and $\Ext_g^1(C_{1,\sigma}, C_{1,\sigma}) \subset \Ext^1_{\iota_{\sigma}}(C_{1,\sigma},C_{1,\sigma})$.

(3) For $\iota_{\sigma}'\in \Hom(D_{1,\sigma},C_{1,\sigma})$, $\Ext^1_{\iota_{\sigma}'}(D_{1,\sigma}, D_{1,\sigma})=\Ext^1_{\iota_{\sigma}}(D_{1,\sigma},D_{1,\sigma})$ if and only if  $\Ext^1_{\iota_{\sigma}'}(C_{1,\sigma}, C_{1,\sigma})=\Ext^1_{\iota_{\sigma}}(C_{1,\sigma},C_{1,\sigma})$ if and only if $\iota_{\sigma}'=a \iota_{\sigma}$ for some $a\in E^{\times}$. 
\end{proposition}
\begin{proof}We \ still \ only \ prove \ the \ statements \ for \ $D_{1,\sigma}$. \ 
By \ \cite[Cor.~A.4]{Ding6}, \ $\dim_E \Ext^1_{\sigma}(C_{1,\sigma}, D_{1,\sigma})=(n-1)^2 d_K-\sum_{\tau\in \Sigma_K \setminus \{\sigma\}} \dim_E D_{\dR}^+(D_{1,\sigma}\otimes_{\cR_{K,E}} C_{1,\sigma}^{\vee})_{\tau}=(n-1)^2.$ By similar arguments as in the proof of Proposition \ref{Phint1} (1), the kernel of $ \Ext^1(C_{1,\sigma}, D_{1,\sigma}) \ra \Ext^1(D_{1,\sigma}, D_{1,\sigma})$ has dimension $(n-1)-1$. But any element in this kernel contains $D_{1,\sigma} \oplus D_{1,\sigma}$ hence is de Rham. We see it is the same as $\Ker \iota_{\sigma}^-$ and $\Ker \big(\iota_{\sigma}^-|_{\Ext^1_g(C_{1,\sigma}, D_{1,\sigma})}\big)$. (1) follows. Using \cite[Cor.~A.4]{Ding6}, $\dim_E\Ext^1_g(C_{1,\sigma}, D_{1,\sigma})=\frac{n(n-1)}{2}$. Together with Proposition \ref{Ppara2} (2) and comparing dimensions, the induced map $\Ext^1_g(C_{1,\sigma}, D_{1,\sigma}) \ra \Ext^1_g(D_{1,\sigma},D_{1,\sigma})$ is surjective. (2) follows. (3) follows from  similar arguments as in the proof of Proposition \ref{Phint1} (3) using an analogue of Lemma \ref{Ldalphai} for $\Ext^1_{\alpha_{i,\sigma}}$ with $\alpha_{i,\sigma}$ given as in (\ref{Ealphaisig}) (when $\dim_E \Hom(D_{1,\sigma}, C_{1,\sigma})=2$). Note the d\'evissage arguments in the proof of Lemma \ref{Ldalphai} work when  $\Ext^1$'s are all replaced by $\Ext^1_{\sigma}$'s, by \cite[Prop.~A.5]{Ding6}. We leave the details to the reader.
\end{proof}
For $D_{\sigma}\in \Phi\Gamma_{\nc}(D_{1,\sigma}, C_{1,\sigma}, \phi_n)$ (which is the subset of $\Phi\Gamma_{\nc}(\phi, \sT_{\sigma}(\textbf{h}))$ defined similarly as $\Phi\Gamma_{\nc}(D_1,C_1, \phi_n)$), set $\sI_{D_{\sigma}}:=\sI_{\iota_{D_{\sigma}}}$ (cf. (\ref{EsIiotasigma})) where $\iota_{D_{\sigma}}$ is the composition $D_{1,\sigma}\hookrightarrow D_{\sigma} \twoheadrightarrow C_{1,\sigma}$.
We have by Proposition \ref{Piotasigma} (3) and Proposition \ref{Piotasigma0}:
\begin{corollary}
For $D_{\sigma}$, $D_{\sigma}'\in \Phi\Gamma_{\nc}(D_{1,\sigma},C_{1,\sigma}, \phi_n)$, we have $\sI_{D_{\sigma}}=\sI_{D_{\sigma}'}$ if and only if $D_{\sigma}\cong D_{\sigma}'$.
\end{corollary}
Consider $\kappa_{\sF}: \Ext^1_{\sigma, \sF_P}(D_{\sigma}, D_{\sigma}) \ra \Ext^1_{\sigma}(D_{1,\sigma}, D_{1,\sigma}) \times \Hom_{\sigma}(K^{\times},E)$, and $\kappa_{\sG}: \Ext^1_{\sigma, \sF_P}(D_{\sigma}, D_{\sigma}) \ra \Ext^1_{\sigma}(C_{1,\sigma}, C_{1,\sigma}) \times \Hom_{\sigma}(K^{\times},E)$ (cf.  (\ref{EkapFpsigma}) and (\ref{Echaracsigma})). 
The following theorem follows by the same argument as in the proof of Theorem \ref{ThIW1} (note that all the d\'evissage arguments used in \textit{loc. cit.} work if $\Ext^1$'s are all  replaced by $\Ext^1_{\sigma}$'s by \cite[Prop.~A.5]{Ding6}).
\begin{theorem}\label{ThIW2}
Let $\widetilde{D}_{\sigma}\in \Ext^1_{\sigma, \sF}(D_{\sigma}, D_{\sigma})$ with $\kappa_{\sF}(\widetilde{D}_{\sigma})=(\widetilde{D}_{\sigma,1}, \psi)$. The followings are equivalent:
\begin{enumerate}
	\item  $\widetilde{D}_{\sigma} \in \Ext^1_{\sigma, \sF}(D_{\sigma},D_{\sigma}) \cap \Ext^1_{\sigma,\sG}(D_{\sigma},D_{\sigma})$,
	\item  $\widetilde{D}_{1,\sigma}\otimes_{\cR_{K, E[\epsilon]/\epsilon^2}} \cR_{E[\epsilon]/\epsilon^2}(1-\psi \epsilon)\in \Ext^1_{\iota_{D_{\sigma}}}(D_{1,\sigma},D_{1,\sigma})$.
\end{enumerate}		
Moreover, if the equivalent conditions hold, then $\kappa_{\sG,2}(\widetilde{D}_{\sigma})=\psi$ and	 there exists $M\in \Ext^1_{\sigma}(C_{1,\sigma},D_{1,\sigma})$ such that $\widetilde{D}_{1,\sigma} \cong \iota_{D_{\sigma}}^-(M) \otimes_{\cR_{K, E[\epsilon]/\epsilon^2}} \cR_{K, E[\epsilon]/\epsilon^2}(1+\psi \epsilon)$ and $\kappa_{\sG,1}(\widetilde{D}_{\sigma})=\iota_{D_{\sigma}}^+(M)\otimes_{\cR_{K, E[\epsilon]/\epsilon^2}}\cR_{K, E[\epsilon]/\epsilon^2}(1+\psi \epsilon)$.
\end{theorem}
Set  $V(D_{1,\sigma}, C_{1,\sigma})_{\sigma}:=  \big(\ol{\Ext}^1_{\sigma}(D_{1,\sigma}, D_{1,\sigma}) \times \Hom_{\sigma}(K^{\times},E)\big) \oplus 	\big(\ol{\Ext}^1_{\sigma}(C_{1,\sigma}, C_{1,\sigma}) \times \Hom_{\sigma}(K^{\times},E)\big)$ and $\cL(D_{\sigma},D_{1,\sigma},D_{2,\sigma})_{\sigma}$ to be the  the subspace consisting of those $\big((\overline{\widetilde{D}_{1,\sigma}}, \psi), (\overline{\widetilde{C}_{1,\sigma}}, -\psi)\big)\in V(D_{1,\sigma}, C_{1,\sigma})_{\sigma}$ such that  (cf. (\ref{EsIiotasigma}))
$$\big(\widetilde{D}_{1,\sigma}\otimes_{\cR_{K, E[\epsilon]/\epsilon^2}} \cR_{K, E[\epsilon]/\epsilon^2}(1-\psi \epsilon), \widetilde{C}_{1,\sigma}\otimes_{\cR_{K, E[\epsilon]/\epsilon^2}} \cR_{K, E[\epsilon]/\epsilon^2}(1+\psi \epsilon)\big)\in \sI_{D_{\sigma}}.$$
By Proposition \ref{Piotasigma0} and  the same arguments as in Corollary \ref{ChIw}, we have:
\begin{corollary}\label{Chodge1}
(1) Let $D_{\sigma}$, $D'_{\sigma}\in \Phi\Gamma_{\nc}(D_{1,\sigma},C_{1,\sigma}, \phi_n)$, then $\cL(D'_{\sigma},D_{1,\sigma},C_{1,\sigma})=\cL(D_{\sigma},D_{1,\sigma},C_{1,\sigma})$ if and only if $D_{\sigma}\cong D'_{\sigma}$.

(2) There is a natural exact sequence
\begin{equation}\label{Eparainf0}
	0 \lra \cL(D_{\sigma},D_{1,\sigma},C_{1,\sigma}) \lra V(D_{1,\sigma}, C_{1,\sigma})_{\sigma} \lra \ol{\Ext}^1_{\sigma}(D_{\sigma},D_{\sigma}) \lra 0.
\end{equation}
\end{corollary}
Set $V(D_1, C_1)_{\sigma}$ to be 
\begin{equation*}
	\big(\ol{\Ext}^1_{\sigma}(D_1, D_1) \times\Hom_{\sigma}(K^{\times},E)\big) \oplus 	\big(\ol{\Ext}^1_{\sigma}(C_1, C_1) \times \Hom_{\sigma}(K^{\times},E)\big) \subset V(D_1,C_1),
\end{equation*}
and $\cL(D,D_1,C_1)_{\sigma}:=\cL(D,D_1,C_1) \cap V(D_1, C_1)_{\sigma}\subset V(D_1,C_1)$. Note $V(D_1,C_1)_{\sigma} \cong \ol{\Ext}^1_{\sigma, \sF}(D,D) \oplus \ol{\Ext}^1_{\sigma,\sG}(D,D)$ by Proposition \ref{Pparasigma} (2).
\begin{proposition}\label{Phodgecow}
The functor $\fT_\sigma$ induces a commutative diagram of short exact sequences
\begin{equation*}
	\begin{CD}			0 @>>> \cL(D,D_1, C_1)_{\sigma}  @>>> V(D_{1}, C_{1})_{\sigma} @>>> \ol{\Ext}^1_{\sigma}(D, D) @>>> 0	 \\
		@. @V \fT_{\sigma} V \sim V @V \fT_{\sigma} V\sim V @V \fT_{\sigma} V \sim V \\
		0 @>>> \cL(D_{\sigma},D_{1,\sigma}, C_{1,\sigma})_{\sigma}  @>>> V(D_{1,\sigma}, C_{1,\sigma})_{\sigma} @>>> \ol{\Ext}^1_{\sigma}(D_{\sigma}, D_{\sigma}) @>>> 0
	\end{CD}
\end{equation*}
where the top sequence is induced by (\ref{Eexa1}). 
\end{proposition}
\begin{proof}
All the maps are clear, and we have seen in the above corollary that the bottom sequence is exact.   The left exactness of the top sequence is clear. It is also exact in the middle because of the definition of $\cL(D,D_1,C_1)_{\sigma}$. By Corollary \ref{Ccow2}, the two right vertical maps are both isomorphisms. The proposition follows.
\end{proof}
\begin{corollary}\label{Cinf2}
The map  (\ref{Einf2}) is surjective. And the same holds with $D$ replaced by $D_{\sigma}$.
\end{corollary}
\begin{proof}
By the above proposition,  $\ol{\Ext}^1_{\sigma, \sF}(D,D) \oplus \ol{\Ext}^1_{\sigma,\sG}(D,D) \cong V(D_1,C_1)_{\sigma}\ra \ol{\Ext}^1_{\sigma}(D,D)$ is surjective. Using Proposition \ref{Pparasigma} (2), Corollary \ref{CFpwsigma} and induction on the rank $n$, one deduces $\oplus_{w\in S_n} \ol{\Ext}^1_{\sigma, w}(D,D) \ra \ol{\Ext}^1_{\sigma}(D,D)$ is surjective. As $\Ext^1_0(D,D)\subset \Ext^1_{\sigma, w}(D,D)$ for any $w\in S_n$, we see (\ref{Einf2}) is also surjective. The statement for $D_{\sigma}$ follows by similar arguments or using Corollary \ref{Ccow2}.
\end{proof}

\section{Locally analytic crystabelline representations of $\GL_n(K)$}

\subsection{Locally analytic representations of $\GL_n(K)$ and extensions}
\label{S3.1}
\subsubsection{Notation and preliminaries}\label{S311}
We introduce some (more) notation on the $\GL_n$-side.
Recall  $T$ is the torus subgroup of $\GL_n$, and $B\supset T$ is the Borel subgroup of upper triangular matrices. For a standard parabolic subgroup $P\supset B$ of  $\GL_n$, let $L_P\supset T$ be its standard Levi subgroup and $P^-$ its opposite parabolic subgroup.  Denote by $\ft\subset \ub \subset \fp \subset \gl_n$ the corresponding Lie algebras over $K$. 
Let $\theta:=(0, \cdots,1-i, \cdots,  1-n)$. For a parabolic subgroup $P$, let $n_i\in \Z_{\geq 1}$ such that  the simple roots of $L_P$ are given by $\{1,\cdots, n-1\}\setminus \{n_1, n_2+n_1, \cdots, n_1+\cdots +n_{r-1}\}$ (so $L_P\cong \GL_{n_1} \times \GL_{n_2} \times \cdots \GL_{n_r}$).  Let $\theta^P:=(\underbrace{0,\cdots, 0}_{n_1}, \underbrace{-n_1, \cdots, -n_1}_{n_2}, \cdots, \underbrace{-(n_1+\cdots+n_{r-1}), \cdots, -(n_1+\cdots+n_{r-1})}_{n_r})$ (so $\theta=\theta^B$), that we view as an algebraic character of $L_P$.  For simplicity,  for $i\in \{1,\cdots, n-1\}$, we denote by $P_i$ the associated maximal parabolic subgroup such that  its standard Levi subgroup $L_i\supset T$ has simple roots $\{1,\cdots, n-1\}\setminus \{i\}$. 

For a Lie algebra $\ug$ over $K$, denote by $\ug_{\Sigma_K}:=\ug \otimes_{\Q_p} E \cong \prod_{\sigma\in \Sigma_K} \ug \otimes_{K,\sigma} E=:\prod_{\sigma\in \Sigma_K} \ug_{\sigma}$. For a weight $\mu$ of $\ft_{\Sigma_K}$, denote by $M^-(\mu):=\text{U}(\gl_{n,\Sigma_K})  \otimes_{\text{U}(\ub^-_{\Sigma_K})} \mu$, and let $L^-(\mu)$ be its unique simple quotient. If $\mu$ is anti-dominant (i.e. $\mu_{\sigma,1}<\mu_{\sigma,2}<\cdots<\mu_{\sigma,n}$ for all $\sigma\in \Sigma_K$, where $\mu=(\mu_{\sigma,i})_{\substack{\sigma\in \Sigma_K \\ i=1, \cdots, n}}$), then $L^-(\mu)$ is finite dimensional and isomorphic to the dual $L(-\mu)^{\vee}$, where $L(-\mu)$ is the algebraic representation of $\Res^K_{\Q_p} \GL_n$ of highest weight $-\mu$ with respect to $\Res^K_{\Q_p} B$.

For  an admissible locally $\Q_p$-analytic representation $V$ of $\GL_n(K)$, by \cite{ST03}, its continuous dual $V^{\vee}$ is naturally a module over the ($\Q_p$-analytic) distribution algebra $\cD(\GL_n(K),E)$, which, equipped with the strong topology, is a coadmissible module over $\cD(H,E)$ for a(ny) compact open subgroup $H$ of $\GL_n(K)$. For admissible locally $\Q_p$-analytic representations $V_1$, $V_2$ of $\GL_n(K)$, set $\Ext^i_{\GL_n(K)}(V_1,V_2):=\Ext^i_{\cD(\GL_n(K),E)}(V_2^{\vee},V_1^{\vee})$, where the latter is defined in the abelian category of abstract $\cD(\GL_n(K),E)$-modules. By \cite[Lem.~2.1.1]{Br16}, $\Ext^1_{\GL_n(K)}(V_1,V_2)$ is equal to the extension group of admissible locally $\Q_p$-analytic representations of $V_1$ by $V_2$. If $V_1$, $V_2$ are locally algebraic, set $\Ext^1_{\lalg}(V_1,V_2)$ to be the subgroup of locally algebraic extensions. Any representation  $\widetilde{V}$ in  $\Ext^1_{\GL_n(K)}(V,V)$  is equipped with a natural $E[\epsilon]/\epsilon^2$ structure where $\epsilon$ acts via $\widetilde{V} \twoheadrightarrow V \xrightarrow{\id} V \hookrightarrow \widetilde{V}$. 

Suppose $\Ext^1_{\GL_n(K)}(V_1,V_2)$ is finite dimensional over $E$. For a subspace $U\subset \Ext^1_{\GL_n(K)}(V_1,V_2)$, we can associate a tautological extension of $V_1 \otimes_E U$ by $V_2$ (for example see the discussion below Theorem \ref{Tintro2}). When $U= \Ext^1_{\GL_n(K)}(V_1,V_2)$, we call the corresponding extension the \textit{universal} extension of $V_1$ (or $V_1 \otimes_E \Ext^1_{\GL_n(K)}(V_1,V_2)$) by $V_2$.

Let $\phi=\phi_1 \boxtimes \cdots \boxtimes \phi_n: T(K)\ra E^{\times}$ be a smooth character. We call $\phi$ generic if $\phi_i \phi_j^{-1}\neq 1, |\cdot |_K$ for $i \neq j$. For $w\in S_n$, let $w(\phi):=\phi_{w^{-1}(1)}  \boxtimes \cdots \boxtimes \phi_{w^{-1}(n)}$. Let $\delta_B=|\cdot|_K^{n-1} \boxtimes \cdots \boxtimes |\cdot|_K^{n+1-2i} \boxtimes \cdots \boxtimes |\cdot|_K^{1-n}$ be the modulus character of $B(K)$ and $\eta:=1  \boxtimes |\cdot|_K \boxtimes \cdots \boxtimes |\cdot|_K^{n-1}=|\cdot|_K^{-1} \circ \theta$.  Let $I_{\sm}(\phi):=(\Ind_{B^-(K)}^{\GL_n(K)} \phi \eta)^{\infty}$, which is an absolutely irreducible smooth admissible representation of $\GL_n(K)$ when  $\phi$ is generic. Moreover, when $\phi$ is generic,  $I_{\sm}(\phi)\cong I_{\sm}(w(\phi))=:\pi_{\sm}(\phi)$ for all $w\in S_n$, which is  in fact the smooth representation of $\GL_n(K)$ corresponding to the Weil-Deligne representation $\oplus_{i=1}^n \phi_i$ in the classical local Langlands correspondence.

\subsubsection{Principal series}\label{Sps}We collect some facts on the locally $\Q_p$-analytic principal series of $\GL_n(K)$.

Let $\textbf{h}$  be a strictly dominant weight  of $\ft_{\Sigma_K}$, put $\lambda:=\textbf{h}-\theta^{[K:\Q_p]}=(\lambda_{i,\sigma}=h_{i,\sigma}+i-1)_{\substack{\sigma\in \Sigma_K\\ i=1,\cdots, n}}$, which is a dominant weight of $\ft$. Let $\phi$ be a generic smooth character of $T(K)$. Put $\pi_{\alg}(\phi, \textbf{h}):=\pi_{\sm}(\phi) \otimes_E L(\lambda)$ ($\cong I_{\sm}(w(\phi)) \otimes_E L(\lambda)$ for all $w\in S_n$), which is an irreducible  locally algebraic representation of $\GL_n(K)$. 
For $w\in S_n$, put $\PS(w(\phi), \textbf{h}):=(\Ind_{B^-(K)}^{\GL_n(K)} w(\phi)\eta z^{\lambda})^{\Q_p-\an}=\big(\Ind_{B^-(K)}^{\GL_n(K)}w(\phi) z^{\textbf{h}} (\varepsilon^{-1}\circ \theta)\big)^{\Q_p-\an}.$ We have (where $\cF_{B^-}^{\GL_n}(-,-)$ denotes Orlik-Strauch functor \cite{OS}):
\begin{proposition}\label{LGLn1}
Let  $w\in S_n$. 

(1) The irreducible constituents of $\PS(w(\phi), \textbf{h})$ are given by $\big\{\sC(w,u):=\cF_{B^-}^{\GL_n}(L^-(-u \cdot \lambda), w(\phi) \eta)\big\}_{u=(u_{\sigma})\in S_n^{|\Sigma_K|}},$ which are pairwisely distinct. Moreover, if $\lg(u)=1$, then $\sC(w,u)$ has multiplicity one.

(2) $\soc_{\GL_n(K)} \PS(w(\phi), \textbf{h})\cong I_{\sm}(w(\phi)) \otimes_E L(\lambda)\cong \pi_{\alg}(\phi, \textbf{h})$.

(3) $\soc_{\GL_n(K)}(\PS(w(\phi), \textbf{h})/\pi_{\alg}(\phi, \textbf{h}))\cong \oplus_{\substack{u\in S_n^{|\Sigma_K|}, \\ \lg(u)=1}} \sC(w,u)$.

(4) For $w'\in S_n$, and $u, u'\in S_n^{|\Sigma_K|}$ with $\lg(u)=\lg(u')=1$, $\sC(w,u)\cong \sC(w',u')$ if and only if $u=u'=s_{i,\sigma}$ for some $i \in \{1, \cdots, n-1\}$ and $\sigma\in \Sigma_K$,  and  $w(w')^{-1}$ lies in the Weyl group of $L_{P_i}$.
\end{proposition}
\begin{proof}
(1) and (4) follow from \cite[Thm.]{OS} (together with some standard facts on the constituents of the Verma module, see for example \cite[Chap.~6]{Hum08}). (2) (3) follow from \cite[Thm.~1]{Or18}. 
\end{proof}
For $i\in \{1,\cdots, n-1\}$, let $I \subset \{1,\cdots, n\}$ be a subset of cardinality $i$. By Proposition \ref{LGLn1} (4), all the representations $\sC(w,s_{i,\sigma})$ with $w(\{1,\cdots, i\})=I$ are isomorphic, which we denote by $\sC(I,s_{i,\sigma})$. Moreover, $\sC(I,s_{i,\sigma})$ are pairwisely distinct for different $s_{i,\sigma}$ or $I$. For $w\in S_n$ with $w(\{1,\cdots, i\})=I$,  we have (by \cite[Thm.~1]{Or18})
\begin{equation}\label{ECIsigma}
\sC(I,s_{i,\sigma})\cong \soc_{\GL_n(K)} (\Ind_{B^-(K)}^{\GL_n(K)} z^{-s_{i,\sigma} \cdot \lambda} w(\phi) \eta)^{\Q_p-\an}.
\end{equation}
\begin{lemma}\label{Ljacsi} Let $w\in S_n$ such that $w(\{1,\cdots, i\})=I$.

(1)We have  $\Hom_{T(\Q_p)}(z^{-s_{i,\sigma} \cdot \lambda} w(\phi) \eta\delta_B, J_B(\sC(I,s_{i,\sigma})))\cong E$, where $J_B(-)$ denotes the Jacquet-Emerton functor for $B$ (cf. \cite{Em11}).

(2) We have $I_{B^-(K)}^{\GL_n(K)} (z^{-s_{i,\sigma} \cdot \lambda} w(\phi) \eta)\cong \sC(I,s_{i,\sigma})$, where  $I_{B^-(K)}^{\GL_n(K)}(-)$ is  Emerton's induction functor \cite[\S~(2.8)]{Em2}. 
\end{lemma} 
\begin{proof}By \cite[Thm.]{OS}, it is easy to see any irreducible constituent of $\PS(w(\phi),\textbf{h})$ is a subrepresentation of a certain locally $\Q_p$-analytic principal series, hence is very strongly admissible by \cite[Prop.~2.1.2]{Em2}. 
(1) then follows by \cite[Thm.~4.3, Rem.~4.4 (i)]{Br13II}. By \textit{loc. cit.} and \cite[Thm.~1]{Or18}, $\Hom_{T(\Q_p)}(z^{-s_{i,\sigma} \cdot \lambda} w(\phi) \eta \delta_B, J_B(\sC))=0$ for any irreducible constituent $\sC$ of $(\Ind_{B^-(K)}^{\GL_n(K)} z^{-s_{i,\sigma} \cdot \lambda} w(\phi) \eta)^{\Q_p-\an}$ with $\sC \neq \sC(I,s_{i, \sigma})$. The natural map
$z^{-s_{i,\sigma}\cdot \lambda} w(\phi) \eta \delta_B \hookrightarrow J_B\big((\Ind_{B^-(K)}^{\GL_n(K)} z^{-s_{i,\sigma} \cdot \lambda} w(\phi) \eta)^{\Q_p-\an}\big)$
hence has image contained in $J_B(\sC(I,s_{i,\sigma}))$. By definition of $I_{B^-(K)}^{\GL_n(K)}(-)$ (cf. \cite[\S~(2.8)]{Em11}), (2) follows. 
\end{proof}
Let $\PS_1(w(\phi), \textbf{h})$ be the unique subrepresentation of $\PS(w(\phi), \textbf{h})$ of socle $I_{\sm}(w(\phi)) \otimes_E L(\lambda)$ and cosocle $\oplus_{\substack{i=1,\cdots, n-1\\ \sigma\in \Sigma_K}} \sC(w, s_{i,\sigma})$ (with the tautological injection $\PS_1(w(\phi), \textbf{h})\hookrightarrow \PS(w(\phi), \textbf{h})$). Throughout the section, we fix isomorphisms\begin{equation} \label{Eintisom}\pi_{\alg}(\phi, \textbf{h})\cong  L(\lambda)\otimes_E I_{\sm}(w(\phi)) \big(\hooklongrightarrow \PS_1(w(\phi), \textbf{h})\big).\end{equation} for all $w\in S_n$.
The amalgamated sum $\oplus_{\pi_{\alg}(\phi ,\lambda)}^{w\in S_n} \PS_1(w(\phi), \textbf{h})$ admits a unique quotient, denoted by $\pi_1(\phi, \textbf{h})$ of socle $\pi_{\alg}(\phi, \textbf{h})$. By Lemma \ref{LGLn1} (3) (4), $\pi_1(\phi, \textbf{h})$ is given by an extension of $\oplus_{\substack{i=1,\cdots, n-1, \sigma\in \Sigma_K\\ I \subset \{1,\cdots, n\}, \# I =i}} \sC(I,s_{i,\sigma})$ ($(2^n-2)d_K$ constituents in total) by $\pi_{\alg}(\phi, \textbf{h})$. Note we have a tautological injection \begin{equation}\label{Etautoinj}\pi_{\alg}(\phi, \textbf{h}) \hooklongrightarrow \pi_1(\phi, \textbf{h}).\end{equation}  We study the extension group of 
$\pi_{\alg}(\phi, \textbf{h})$ by $\pi_1(\phi, \textbf{h})$. 
\begin{proposition}\label{Pextalg}
(1)  For $w\in S_n$ and  $\psi  \in \Hom_{g'}(T(K),E)$ (cf. (\ref{EHomg'T})), we have $ I_{B^-(K)}^{\GL_n(K)} (w(\phi) \eta z^{\lambda}(1+\psi \epsilon))\in  \Ext^1_{\GL_n(K)}(\pi_{\alg}(\phi, \textbf{h}), \pi_{\alg}(\phi, \textbf{h})\big)$ (using (\ref{Eintisom})). Moreover, the following map is a bijection:
\begin{eqnarray}\label{Ezetaw0}
	\zeta_{w}:	\Hom_{g'}(T(K),E) &\xlongrightarrow{\sim} & \Ext^1_{\GL_n(K)}(\pi_{\alg}(\phi, \textbf{h}), \pi_{\alg}(\phi, \textbf{h})\big), \\ \psi &\mapsto& I_{B^-(K)}^{\GL_n(K)} (w(\phi) \eta z^{\lambda}(1+\psi \epsilon)), \nonumber
\end{eqnarray}
and induces $\Hom_{\sm}(T(K),E) \xrightarrow{\sim}  \Ext^1_{\lalg}\big(\pi_{\alg}(\phi, \textbf{h}), \pi_{\alg}(\phi, \textbf{h})\big)$.
In particular, $\dim_E \Ext^1_{\lalg}\big(\pi_{\alg}(\phi, \textbf{h}), \pi_{\alg}(\phi, \textbf{h})\big)=n$, $\dim_E  \Ext^1_{\GL_n(K)}(\pi_{\alg}(\phi, \textbf{h}), \pi_{\alg}(\phi, \textbf{h})\big)=n+d_K.$

(2) For $w_1, w_2\in S_n$, the following diagram commutes:
\begin{equation}\label{Ecomg'int}
	\begin{CD}\Hom_{g'}(T(K),E) @> \zeta_{w_1} > \sim> \Ext^1_{\GL_n(K)}(\pi_{\alg}(\phi, \textbf{h}), \pi_{\alg}(\phi, \textbf{h})\big) \\
		@V w_2w_1^{-1} V \sim V @| \\
		\Hom_{g'}(T(K),E) @> \zeta_{w_2} > \sim> \Ext^1_{\GL_n(K)}(\pi_{\alg}(\phi, \textbf{h}), \pi_{\alg}(\phi, \textbf{h})\big).
	\end{CD}
\end{equation}
\end{proposition}
\begin{proof}For $\psi=\psi_1+\psi_0 \circ \dett$ with $\psi_1\in \Hom_{\sm}(T(K),E)$ and $\psi_0\in \Hom(K^{\times}, E)$, it is easy to see the natural map
\begin{multline*}
	w(\phi) \eta z^{\lambda}(1+\psi \epsilon) \delta_B 
	 \hooklongrightarrow J_B\big((\Ind_{B^-(K)}^{\GL_n(K)} w(\phi) \eta z^{\lambda} (1+\psi \epsilon))^{\Q_p-\an}\big) \\ \hooklongrightarrow (\Ind_{B^-(K)}^{\GL_n(K)} w(\phi) z^{\lambda}\eta (1+\psi \epsilon))^{\Q_p-\an}
\end{multline*}
factors through the subrepresentation $(\Ind_{B^-(K)}^{\GL_n(K)} w(\phi) \eta (1+\psi_1 \epsilon))^{\sm} \otimes_E L(\lambda) \otimes_{E[\epsilon]/\epsilon^2} (1+\epsilon \psi_0 \circ \dett)$. By definition (\cite[\S~(2.8)]{Em2}), we see
\begin{equation*}
	I_{B^-(K)}^{\GL_n(K)} (w(\phi) \eta z^{\lambda}(1+\psi \epsilon))\cong (\Ind_{B^-(K)}^{\GL_n(K)} w(\phi) \eta (1+\psi_1 \epsilon))^{\sm} \otimes_E L(\lambda) \otimes_{E[\epsilon]/\epsilon^2} (1+\psi_0 \circ \dett).
\end{equation*}
Together with the isomorphism $I_{\sm}(w(\phi)) \otimes_E L(\lambda)\cong \pi_{\alg}(\phi, \textbf{h})$ (\ref{Eintisom}), it gives a well-defined element in $\Ext^1_{\GL_n(K)}(\pi_{\alg}(\phi, \textbf{h}), \pi_{\alg}(\phi, \textbf{h}))$.
By \cite[Prop.~4.7]{Sch11}, we have (where $Z\subset \GL_n$ denotes the centre, and the subscript ``$Z$" stands for fixing central character)
\begin{equation}\label{Eautolalg}\Ext^1_{\lalg, Z}(\pi_{\alg}(\phi, \textbf{h}), \pi_{\alg}(\phi, \textbf{h})) \xlongrightarrow{\sim} \Ext^1_{\GL_n(K),Z}(\pi_{\alg}(\phi, \textbf{h}), \pi_{\alg}(\phi, \textbf{h})).\end{equation} 
By classical smooth representation theory, the restriction of $\zeta_w$  induces an isomorphism $\Hom_{\sm}(T(K)/Z(K),E) \xrightarrow{\sim} \Ext^1_{\lalg, Z}(\pi_{\alg}(\phi, \textbf{h}), \pi_{\alg}(\phi, \textbf{h}))$ (so the latter has dimension $n-1$).
Using similar arguments as in \cite[Lem.~3.16]{BD1} (and the aforementioned discussion), we obtain a commutative diagram of short exact sequences (we omit $\GL_n(K)$, $(\phi, \textbf{h})$)
\begin{equation*}
	\begin{tikzcd}
		\Hom_{\mathrm{sm}}(T(K)/Z(K),E) \arrow[r, hook] \arrow[d, "\sim"'] 
		& \Hom_{g'}(T(K),E) \arrow[r, two heads] \arrow[d, "\zeta_w"] 
		& \Hom(Z(K),E) \arrow[d, equal] \\
		\Ext^1_Z(\pi_{\mathrm{alg}}, \pi_{\mathrm{alg}}) \arrow[r, hook] 
		& \Ext^1(\pi_{\mathrm{alg}}, \pi_{\mathrm{alg}}) \arrow[r, two heads] 
		& \Hom(Z(K),E)
	\end{tikzcd}
\end{equation*}
So $\zeta_w$ is a bijection and $\dim_E \Ext^1_{\GL_n(K)}(\pi_{\alg}(\phi, \textbf{h}), \pi_{\alg}(\phi, \textbf{h}))=n+d_K.$ For (2), it suffices to prove the statement for $g'$ replaced by ``$\sm$". This is a classical fact. Indeed, let $\cH$ (resp. $\cH_i\cong \bG_m$) be the Bernstein centre over $E$ associated to the smooth representation $\pi_{\sm}(\phi)$ of $\GL_{n}(K)$ (resp. $\phi_i$ of $K^{\times}$) (cf. \cite[\S~3.13]{CEGGPS1}). By \cite[Lem.~3.22]{CEGGPS1}, for each $w\in S_n$, there is a natural map $\cJ_w: \prod_{i=1}^n \Spec \cH_{w^{-1}(i)} \ra \Spec \cH$ sending a point $(\phi_i')\in \prod_{i=1}^n \Spec \cH_{w^{-1}(i)}$ to the point (associated to) $(\Ind_{B^-(K)}^{\GL_n(K)} (\boxtimes_{i=1}^r \phi_i')\eta)^{\sm}$ of $\cH$.  Moreover, the tangent map of $\cJ_w$ at $(\phi_{w^{-1}(i)})$ coincides with $\zeta_w$. The intertwining property implies that for $w_1, w_2\in S_n$, $\cJ_{w_2}=(w_2w_1^{-1})\circ \cJ_{w_1}$ where $w_2w_1^{-1}$ here denotes the morphism $\prod_{i=1}^n \Spec \cH_{w_1^{-1}(i)}\ra \prod_{i=1}^n \Spec \cH_{w_2^{-1}(i)}$, $(\phi_i')\mapsto (\phi_{(w_2w_1^{-1})^{-1}(i)}')$. By considering the corresponding tangent maps of $\cJ_{w_1}$, $\cJ_{w_2}$, we deduce the commutativity of (\ref{Ecomg'int}) (with $g'$ replaced by ``$\sm$"). This concludes the proof.
\end{proof}
\begin{remark}
	Note that $\zeta_w$ is in fact  independent of the choice of (\ref{Eintisom}).
\end{remark}

\begin{lemma}\label{Lextps}
For any $\sC(I,s_{i,\sigma})$, we have: 

\noindent (1) $\dim_E \Ext^1_{\GL_n}(\sC(I,s_{i,\sigma}), \pi_{\alg}(\phi, \textbf{h}))=\dim_E \Ext^1_{\GL_n}(\pi_{\alg}(\phi, \textbf{h}), \sC(I,s_{i,\sigma}))=1.$

\noindent (2) Let $\tilde{\pi}_{\alg}(\phi, \textbf{h})\in \Ext^1_{\GL_n(K)}(\pi_{\alg}(\phi, \textbf{h}), \pi_{\alg}(\phi, \textbf{h}))$ be non-split, then the following pull-back map (via $\tilde{\pi}_{\alg}(\phi, \textbf{h}) \twoheadrightarrow \pi_{\alg}(\phi, \textbf{h})$) is a bijection:
\begin{equation}\label{Etildealg}
	\Ext^1_{\GL_n(K)}(\pi_{\alg}(\phi, \textbf{h}), \sC(I,s_{i,\sigma}))  \xlongrightarrow{\sim}	\Ext^1_{\GL_n(K)}(\tilde{\pi}_{\alg}(\phi, \textbf{h}), \sC(I,s_{i,\sigma})).
\end{equation}
\end{lemma}
\begin{proof}(1) follows from \cite[Prop.~5.1.14]{BQ} together with \cite[Lem.~3.2.4 (ii)]{BQ} (when $K=\Q_p$, the part on  $\Ext^1_{\GL_n}(\sC(I,s_{i,\sigma}), \pi_{\alg}(\phi, \textbf{h}))$ was proved in \cite[Cor.~5.9]{BD2}).  We give a proof of (2) and an alternative proof of the second equality in (1) using Schraen's spectral sequence \cite[Cor.~4.9]{Sch11} (for $G=\Res^K_{\Q_p} \GL_n$). First, note by the same argument below \cite[Cor.~4.9]{Sch11}, the separatedness assumption in \cite[Cor.~4.9]{Sch11} is satisfied for either $\pi_{\alg}(\phi, \textbf{h})$  or any $\widetilde{\pi}_{\alg}(\phi, \textbf{h})$ in (2) (noting by Proposition \ref{Pextalg} (1) and the proof, $\widetilde{\pi}_{\alg}(\phi, \textbf{h})|_{\SL_n(K)}$ is locally algebraic). Let $\delta:= z^{- s_{i,\sigma} \cdot \lambda} w(\phi) \eta$. By 
	\cite[Cor.~4.9]{Sch11}, we have a spectral sequence \begin{equation}\label{Espsq1}\Ext^p_{T(K)}(H_q(N^-(K),\pi_{\alg}(\phi, \textbf{h})), \delta)\Rightarrow \Ext^{p+q}_{\GL_n(K)}\big(\pi_{\alg}(\phi, \textbf{h}), (\Ind_{B^-(K)}^{\GL_n(K)}\delta)^{\Q_p-\an}\big)\end{equation} where $N^-$ is the unipotent radical of $B^-$. Recall for characters $\chi$, $\chi'$ of $T(K)$ over $E$, we have $\Ext^i_{T(K)}(\chi,\chi')=0$ for all $i$ if $\chi\neq \chi'$. This, together with \cite[(4.40), (4.41), (4.42)]{Sch11} and the classical fact $J_{N^-}(I_{\sm}(\phi))\cong \oplus_{w'\in S_n} w'(\phi) \eta$ (where $J_{N^-}(-)$ denotes the classical Jacquet module for $N^-$), imply that for $p+q=1$, the only non-zero term on the left hand side of (\ref{Espsq1}) is $\Hom_{T(K)}(H_1(N^-(K), \pi_{\alg}(\phi, \textbf{h})), \delta) \cong \Hom_{T(K)}(\delta, \delta)\cong E$. So $ \Ext^1_{\GL_n(K)}\big(\pi_{\alg}(\phi, \textbf{h}), (\Ind_{B^-(K)}^{\GL_n(K)}\delta)^{\Q_p-\an}\big)\cong E$. By   (\ref{ECIsigma}) and \ \cite[Lem.~2.26]{Ding7}, \ the \ natural \ push-forward \ map \ is \ an \ isomorphism: \ $\Ext^1_{\GL_n(K)}(\pi_{\alg}(\phi, \textbf{h}), \sC(I,s_{i,\sigma})) \xrightarrow{\sim}\Ext^1_{\GL_n(K)}\big(\pi_{\alg}(\phi, \textbf{h}), (\Ind_{B^-(K)}^{\GL_n(K)}\delta)^{\Q_p-\an}\big)$. This proves the second equality in (1).
	Similarly with $\pi_{\alg}(\phi, \textbf{h})$ replaced by $\tilde{\pi}_{\alg}(\phi, \textbf{h})$, we get
	\begin{multline}\label{EisomExtGT}
		\Ext^1_{\GL_n(K)}(\tilde{\pi}_{\alg}(\phi, \textbf{h}), \sC(I,s_{i,\sigma}))\\
		\cong \Ext^1_{\GL_n(K)}\big(\tilde{\pi}_{\alg}(\phi, \textbf{h}), (\Ind_{B^-(K)}^{\GL_n(K)}z^{- s_{i,\sigma} \cdot \lambda} w(\phi) \eta)^{\Q_p-\an}\big)\\
		\cong \Hom_{T(K)}\big(w(\phi) \eta z^{-s_{i,\sigma} \cdot \lambda}(1+\psi \epsilon), w(\phi) \eta z^{-s_{i,\sigma} \cdot \lambda} \big)
	\end{multline}
	where $\psi=\zeta_w^{-1}\big(\tilde{\pi}_{\alg}(\phi, \textbf{h})\big)\in \Hom_{g'}(T(K),E)$. As $\widetilde{\pi}_{\alg}(\phi, \textbf{h})$ is non-split, $\psi\neq 0$ hence the right hand side of (\ref{EisomExtGT}) is one dimensional over $E$. However, by an easy d\'evissage, (\ref{Etildealg}) is injective hence has to be bijective.
\end{proof}
For $w\in S_n$, consider the  natural map
\begin{equation*}
\Hom(T(K),E) \lra\Ext^1_{\GL_n(K)}\big(\PS(w(\phi), \textbf{h}), \PS(w(\phi), \textbf{h})\big)
\end{equation*}
sending $\psi $ to $(\Ind_{B^-(K)}^{\GL_n(K)} w(\phi) \eta z^{\lambda} (1+\psi \epsilon))^{\Q_p-\an}$.
Composed with the pull-back map for (\ref{Eintisom}) and using \cite[Lem.~2.26]{Ding7},  it induces:
\begin{equation}\label{Eind1}
\Hom(T(K),E) \lra \Ext^1_{\GL_n(K)}(\pi_{\alg}(\phi, \textbf{h}), \PS_1(w(\phi), \textbf{h})).
\end{equation}
Composed furthermore with the push-forward map via the injection $\PS_1(w(\phi), \textbf{h})\hookrightarrow \pi_1(\phi, \textbf{h})$ (associated to (\ref{Eintisom}), see also (\ref{Etautoinj})), we finally obtain a map
\begin{equation}\label{Ezetaw}
\zeta_w: \Hom(T(K),E) \lra \Ext^1_{\GL_n(K)}(\pi_{\alg}(\phi, \textbf{h}), \pi_1(\phi, \textbf{h})).
\end{equation}
Note that the map $\zeta_w$ does not depend on the choice of (\ref{Eintisom}).
\begin{proposition}\label{Pextps}

\noindent (1)  For $w\in S_n$, the map (\ref{Eind1}) is bijective.  In particular, we have $\dim_E \Ext^1_{\GL_n(K)}(\pi_{\alg}(\phi, \textbf{h}), \PS_1(w(\phi), \textbf{h}))=n+nd_K$.

\noindent (2) For $w\in S_n$, $\zeta_w|_{\Hom_{g'}(T(K),E)}$ is equal to the composition of (\ref{Ezetaw0}) with the  push-forward  map $\Ext^1_{\GL_n(K)}(\pi_{\alg}(\phi, \textbf{h}), \pi_{\alg}(\phi, \textbf{h})) \hookrightarrow \Ext^1_{\GL_n(K)}(\pi_{\alg}(\phi, \textbf{h}), \pi_1(\phi, \textbf{h})).$

\end{proposition}
\begin{proof}(1) follows from similar arguments as  in the proof of Lemma \ref{Lextps}, using Schraen's spectral sequence \cite[Cor.~4.9]{Sch11} and \cite[Lem.~2.26]{Ding6}. We leave the details to the reader. (2) is clear (see also Remark \ref{Rind1} below). 
\end{proof}
\begin{remark}\label{Rind1}
The map $\zeta_w$ can also be obtained by using Emerton's functor $I_{B^-(K)}^{\GL_n(K)} (-)$. In fact, by definition (cf. \cite[\S~(2.8)]{Em2}) and using \cite[Lem.~2.26]{Ding7}, it is straightforward  to see for $\psi \in \Hom(T(K),E)$, $I_{B^-(K)}^{\GL_n(K)} (w(\phi) \eta z^{\lambda} (1+\psi \epsilon))\subset (\Ind_{B^-(K)}^{\GL_n(K)} w(\phi) \eta z^{\lambda} (1+\psi \epsilon))^{\Q_p-\an}$ is an extension of $\pi_{\alg}(\phi, \textbf{h})\cong I_{\sm}(w(\phi))\otimes_E L(\lambda)$ by a certain subrepresentation $V$ of $\PS_1(w(\phi), \textbf{h})$. Then $\zeta_w(\psi)$ is just its image of  the push-forward map via $V\hookrightarrow \PS_1(w(\phi),\textbf{h}) \hookrightarrow \pi_1(\phi, \textbf{h})$.
\end{remark}
\begin{proposition}\label{Pextpi1}(1) We have an exact sequence
\begin{multline}\label{Edivi}
	0 \lra \Ext^1_{\GL_n(K)}(\pi_{\alg}(\phi, \textbf{h}), \pi_{\alg}(\phi, \textbf{h})) \lra \Ext^1_{\GL_n(K)}(\pi_{\alg}(\phi, \textbf{h}), \pi_1(\phi, \textbf{h})) \\ \lra \oplus_{\substack{i=1,\cdots, n-1, \sigma\in \Sigma_K\\ I \subset \{1,\cdots, n-1\}, \# I =i}} \Ext^1_{\GL_n(K)}(\pi_{\alg}(\phi, \textbf{h}), \sC(I,s_{i,\sigma})) \lra 0.
\end{multline}
In particular, $\dim_E \Ext^1_{\GL_n(K)} (\pi_{\alg}(\phi, \textbf{h}), \pi_1(\phi, \textbf{h}))=n+(2^n-1)d_K$.

(2) The following map is surjective:
\begin{equation}\label{Etphih}
	t_{\phi, \textbf{h}}:	\oplus_{w\in S_n} 
	\Hom(T(K),E) \xlongrightarrow{(\zeta_w)} \Ext^1_{\GL_n(K)}(\pi_{\alg}(\phi, \textbf{h}), \pi_1(\phi, \textbf{h})).
\end{equation}
\end{proposition}
\begin{proof}We omit the subscript ``$\GL_n(K)$" in the proof. The sequence follows by d\'evissage,  and it suffices to prove the second last map in (\ref{Edivi}) is surjective. For $w\in S_n$, using d\'evissage, we have an exact sequence
\begin{multline}\label{Edivips}
	0 \ra \Ext^1(\pi_{\alg}(\phi, \textbf{h}), \pi_{\alg}(\phi, \textbf{h})) \ra \Ext^1(\pi_{\alg}(\phi, \textbf{h}), \PS_1(w(\phi), \textbf{h})) \\ \ra \oplus_{\substack{i=1,\cdots, n-1\\ \sigma\in \Sigma_K}} \Ext^1(\pi_{\alg}(\phi, \textbf{h}), \sC(w,s_{i,\sigma})).
\end{multline}
By comparing dimensions (using Proposition \ref{Pextalg} (1), Proposition \ref{Pextps} (1) and Lemma \ref{Lextps} (1)), the last map in (\ref{Edivips}) is surjective. The following diagram clearly commutes
\begin{equation}\label{Epspi1}\small
	\begin{tikzcd}
	&	\Ext^1(\pi_{\alg}(\phi, \textbf{h}), \PS_1(w(\phi), \textbf{h})) \arrow[r, two heads] \arrow[d, hook] &\oplus_{\substack{i=1,\cdots, n-1\\ \sigma\in \Sigma_K}} \Ext^1(\pi_{\alg}(\phi, \textbf{h}), \sC(w,s_{i,\sigma})) \arrow[d, hook] \\
	&	\Ext^1(\pi_{\alg}(\phi, \textbf{h}), \pi_1(\phi, \textbf{h})) \arrow[r] & \oplus_{\substack{i=1,\cdots, n-1, \sigma\in \Sigma_K\\ I \subset \{1,\cdots, n-1\}, \# I =i}} \Ext^1(\pi_{\alg}(\phi, \textbf{h}), \sC(I,s_{i,\sigma})).
	\end{tikzcd}
\end{equation}Varying $w$, the image of the right vertical map  can ``cover" the target.
Together with  the surjectivity of the top map, we see the bottom map is also  surjective. (2) follows by the first statement in Proposition \ref{Pextps} (1) and (\ref{Edivi}).
And the dimension part in (1) follows then from Lemma \ref{Lextps} (1) and Proposition \ref{Pextps} (1).  
\end{proof}
\begin{remark}
By Proposition \ref{Pextps} (1) and \cite[Lem.~2.26]{Ding7} (and using (\ref{Eintisom})), for $w\in S_n$, we have 
$	\zeta_w: \Hom(T(K),E) \xrightarrow{\sim} \Ext^1_{\GL_n(K)}(\pi_{\alg}(\phi, \textbf{h}), \PS(w(\phi), \textbf{h}))$.
Denote by $\pi(\phi, \textbf{h})$ the unique quotient of $\oplus_{\pi_{\alg}(\phi, \textbf{h})}^{w\in S_n} \PS(w(\phi), \textbf{h})$ of socle $\pi_{\alg}(\phi, \textbf{h})$ (cf. \cite[Def.~5.7]{BH2}, which is the representation $\pi(\rho)^{\fss}$ of \textit{loc. cit.}). The representation $\pi_1(\phi, \textbf{h})$ is in fact the first two layers in the socle filtration of $\pi(\phi, \textbf{h})$. Moreover, using again \cite[Lem.~2.26]{Ding7}, we have 
\begin{equation}\label{Epi1pi}
	\Ext^1_{\GL_n(K)}(\pi_{\alg}(\phi, \textbf{h}), \pi_1(\phi, \textbf{h})) \xlongrightarrow{\sim} \Ext^1_{\GL_n(K)}(\pi_{\alg}(\phi, \textbf{h}), \pi(\phi, \textbf{h})).
\end{equation} 
Proposition \ref{Pextpi1} (2) hence holds with $\pi_1(\phi, \textbf{h})$ replaced by $\pi(\phi, \textbf{h})$.
\end{remark}
Denote by 
\begin{equation*}
\Ext^1_{g}(\pi_{\alg}(\phi, \textbf{h}), \pi_1(\phi, \textbf{h})) \subset \Ext^1_{g'}(\pi_{\alg}(\phi, \textbf{h}), \pi_1(\phi,\textbf{h})) \subset \Ext^1_w(\pi_{\alg}(\phi, \textbf{h}), \pi_1(\phi, \textbf{h}))
\end{equation*}
the respective image of  $\Ext^1_{\lalg}(\pi_{\alg}(\phi, \textbf{h}), \pi_{\alg}(\phi, \textbf{h}))$, $\Ext^1_{\GL_n(K)}(\pi_{\alg}(\phi, \textbf{h}), \pi_{\alg}(\phi, \textbf{h}))$, and $\Ima(\zeta_w)$ for $w\in S_n$. We also use the notation $\Ext^1_{\sT_w}$ for $\Ext^1_{w}$ whenever it is convenient for the context where $\sT_w$ is the $B$-filtration of $\oplus_{i=1}^n \phi_i$ associated to $w$. So $\zeta_w$ (\ref{Ezetaw}) induces an isomorphism
\begin{equation}\label{Eisomzetaw}
	\zeta_w: \Hom(T(K),E) \xlongrightarrow{\sim} \Ext^1_{w}(\pi_{\alg}(\phi, \textbf{h}), \pi_1(\phi, \textbf{h})).
\end{equation}
By Proposition \ref{Pextalg} (2),  for $w_1, w_2\in S_n$, the following diagram commutes
\begin{equation}\label{Eintg'}
\begin{CD}
	\Hom_{g'}(T(K),E)@> \zeta_{w_1}> \sim > \Ext^1_{g'}(\pi_{\alg}(\phi, \textbf{h}), \pi_1(\phi, \textbf{h})) \\
	@V w_2w_1^{-1} V \sim V @| \\
	\Hom_{g'}(T(K),E)@> \zeta_{w_2} > \sim > \Ext^1_{g'}(\pi_{\alg}(\phi, \textbf{h}), \pi_1(\phi, \textbf{h})).
\end{CD}
\end{equation}
\subsubsection{Parabolic inductions}\label{S313}
Let $P\supset B$ be a standard parabolic subgroup of $\GL_n$ with $L_P=\diag(\GL_{n_1}, \cdots, \GL_{n_r})$. Let $\sW_P$ be the Weyl group of $L_P$. Let $\sF_P$ be a $P$-filtration of $\oplus_{i=1}^n \phi_i$ and $\phi_{\sF_P,i}:=\otimes \phi_j$ for $\phi_j\in \gr_i \sF_P$ (where the order of these $\phi_j$ does not matter here). For $i=1, \cdots, r$, let $\textbf{h}^i:=(\textbf{h}_{n_1+\cdots n_{i-1}+1}, \cdots, \textbf{h}_{n_1+\cdots+n_i})$, $\lambda^i=(h_{n_1+\cdots + n_{i-1}+1,\sigma}, \cdots, h_{n_1+\cdots+n_i,\sigma}+n_{i}-1)_{\sigma\in \Sigma_K}$. \ Applying \ the \ constructions \ in \ \S~\ref{Sps} \ to \ $(\phi_{\sF_{P,i}}, \textbf{h}^i)$, \ we \ obtain \ $\GL_{n_i}(K)$-representations $\pi_{\alg}(\phi_{\sF_P,i}, \textbf{h}^i)$, $\pi_1(\phi_{\sF_P,i}, \textbf{h}^i)$ etc. Note when $n_i=1$, we have $\pi_{\alg}(\phi_{\sF_P,i}, \textbf{h}^i)=\pi_1(\phi_{\sF_P,i}, \textbf{h}^i)=\phi_{n_1+\cdots+n_{i-1}}z^{\textbf{h}_{n_1+\cdots+n_{i-1}}}$. We fix an isomorphism $\big(\Ind_{P^-(K)}^{\GL_n(K)} (\boxtimes_{i=1}^r \pi_{\alg}(\phi_{\sF_P,i}, \textbf{h}^i))\varepsilon^{-1} \circ \theta^P\big)^{\lalg}\cong \pi_{\alg}(\phi, \textbf{h})$ (where the supscript ``$\lalg$" means locally algebraic induction) and  (then) fix isomorphisms $\pi_{\alg}(\phi_{\sF_P,i},\textbf{h}^i)\cong I_{\sm}(w_i(\phi_{\sF_p,i}))\otimes_E L(\lambda^i)$ for all $i$ and $w_i\in S_{n_i}$  such that the composition (noting the first isomorphism is obtained by the transitivity of parabolic induction) $I_{\sm}(w(\phi))\otimes_E L(\lambda) \cong \big(\Ind_{P^-(K)}^{\GL_n(K)} (\boxtimes_{i=1}^r  I_{\sm}(w_i(\phi_{\sF_p,i}))\otimes_E L(\lambda^i)) \varepsilon^{-1} \circ \theta^P\big)^{\lalg}\cong  \pi_{\alg}(\phi, \textbf{h})$ coincides the fixed isomorphism (\ref{Eintisom}) for all $w\in \sW_P$. Consider the parabolic induction 
\begin{multline}\label{Eparaind}
(\Ind_{P^-(K)}^{\GL_n(K)} (\widehat{\boxtimes}_{i=1}^r \pi_1 (\phi_{\sF_P,i}, \textbf{h}^i)) \varepsilon^{-1}\circ \theta^P)^{\Q_p-\an}\\
\hooklongleftarrow (\Ind_{P^-(K)}^{\GL_n(K)} (\boxtimes_{i=1}^r \pi_{\alg}(\phi_{\sF_P,i}, \textbf{h}^i))\varepsilon^{-1} \circ \theta^P)^{\lalg} \cong \pi_{\alg}(\phi, \textbf{h}),
\end{multline}
where $\widehat{\boxtimes}$ denotes the completed (injective or equivalently projective) tensor product over $E$ (cf. \cite[Prop.~1.1.31]{Em04}).
\begin{lemma}\label{Lparaind}
For $i=1,\cdots, n-1$, $\sigma\in \Sigma_K$ and $I\subset \{1,\cdots, n\}$, $\# I=i$, $\sC(I,s_{i,\sigma})$ appears as an irreducible constituent of $\big(\Ind_{P^-(K)}^{\GL_n(K)} (\widehat{\boxtimes}_{i=1}^r \pi_1 (\phi_{\sF_P,i}, \textbf{h}^i)) \varepsilon^{-1}\circ \theta^P\big)^{\Q_p-\an}$ if and only if one of the following conditions holds:

(1) there exists $k\in \{1,\cdots, r\}$ such that $(n_1+\cdots+n_{k-1})+1\leq i\leq (n_1+\cdots+n_k)-1$ and $
\{j\ |\ \phi_j\in  \Fil_{\sF_P,k-1}\} \subset I\subset \{j\ | \ \phi_j \in \Fil_{\sF_P, k}\}$,

(2) $i=n_1+\cdots+n_k$ for some $k=1, \cdots, r-1$, and $I=\{j \ |\ \phi_j \in \Fil_k \sF_P\}$. 

\noindent Moreover, each of such constituents has multiplicity one, and lies in the socle of 
\begin{equation}\label{Eparaind3}\big(\Ind_{P^-(K)}^{\GL_n(K)} (\widehat{\boxtimes}_{i=1}^r \pi_1 (\phi_{\sF_P,i}, \textbf{h}^i)) \varepsilon^{-1}\circ \theta^P\big)^{\Q_p-\an}/\pi_{\alg}(\phi, \textbf{h}).\end{equation}
\end{lemma}
\begin{proof}
Let $V_0:=(\Ind_{P^-(K)}^{\GL_n(K)} (\boxtimes_{i=1}^r \pi_{\alg}(\phi_{\sF_P,i}, \textbf{h}^i)) \varepsilon^{-1}\circ \theta^P)^{\Q_p-\an} $. For all $i$ as in (2), $L^-(-s_{i,\sigma}\cdot \lambda)$ has multiplicity one in the parabolic Verma module $\text{U}(\gl_{n,\Sigma_K}) \otimes_{\text{U}(\fp_{\Sigma_K})}  (-\lambda)$ and  lies in the cosocle of $\Ker[\text{U}(\gl_{n,\Sigma_K}) \otimes_{\text{U}(\fp_{\Sigma_K})}  (-\lambda)\ra L^-(-\lambda)]$. Using \cite[Thm.]{OS}, we deduce the constituents for $i$ as in (2) appear with multiplicity one in  $V_0$, and all  lie  in  the socle  of (\ref{Eparaind3}). 

For  $i$  as  in (1), let $V_I:=\big(\Ind_{P^-(K)}^{\GL_n(K)} ((\boxtimes_{\substack{i=1, \cdots, r\\  i \neq k}} \pi_{\alg}(\phi_{\sF_P,i}, \textbf{h}^i) \widehat{\boxtimes} \sC(I, s_{i,\sigma})_k)\varepsilon^{-1}\circ \theta^P\big)^{\Q_p-\an}$
where \ $\sC(I,s_{i,\sigma})_k$ \ denotes \ the \ corresponding \ representation \ in \ the \ cosocle \ of \ $\pi_1(\phi_{\sF_P,i}, \textbf{h}^i)$. 
By  (\ref{ECIsigma}) for $\sC(I,s_{i,\sigma})_k$ and the transitivity of parabolic inductions, 
$V_I$ injects into $\big(\Ind_{B^-(K)}^{\GL_n(K)} z^{-s_{i,\sigma}\cdot \lambda} w(\phi) \eta\big)^{\Q_p-\an}$ for any $w\in S_n$ satisfying $w(\{1,\cdots, i\})=I$. Since the latter representation has socle $\sC(I,s_{i,\sigma})$ with multiplicity one (cf. (\ref{ECIsigma})), so does its subrepresentation  $V_I$. It is not difficult to see these give all the $\sC(I,s_{i,\sigma})$ appearing in (1), and  they all have multiplicity one.  Let $U$ be the closed subrepresentation of $\pi_1(\phi_{\sF_P,k},\textbf{h}^k)$ of the form $[\pi_{\alg}(\phi_{\sF_P,k},\textbf{h}^k) \lin \sC(I, s_{i,\sigma})_k]$, which is clearly a closed subrepresentation of a certain principal series of $\GL_{n_k}(K)$. Using  the transitivity of parabolic inductions, one sees 
$W:=\big(\Ind_{P^-(K)}^{\GL_n(K)} \big((\boxtimes_{\substack{i=1, \cdots, r\\  i \neq k}} \pi_{\alg}(\phi_{\sF_P,i}, \textbf{h}^i) \widehat{\boxtimes} U\big)\varepsilon^{-1}\circ \theta^P\big)^{\Q_p-\an}$ is a closed subrepresentation of $(\Ind_{B^-(K)}^{\GL_n(K)} z^{\lambda} w(\phi) \eta)^{\Q_p-\an}$ with $w\in S_n$ satisfying $w(\{1,\cdots, i\})=I$. For the latter representation, $\sC(I,s_{i,\sigma})$ has multiplicity one and lies in the socle of its quotient by $\pi_{\alg}(\phi,\textbf{h})$. We then deduce $\sC(I,s_{i,\sigma})$ lies in the socle of $W/\pi_{\alg}(\phi, \textbf{h})$ hence in the socle of (\ref{Eparaind3}).

 Finally, \ by \ \cite[Thm.]{OS}, \ one \ sees \ every \ $\sC(I,s_{i,\sigma})$ \ in \ the \ representation \ $\big(\Ind_{P^-(K)}^{\GL_n(K)} (\widehat{\boxtimes}_{i=1}^r \pi_1 (\phi_{\sF_P,i}, \textbf{h}^i)) \varepsilon^{-1}\circ \theta^P\big)^{\Q_p-\an}$ must come from either $V_0$ or $V_I$ with $I$ as in (1), and has multiplicity one. This completes the proof.
\end{proof}
Denote by $S_{\sF_P}$ the subset of the constituents $\sC(I,s_{i,\sigma})$, which satisfy one of the conditions in Lemma \ref{Lparaind}. We have 
\begin{equation}\label{ESFP}
	\# S_{\sF_P}=\big(\sum_{i=1}^r (2^{n_i}-2)+(r-1)\big)d_K.
\end{equation}The representation $\big(\Ind_{P^-(K)}^{\GL_n(K)} (\widehat{\boxtimes}_{i=1}^r \pi_1 (\phi_{\sF_P,i}, \textbf{h}^i)) \varepsilon^{-1}\circ \theta^P\big)^{\Q_p-\an}$ contains a unique \ subrepresentation \ $\pi_{\sF_P}(\phi, \textbf{h})$ \ such \ that \ $\soc_{\GL_n(K)} \pi_{\sF_P}(\phi, \textbf{h}) \cong \pi_{\alg}(\phi, \textbf{h})$ and $\pi_{\sF_P}(\phi, \textbf{h})/\pi_{\alg}(\phi, \textbf{h})\cong \oplus_{\sC \in S_{\sF_P}} \sC$. Note when $P=B$, $\sF_P=\sT_w$, then $\pi_{\sF_P}(\phi, \textbf{h})\cong \PS_1(w(\phi), \textbf{h})$. It is easy to see the  injection $\pi_{\alg}(\phi, \textbf{h})\hookrightarrow \pi_{\sF_P}(\phi, \textbf{h})$ (cf. (\ref{Eparaind})) uniquely extends to $\pi_{\sF_P}(\phi, \textbf{h})\hookrightarrow \pi_1(\phi, \textbf{h})$. 
\begin{proposition}\label{PextpiFP1}
We have 
$\dim_E \Ext^1_{\GL_n(K)}(\pi_{\alg}(\phi, \textbf{h}), \pi_{\sF_P}(\phi, \textbf{h}))=n+d_K r+d_K\sum_{i=1}^r (2^{n_i}-2)$.
And the following push-forward map is injective
\begin{equation}\label{EpiPpi1}
	\Ext^1_{\GL_n(K)}(\pi_{\alg}(\phi, \textbf{h}), \pi_{\sF_P}(\phi, \textbf{h})) \hooklongrightarrow \Ext^1_{\GL_n(K)}(\pi_{\alg}(\phi, \textbf{h}), \pi_1(\phi, \textbf{h})).
\end{equation} 
\end{proposition}
\begin{proof}
We have an exact sequence by d\'evissage
\begin{multline}\label{EpiPpi12}
	0 \lra 	\Ext^1_{\GL_n(K)}(\pi_{\alg}(\phi, \textbf{h}), \pi_{\sF_P}(\phi, \textbf{h})) \lra \Ext^1_{\GL_n(K)}(\pi_{\alg}(\phi, \textbf{h}), \pi_1(\phi, \textbf{h})) \\ \lra \Ext^1_{\GL_n(K)}(\pi_{\alg}(\phi, \textbf{h}), \oplus_{\sC\notin S_{\sF_P}} \sC).
\end{multline} 
The injectivity of (\ref{EpiPpi1}) follows. By Proposition \ref{Pextpi1} (1), the last map in (\ref{EpiPpi12}) is surjective. The first part follows then by a direct calculation using Proposition \ref{Pextpi1} (1), Lemma \ref{Lextps} (1) and (\ref{ESFP}).
\end{proof}
Set $\Ext^1_{\sF_P}(\pi_{\alg}(\phi, \textbf{h}), \pi_1(\phi, \textbf{h}))$ to be the image of (\ref{EpiPpi1}). The injection $\pi_{\alg}(\phi, \textbf{h})\hookrightarrow \pi_{\sF_P}(\phi, \textbf{h})$ induces a natural injection
\begin{equation}\label{Einjg'Fp}
	\Ext^1_{g'}(\pi_{\alg}(\phi, \textbf{h}), \pi_1(\phi, \textbf{h})) \hooklongrightarrow\Ext^1_{\sF_P}(\pi_{\alg}(\phi, \textbf{h}), \pi_1(\phi, \textbf{h})).
\end{equation}
We have natural maps
\begin{multline}\label{Eparaind000}
	\Ext^1_{L_P(K)}\big(\boxtimes_{i=1}^r  \pi_{\alg}(\phi_{\sF_P,i}, \textbf{h}^i), \widehat{\boxtimes}_{i=1}^r \pi_1(\phi_{\sF_P,i}, \textbf{h}^i)\big) \\
	\lra
	\Ext^1_{\GL_n(K)}\Big(\pi_{\alg}(\phi, \textbf{h}), \big(\Ind_{P^-(K)}^{\GL_n(K)} (\widehat{\boxtimes}_{i=1}^r \pi_1 (\phi_{\sF_P,i}, \textbf{h}^i)) \varepsilon^{-1}\circ \theta^P\big)^{\Q_p-\an}\Big)\\
	\xlongleftarrow{\sim} \Ext^1_{\GL_n(K)}(\pi_{\alg}(\phi,\textbf{h}), \pi_{\sF_P}(\phi, \textbf{h})),
\end{multline}
where the first map is obtained by taking $(\Ind_{P^-(K)}^{\GL_n(K)} - \otimes_E \varepsilon^{-1} \circ \theta^P)^{\Q_p-\an}$ and using the pull-back via  $\pi_{\alg}(\phi, \textbf{h})\hookrightarrow(\Ind_{P^-(K)}^{\GL_n(K)} (\boxtimes_{i=1}^r  \pi_{\alg}(\phi_{\sF_P,i}, \textbf{h}^i))\varepsilon^{-1} \circ \theta^P)^{\Q_p-\an}$ (cf. (\ref{Eparaind})),  and where the second map is the natural push-forward map, which is bijective by \cite[Lem.~2.26]{Ding7} (and Lemma \ref{Lparaind}). 

For $(\widetilde{\pi}_i)\in \prod_{i=1}^r \Ext^1_{\GL_{n_i}}\big(\pi_{\alg}(\phi_{\sF_P,i}, \textbf{h}^i), \pi_1(\phi_{\sF_P,i}, \textbf{h}^i)\big)$, consider the $L_P(K)$-representation $\widehat{\boxtimes}_{i=1}^r \widetilde{\pi}_i$ (where the completed tensor product is taken over $E$). It is clear that $\widehat{\boxtimes}_{i=1}^r \widetilde{\pi}_i$ admits a quotient $V$ given by an extension of $\boxtimes_{i=1}^r \pi_{\alg}(\phi_{\sF_P,i},\textbf{h}^i)$ by $W:=\oplus_{i=1}^r \big(\widehat{\boxtimes}_{\substack{j=1,\cdots, r\\ j\neq i}} \pi_1(\phi_{\sF_P,j},\textbf{h}^j) \boxtimes_E \pi_{\alg}(\phi_{\sF_P,i}\textbf{h}^i)\big)$. The push-forward of $V$ via the natural map $W\ra  \widehat{\boxtimes}_{i=1}^r \pi_1(\phi_{\sF_P,i}, \textbf{h}^i)$ (induced by $\pi_{\alg}(\phi_{\sF_P,i},\textbf{h}^i)\hookrightarrow \pi_1(\phi_{\sF_P,i},\textbf{h}^i)$) gives an element in $\Ext^1_{L_P(K)}\big(\boxtimes_{i=1}^r  \pi_{\alg}(\phi_{\sF_P,i}, \textbf{h}^i), \widehat{\boxtimes}_{i=1}^r \pi_1(\phi_{\sF_P,i}, \textbf{h}^i)\big)$. In this way, we obtain a map
\begin{multline}\label{EextLevi}
	\prod_{i=1}^r \Ext^1_{\GL_{n_i}(K)}\big(\pi_{\alg}(\phi_{\sF_P,i}, \textbf{h}^i), \pi_1(\phi_{\sF_P,i}, \textbf{h}^i)\big) \\
	\lra 	\Ext^1_{L_P(K)}\big(\boxtimes_{i=1}^r  \pi_{\alg}(\phi_{\sF_P,i}, \textbf{h}^i), \widehat{\boxtimes}_{i=1}^r \pi_1(\phi_{\sF_P,i}, \textbf{h}^i)\big).
\end{multline}
Together with (\ref{Eparaind000}), we finally get a map:
\begin{equation}\label{EzetaFp}
	\zeta_{\sF_P}: 	\prod_{i=1}^r \Ext^1_{\GL_{n_i}(K)}\big(\pi_{\alg}(\phi_{\sF_P,i}, \textbf{h}^i), \pi_1(\phi_{\sF_P,i}, \textbf{h}^i)\big) 
	\lra  \Ext^1_{\sF_P}\big(\pi_{\alg}(\phi, \textbf{h}),  \pi_{1}(\phi, \textbf{h})\big).
\end{equation}
For $w\in S_n$, let $\sT_w$ be the $B$-filtration of $\oplus_{i=1}^n\phi_i$ associated to $w$. Suppose $\sT_w$ is compatible with $\sF_P$. It is clear that $\PS_1(w(\phi),\textbf{h})$ is a subrepresentation of  $\pi_{\sF_P}(\phi, \textbf{h})$  (e.g. by comparing constituents and using Lemma \ref{Lextps} (1)), hence (by d\'evissage) $\Ext^1_w(\pi_{\alg}(\phi, \textbf{h}), \pi_1(\phi, \textbf{h})) \hookrightarrow \Ext^1_{\sF_P}(\pi_{\alg}(\phi, \textbf{h}), \pi_1(\phi, \textbf{h}))$.
\begin{proposition}\label{PextpiFP}
(1) The map $\zeta_{\sF_P}$ is bijective. 

(2) For any $w$ such that the associated $B$-filtration $\sT_w$ is compatible with $\sF_P$, the following diagram commutes
\begin{equation}\label{Eparaps}
	\begin{CD}
		\prod_{i=1}^r \Hom(T(K)\cap L_{P,i}(K),E) @> \sim >> \Hom(T(K),E) \\
		@V \sim V (\ref{Eisomzetaw}) V @V \sim V (\ref{Eisomzetaw}) V \\
		\prod_{i=1}^r \Ext^1_{\sT_{w,i}}\big(\pi_{\alg}(\phi_{\sF_P,i}, \textbf{h}^i), \pi_1(\phi_{\sF_P,i}, \textbf{h}^i)\big) 
		@>\sim >>  \Ext^1_{w}\big(\pi_{\alg}(\phi, \textbf{h}),  \pi_{1}(\phi, \textbf{h})\big)
	\end{CD}
\end{equation}
where $\sT_{w,i}$ is the induced $B\cap L_P$-filtration on $\gr_i \sF_P$, and the bottom map is induced by $\zeta_{\sF_P}$.
\end{proposition}
\begin{proof}Given $\psi=(\psi_i)\in 	\prod_{i=1}^r \Hom(T(K)\cap L_{P,i}(K),E)\cong \Hom(T(K),E)$, we have $(1+\psi\epsilon)\cong \boxtimes^{i=1,\cdots r}_{E[\epsilon]/\epsilon^2} (1+\psi_i \epsilon)$ as character of $T(K)$ over $E[\epsilon]/\epsilon^2$ hence as element in $\Ext^1_{T(K)}(1,1)$. Note $\boxtimes^{i=1,\cdots r}_{E[\epsilon]/\epsilon^2} (1+\psi_i \epsilon)$ admits an extension construction in a similar way as given above (\ref{EextLevi}). 
The commutativity of (\ref{Eparaps}) then follows by definition and the transitivity of parabolic inductions (see also the discussion above (\ref{Eparaind})). In particular, we deduce  the bottom map of (\ref{Eparaps}) is bijective. Note that any $\sC(I,s_{i,\sigma})\in S_{\sF_P}$ is a constituent of some $\PS_1(w(\phi), \textbf{h})$ with $\sT_w$ compatible with $\sF_P$. By similar arguments  as in the proof of Proposition \ref{Pextpi1}, one sees the natural map (``$\subset$" means compatible)
\begin{equation} \label{EinffernFp}\bigoplus_{\text{$\sT_w\subset\sF_P$}} \Ext^1_w(\pi_{\alg}(\phi, \textbf{h}), \pi_1(\phi, \textbf{h}))\lra \Ext^1_{\sF_P}(\pi_{\alg}(\phi, \textbf{h}), \pi_1(\phi, \textbf{h}))
\end{equation} is surjective, hence so is $\zeta_{\sF_P}$. By Proposition \ref{Pextpi1} (1) and Proposition \ref{PextpiFP1},  both sides of (\ref{EzetaFp}) have the same dimension over $E$ hence $\zeta_{\sF_P}$ is bijective. 
\end{proof}
\noindent Let $\Ext^1_{\sF_P,g'}(\pi_{\alg}(\phi, \textbf{h}), \pi_1(\phi, \textbf{h})):=\zeta_{\sF_P}\big(\prod_{i=1}^r \Ext^1_{g'}(\pi_{\alg}(\phi_{\sF_P,i}, \textbf{h}^i), \pi_1(\phi_{\sF_p, i}, \textbf{h}^i))\big)$. By Proposition \ref{PextpiFP} (2), and (\ref{Eintg'}), assuming $\sT_{w_i}$ compatible with $\sF_P$, the following diagram commutes (cf. (\ref{EPg'T})):
\begin{equation}\label{EintP2}
\begin{CD}
	\Hom_{P,g'}(T(K),E) @> \zeta_{w_1} > \sim > \Ext^1_{\sF_P,g'}(\pi_{\alg}(\phi, \textbf{h}), \pi_1(\phi, \textbf{h})) \\
	@V w_2w_1^{-1} V \sim V @| \\
	\Hom_{P,g'}(T(K),E) @> \zeta_{w_2} > \sim > \Ext^1_{\sF_P,g'}(\pi_{\alg}(\phi, \textbf{h}), \pi_1(\phi, \textbf{h})).
\end{CD}
\end{equation}

We finally discuss some intertwining properties related to \S~\ref{S24}.  Let $\phi^1:=\phi_1 \boxtimes \cdots \boxtimes \phi_{n-1}: T_{n-1}(K) \ra E^{\times}$, $\textbf{h}^1:=(\textbf{h}_1, \cdots, \textbf{h}_{n-1})$ and $\textbf{h}^2:=(\textbf{h}_2, \cdots, \textbf{h}_n)$ which are dominant weights of $\ft_{n-1,\Sigma_K}$.  We have locally $\Q_p$-analytic $\GL_{n-1}(K)$-representations 
$\pi_{\alg}(\phi^1, \textbf{h}^i)\subset \pi_1(\phi^1, \textbf{h}^i)$ for $i=1,2$, 
and  parabolic inductions
$(\Ind_{P_1^-(K)}^{\GL_n(K)} (\pi_1(\phi^1, \textbf{h}^1) \otimes \varepsilon)\boxtimes \phi_n z^{\textbf{h}_n})^{\Q_p-\an}$ and $(\Ind_{P_2^-(K)}^{\GL_n(K)} \phi_n z^{\textbf{h}_1}\varepsilon^{n-1} \boxtimes \pi_1(\phi^1, \textbf{h}^2))^{\Q_p-\an}$.
Let $\sF$ be the filtration $\oplus_{i=1}^{n-1} \phi_i \subset \oplus_{i=1}^n \phi_i$ and $\sG$ be the filtration $\phi_n \subset \oplus_{i=1}^n \phi_i$. By Lemma \ref{Lparaind}, $C(I, s_{i,\sigma})$ appears in $\pi_{\sF}(\phi, \textbf{h})$ (resp. in $\pi_{\sG}(\phi, \textbf{h})$) if and only if $i=1,\cdots, n-1$, $\sigma\in \Sigma_K$ and $I\subset \{1,\cdots, n-1\}$, $\#I=i$ (resp. $I=I_1 \cup \{n\}$ with $I_1\subset \{1,\cdots, n-1\}$ and $\#I_1=i-1$). In particular, $\pi_1(\phi, \textbf{h})/\pi_{\alg}(\phi, \textbf{h})\cong \big(\pi_{\sF}(\phi, \textbf{h})/\pi_{\alg}(\phi, \textbf{h})\big)\oplus  \big(\pi_{\sG}(\phi, \textbf{h})/\pi_{\alg}(\phi, \textbf{h})\big)$.
The following proposition is straightforward (where the right exactness of the last sequence follows by comparing dimensions, using Proposition \ref{Pextalg} (1), Proposition \ref{Pextpi1} (1) and Proposition \ref{PextpiFP1}): 
\begin{proposition}\label{PFGfil}
There is a natural exact sequence of $\GL_n(K)$-representations
\begin{equation*}
	0 \lra \pi_{\alg}(\phi, \textbf{h}) \lra \pi_{\sF}(\phi, \textbf{h}) \oplus \pi_{\sG}(\phi, \textbf{h}) \lra \pi_1(\phi, \textbf{h}) \lra 0.
\end{equation*}
Consequently, we have a natural exact sequence
\begin{multline*}
	0 \lra \Ext^1_{g'}(\pi_{\alg}(\phi, \textbf{h}), \pi_1(\phi,\textbf{h})) \lra \\
	\Ext^1_{\sF}(\pi_{\alg}(\phi, \textbf{h}), \pi_1(\phi,\textbf{h}))  \oplus \Ext^1_{\sG}(\pi_{\alg}(\phi, \textbf{h}), \pi_1(\phi,\textbf{h})) \\ \lra \Ext^1_{\GL_n(K)}(\pi_{\alg}(\phi, \textbf{h}), \pi_1(\phi,\textbf{h}))  \lra 0.
\end{multline*}
\end{proposition}
\begin{remark}
By Proposition \ref{PextpiFP} (1), we have a bijection
\begin{equation*}
	\zeta_{\sF}:	\Ext^1_{\GL_{n-1}}(\pi_{\alg}(\phi^1,\textbf{h}^1), \pi_1(\phi^1, \textbf{h}^1)) \times \Hom(K^{\times},E) \xlongrightarrow{\sim} 	\Ext^1_{\sF}(\pi_{\alg}(\phi, \textbf{h}), \pi_1(\phi,\textbf{h})),
\end{equation*}
and a similar bijection $\zeta_{\sG}$. 
\end{remark}
\subsubsection{Locally $\sigma$-analytic parabolic inductions}
Let $\sigma\in \Sigma_K$. Recall a locally $\Q_p$-analytic representation $V$ of $\GL_n(K)$ over $E$ is called locally \textit{$\sigma$-analytic} if the $\gl_n(K) \otimes_{\Q_p} E$-action (obtained by derivation) on $V$ factors through $\gl_n(K) \otimes_{K, \sigma} E$ (cf. \cite[\S~2]{Sch10}). And $V$ is called \textit{$\ug_{\Sigma_K\setminus \{\sigma\}}$-algebraic} if $\text{U}(\ug_{\Sigma_K\setminus \{\sigma\}}) v$ is a finite dimensional algebraic representation of $\ug_{\Sigma_K\setminus \{\sigma\}}$ over $E$ for all $v\in V$. Let $\lambda_{\sigma}$ be the $\sigma$-component of $\lambda$, and $\lambda^{\sigma}:=(\lambda_{\tau})_{\tau\neq \sigma}$. We also view them as weights of $\ft_{\Sigma_K}$ in the obvious way. 
For $i=1, \cdots, n-1$, $I\subset\{1,\cdots, n\}$, $\#I=i$, let $w\in S_n$ such that $w(\{1,\cdots, i\})=I$. We have 
\begin{equation*}
\sC(I,s_{i,\sigma})\cong \cF_{B^-}^{\GL_n}(L^-(-s_{i,\sigma} \cdot \lambda), w(\phi) \eta) \cong \cF_{B^-}^{\GL_n}(L^-(-s_{i,\sigma} \cdot \lambda_{\sigma}), w(\phi) \eta) \otimes_E L(\lambda^{\sigma}).
\end{equation*}
Note we have  $\cF_{B^-}^{\GL_n}(L^-_{\sigma}(-s_{i,\sigma} \cdot \lambda_{\sigma}), w(\phi) \eta)  \hookrightarrow (\Ind_{B^-(K)}^{\GL_n(K)}w(\phi) \eta z^{s_{i,\sigma}\cdot \lambda_{\sigma}})^{\sigma-\an}$, where the sup-script ``$\sigma-\an$" means the locally $\sigma$-analytic induction. So the both are  locally $\sigma$-analytic. Let  $\pi_{1,\sigma}(\phi, \textbf{h})$ be the subrepresentation of $\pi_1(\phi, \textbf{h})$ given by the extension of $\oplus_{\substack{i=1,\cdots, n-1 \\ I \subset \{1,\cdots, n\}, \#I=i}} \sC(I,s_{i,\sigma})$ by $\pi_{\alg}(\phi, \lambda)$.  Similarly, for $w\in S_n$, let  $\PS_{1,\sigma}(w(\phi), \textbf{h})\subset \PS_1(w(\phi), \textbf{h})$ be the subrepresentation with irreducible constituents $\pi_{\alg}(\phi, \textbf{h})$ and $\sC(w,s_{i,\sigma})$ for $i=1, \cdots, n-1$. It is easy to see 
\begin{equation*}
\PS_{1,\sigma}(w(\phi), \textbf{h})=\PS_1(w(\phi),\textbf{h}) \cap \big((\Ind_{B^-(K)}^{\GL_n(K)} w(\phi) \eta z^{\lambda_{\sigma}})^{\sigma-\an} \otimes_E L(\lambda^{\sigma}) \big)
\hookrightarrow \PS(w(\phi), \textbf{h}).
\end{equation*}
Moreover, $\pi_{1,\sigma}(\phi, \textbf{h})$ is the unique quotient of $\oplus_{\pi_{\alg}(\phi, \lambda)}^{w\in S_n} \PS_{1,\sigma}(w(\phi), \textbf{h})$ of socle $\pi_{\alg}(\phi, \textbf{h})$. In particular, $\pi_{1,\sigma}(\phi, \textbf{h})$ is $\ug_{\Sigma_K\setminus \{\sigma\}}$-algebraic. In fact, $\pi_{1,\sigma}(\phi, \textbf{h})$  is  the maximal  $\ug_{\Sigma_K\setminus \{\sigma\}}$-algebraic subrepresentation of $\pi_1(\phi, \textbf{h})$. 

For  $\ug_{\Sigma_K\setminus \{\sigma\}}$-algebraic representations $V$, $W$, we denote by $\Ext^1_{\sigma}(V,W) \subset \Ext^1_{\GL_n(K)}(V,W)$ the subspace of extensions, which are  $\ug_{\Sigma_K\setminus \{\sigma\}}$-algebraic. Let $\Hom_{\sigma, g'}(T(K),E):= \Hom_{g'}(T(K),E)\cap \Hom_{\sigma}(T(K),E)$ (recalling $\Hom_{\sigma}(T(K),E)$ is just the subspace of locally $\sigma$-analytic characters).
\begin{lemma}\label{Lsigma00}
We have $\dim_E \Ext^1_{\sigma}(\pi_{\alg}(\phi, \textbf{h}), \pi_{\alg}(\phi, \textbf{h}))=n+1$, and (\ref{Ezetaw0}) induces an isomorphism
$\Hom_{\sigma, g'}(T(K),E) \xrightarrow{\sim} \Ext^1_{\sigma}(\pi_{\alg}(\phi, \textbf{h}), \pi_{\alg}(\phi, \textbf{h}))$.
\end{lemma}
\begin{proof}
As $\Ext^1_{\GL_n(K),Z}(\pi_{\alg}(\phi, \textbf{h}), \pi_{\alg}(\phi, \textbf{h}))\subset \Ext^1_{\lalg}(\pi_{\alg}(\phi, \textbf{h}), \pi_{\alg}(\phi, \textbf{h}))$ (by (\ref{Eautolalg})) hence is contained in $ \Ext^1_{\sigma}(\pi_{\alg}(\phi, \textbf{h}), \pi_{\alg}(\phi, \textbf{h}))$, we have an exact sequence (similarly as in \cite[Lem.~3.16]{BD1}) 
\begin{multline*}
	0 \lra \Ext^1_{\GL_n(K),Z}(\pi_{\alg}(\phi, \textbf{h}), \pi_{\alg}(\phi, \textbf{h})) \lra \Ext^1_{\sigma}(\pi_{\alg}(\phi, \textbf{h}), \pi_{\alg}(\phi, \textbf{h}))\\ \lra \Hom_{\sigma}(Z(K),E) \lra 0.
\end{multline*}
The first part follows. It is clear that (\ref{Ezetaw0}) induces the map in the lemma by restriction, which is hence injective. However, both the source and target spaces have the same dimension $n+1$, so the map  is bijective. 
\end{proof}

\begin{proposition}\label{Pextpssigma}
Let $w\in S_n$, the map (\ref{Eind1}) induces an isomorphism
\begin{equation}\label{Eindsigma}
	\Hom_{\sigma}(T(K),E)\xlongrightarrow{\sim} \Ext^1_{\sigma}(\pi_{\alg}(\phi, \textbf{h}), \PS_{1,\sigma}(w(\phi), \textbf{h})).
\end{equation}
\end{proposition}
\begin{proof}
For $\psi\in \Hom_{\sigma}(T(K),E)$, by similar arguments as in the proof of Proposition \ref{Pextalg},  $I_{B^-(K)}^{\GL_n(K)}\big(w(\phi) \eta z^{\lambda}(1+\psi \epsilon)\big)$ is a subrepresentation of $(\Ind_{B^-(K)}^{\GL_n(K)} w(\phi) \eta z^{\lambda_{\sigma}} (1+\psi \epsilon))^{\sigma-\an} \otimes_E L(\lambda^{\sigma})$, 
hence is $\ug_{\Sigma_K\setminus \{\sigma\}}$-algebraic. Together with the description of (\ref{Eind1})  in  Remark \ref{Rind1}, we deduce (\ref{Eind1}) induces the injective map in (\ref{Eindsigma}) by restriction. We have an exact sequence by d\'evissage \begin{multline}\label{Edivisigma000}
	0 \lra \Ext^1_{\sigma}(\pi_{\alg}(\phi, \textbf{h}), \pi_{\alg}(\phi, \textbf{h})) \lra \Ext^1_{\sigma}(\pi_{\alg}(\phi, \textbf{h}), \PS_{1,\sigma}(w(\phi), \textbf{h})) \\ \lra \oplus_{i=1}^{n-1} \Ext^1_{\sigma}(\pi_{\alg}(\phi, \textbf{h}), \sC(w(\phi), s_{i,\sigma})).
\end{multline}
By Lemma \ref{Lsigma00} and Lemma \ref{Lextps} (1),  
$\dim_E \Ext^1_{\sigma}(\pi_{\alg}(\phi,\textbf{h}), \PS_{1,\sigma}(w(\phi), \textbf{h}))\leq (n+1)+(n-1)=2n$. However, the source of (the injective) (\ref{Eindsigma}) has dimension $2n$, so (\ref{Eindsigma}) must be bijective.  
\end{proof}
\begin{remark}\label{Rautosigma}
By the above proof, we see the last map in (\ref{Edivisigma000}) is surjective and  $\Ext^1_{\sigma}(\pi_{\alg}(\phi, \textbf{h}), \sC(w(\phi), s_{i,\sigma}))\xrightarrow{\sim} \Ext^1_{\GL_n(K)}(\pi_{\alg}(\phi, \textbf{h}), \sC(w(\phi), s_{i,\sigma})).$
\end{remark}
Denote \ by \ $\Ext^1_{\sigma}(\pi_{\alg}(\phi, \textbf{h}), \pi_1(\phi, \textbf{h}))$ \ (resp. \ $\Ext^1_{\sigma, g'}(\pi_{\alg}(\phi, \textbf{h}), \pi_1(\phi, \textbf{h}))$) \ the \ image \ of $\Ext^1_{\sigma}(\pi_{\alg}(\phi,\textbf{h}), \pi_{1,\sigma}(\phi, \textbf{h}))$ (resp.  $\Ext^1_{\sigma}(\pi_{\alg}(\phi, \textbf{h}), \pi_{\alg}(\phi, \textbf{h}))$ ) via the (injective) push-forward map. It is easy to see 
\begin{equation*}
\Ext^1_{\sigma, g'}(\pi_{\alg}(\phi, \textbf{h}), \pi_1(\phi, \textbf{h}))=\Ext^1_{g'}(\pi_{\alg}(\phi, \textbf{h}), \pi_1(\phi, \textbf{h})) \cap \Ext^1_{\sigma}(\pi_{\alg}(\phi, \textbf{h}), \pi_1(\phi, \textbf{h})).
\end{equation*}
\begin{proposition}\label{Pextpi1sigma}
(1) We have an exact sequence 
\begin{multline}\label{Edivsigma}
	0 \lra \Ext^1_{\sigma}(\pi_{\alg}(\phi, \textbf{h}), \pi_{\alg}(\phi, \textbf{h})) \lra \Ext^1_{\sigma}(\pi_{\alg}(\phi, \textbf{h}), \pi_1(\phi, \textbf{h})) \\ \lra \oplus_{\substack{i=1,\cdots, n-1\\ I \subset \{1,\cdots, n-1\}, \# I =i}} \Ext^1_{\GL_n(K)}(\pi_{\alg}(\phi, \textbf{h}), \sC(I,s_{i,\sigma})) \lra 0.
\end{multline}
And $\dim_E \Ext^1_{\sigma}(\pi_{\alg}(\phi, \textbf{h}), \pi_{1}(\phi, \textbf{h}))=n+2^n-1$.

(2) The map (\ref{Etphih}) induces a surjection
	$t_{\phi, \textbf{h}}:	\oplus_{w\in S_n} \Hom_{\sigma}(T(K),E)\twoheadrightarrow \Ext^1_{\sigma}(\pi_{\alg}(\phi, \textbf{h}), \pi_{1}(\phi, \textbf{h}))$.

(3) The following map is surjective
\begin{equation}\label{Evarytau}
	\oplus_{\tau\in \Sigma_K} \Ext^1_{\tau}(\pi_{\alg}(\phi, \textbf{h}), \pi_1(\phi, \textbf{h})) \lra \Ext^1_{\GL_n(K)}(\pi_{\alg}(\phi, \textbf{h}), \pi_1(\phi, \textbf{h})),
\end{equation}
and induces an isomorphism
\begin{multline}\label{Evarytau2}
	\oplus_{\tau\in \Sigma_K} \big(\Ext^1_{\tau}(\pi_{\alg}(\phi, \textbf{h}), \pi_1(\phi, \textbf{h}))/\Ext^1_g(\pi_{\alg}(\phi, \textbf{h}), \pi_1(\phi, \textbf{h}))\big) \\
	\xlongrightarrow{\sim} \Ext^1_{\GL_n(K)}(\pi_{\alg}(\phi, \textbf{h}), \pi_1(\phi, \textbf{h}))/\Ext^1_g(\pi_{\alg}(\phi, \textbf{h}), \pi_1(\phi, \textbf{h})).
\end{multline}
\end{proposition}
\begin{proof}
(1) follows by the same argument as in the proof of Proposition \ref{Pextpi1} (1). 
(2) follows from Proposition \ref{Pextpssigma} and Remark \ref{Rautosigma} by the same argument as in the proof of Proposition \ref{Pextpi1} (2).
The first part of (3) follows easily by comparing  the exact sequences (\ref{Edivi}) and (\ref{Edivsigma}). It is clear that (\ref{Evarytau}) induces (\ref{Evarytau2}), which is hence surjective. However, both the sources and target spaces have the same dimension $(2^n-1)d_K$ by (1) and Proposition \ref{Pextalg} (1),  Proposition \ref{Pextpi1} (1). So (\ref{Evarytau2}) is bijective. 
\end{proof}

Now let $P$ be a standard parabolic subgroup of $\GL_n$, and $\sF_P$ be a $P$-filtration on  $\phi$. We use the notation in \S~\ref{S313}. Let $\pi_{\sF_P,\sigma}(\phi, \textbf{h}):=\pi_{\sF_P}(\phi,\textbf{h})\cap \pi_{1,\sigma}(\phi, \textbf{h})$, which is the maximal $\ug_{\Sigma_K\setminus \{\sigma\}}$-algebraic subrepresentation of $\pi_{\sF_P}(\phi, \textbf{h})$. Then $\pi_{\sF_P,\sigma}(\phi, \textbf{h})$ is an extension of the direct sum of $\sC(I,s_{i,\sigma})\in S_{\sF_P}$ (for the fixed $\sigma$) by $\pi_{\alg}(\phi, \textbf{h})$. We denote by $\Ext^1_{\sigma,\sF_P}(\pi_{\alg}(\phi, \textbf{h}), \pi_1(\phi, \textbf{h}))$ the image of $\Ext^1_{\sigma}(\pi_{\alg}(\phi, \textbf{h}), \pi_{\sF_P, \sigma}(\phi, \textbf{h}))$ via the (injective) push-forward map. As previously, we also write $\Ext^1_{\sigma, w}$ for $\Ext^1_{\sigma, \sT_w}$. One easily sees 
$\Ext^1_{\sigma, \sF_P}(\pi_{\alg}(\phi, \textbf{h}), \pi_1(\phi, \textbf{h}))=\Ext^1_{\sigma}(\pi_{\alg}(\phi, \textbf{h}), \pi_1(\phi,\textbf{h})) \cap \Ext^1_{\sF_P}(\pi_{\alg}(\phi, \textbf{h}), \pi_1(\phi, \textbf{h}))$.
\begin{proposition}\label{PdimsigmaPara}
(1) We have $\dim_E \Ext^1_{\sigma,\sF_P}(\pi_{\alg}(\phi, \textbf{h}), \pi_1(\phi, \textbf{h}))=n+r+\sum_{i=1}^r (2^{n_i}-2)$.

(2) The isomorphism (\ref{EzetaFp}) induces an isomorphism 
\begin{equation}\label{EzetaFpsigma}
	\prod_{i=1}^r \Ext^1_{\sigma}(\pi_{\alg}(\phi_{\sF_P,i}, \textbf{h}^i), \pi_{1}(\phi_{\sF_P,i}, \textbf{h}^i)) \xlongrightarrow{\sim} \Ext^1_{\sigma,\sF_P}(\pi_{\alg}(\phi, \textbf{h}), \pi_1(\phi, \lambda)).
\end{equation}
Moreover, for any $w$ such that the associated $B$-filtration $\sT_w$ is compatible with $\sF_P$, the following diagram commutes
\begin{equation*}
	\begin{CD}
		\prod_{i=1}^r \Hom_{\sigma}(T(K)\cap L_{P,i}(K),E) @> \sim >> \Hom_{\sigma}(T(K),E) \\
		@V \sim V (\ref{Eindsigma})V @V \sim V (\ref{Eindsigma}) V \\
		\prod_{i=1}^r \Ext^1_{\sigma, \sT_{w,i}}\big(\pi_{\alg}(\phi_{\sF_P,i}, \textbf{h}^i), \pi_{1,\sigma}(\phi_{\sF_P,i}, \textbf{h}^i)\big) 
		@> \sim >>  \Ext^1_{\sigma,w}\big(\pi_{\alg}(\phi, \textbf{h}),  \pi_{1,\sigma}(\phi, \textbf{h})\big)
	\end{CD}
\end{equation*}
where $\sT_{w,i}$ is the induced $B\cap L_P$-filtration on $\gr_i\sF_P$.
\end{proposition}
\begin{proof}
By d\'evissage, Lemma \ref{Lparaind} and a similar argument as in the proof of Proposition \ref{Pextpssigma} (noting for a fixed $\sigma$, $\# \{\sC(I,s_{i,\sigma})\in S_{\sF_P}\}=\sum_{i=1}^r (2^{n_i}-2)+(r-1)$), we see $\dim_E \Ext^1_{\sigma,\sF_P}(\pi_{\alg}(\phi, \textbf{h}), \pi_1(\phi, \textbf{h}))\leq n+r+\sum_{i=1}^r (2^{n_i}-2)$ and that (\ref{EzetaFp}) restricts to an injective map as in (\ref{EzetaFpsigma}). Its source space has dimension $n+r+\sum_{i=1}^r (2^{n_i}-2)$ by Proposition \ref{Pextpi1sigma} (1). We deduce  (\ref{EzetaFpsigma}) is bijective and (1) follows. The second part of (2) follows  from (\ref{Eparaps}).
\end{proof}
Finally, we  have similarly as in Proposition \ref{PFGfil}:
\begin{proposition}
Let $\sF$ and $\sG$ be as in Proposition \ref{PFGfil}, there is a natural exact sequence
$0 \ra \pi_{\alg}(\phi, \textbf{h}) \ra \pi_{\sF,\sigma}(\phi, \textbf{h}) \oplus \pi_{\sG,\sigma}(\phi, \textbf{h}) \ra \pi_{1,\sigma}(\phi, \textbf{h}) \ra 0$. 
Consequently, we have a natural exact sequence
\begin{multline}\label{EFGsigma}
	0 \lra \Ext^1_{\sigma, g'}(\pi_{\alg}(\phi, \textbf{h}), \pi_1(\phi,\textbf{h})) \lra \\
	\Ext^1_{\sigma, \sF}(\pi_{\alg}(\phi, \textbf{h}), \pi_1(\phi,\textbf{h}))  \oplus \Ext^1_{\sigma,\sG}(\pi_{\alg}(\phi, \textbf{h}), \pi_1(\phi,\textbf{h})) \\ \lra \Ext^1_{\sigma}(\pi_{\alg}(\phi, \textbf{h}), \pi_1(\phi,\textbf{h}))  \lra 0.
\end{multline}
\end{proposition}

\subsection{Hodge parameters in $\GL_n(K)$-representations}\label{S32}
\subsubsection{Construction and properties}\label{S321}
In this section, we associate to $D\in \Phi\Gamma_{\nc}(\phi, \textbf{h})$ a locally $\Q_p$-analytic representation $\pi_{\min}(D)$ of $\GL_n(K)$ over $E$, which determines those Hodge parameters of $D$ reinterpreted in \S~\ref{S22} (hence determines $D$ when $K=\Q_p$).

Consider the following composition (see (\ref{Ekappaw000}) and (\ref{Etphih}) for the maps)
\begin{equation}\label{ERT}
\oplus_{w\in S_n} \ol{\Ext}^1_w(D,D) \xlongrightarrow[\sim]{(\kappa_w)} \oplus_{w\in S_n}  \Hom(T(K),E)  {\buildrel {t_{\phi,\textbf{h}}} \over \twoheadlongrightarrow}  \Ext^1_{\GL_n(K)}(\pi_{\alg}(\phi, \textbf{h}), \pi_1(\phi, \textbf{h})).
\end{equation}
The following theorem is crucial for the paper. 
\begin{theorem}\label{TtD}
The natural surjection $\oplus_{w\in S_n} \ol{\Ext}^1_w(D,D) \twoheadrightarrow \ol{\Ext}^1(D,D)$ (cf. Proposition \ref{Pinffern}) factors through (\ref{ERT}), i.e. there exists a unique map 
\begin{equation}\label{EtD000}t_D: \Ext^1_{\GL_n(K)}\big(\pi_{\alg}(\phi, \textbf{h}), \pi_1(\phi, \textbf{h})\big)
\twoheadlongrightarrow \ol{\Ext}^1(D,D)
\end{equation} such that $\oplus_{w\in S_n} \ol{\Ext}^1_w(D,D) \twoheadrightarrow \ol{\Ext}^1(D,D)$ is equal to $t_D$ composed with (\ref{ERT}).
\end{theorem}
\begin{proof}
We prove the theorem by induction on $n$. It is trivial for $n=1$. Suppose it holds for $n-1$. As in \S~\ref{S24}, let $D_1\in \Phi\Gamma_{\nc}(\phi^1, \textbf{h}^1)$ (resp. $C_1\in \Phi\Gamma_{\nc}(\phi^1, \textbf{h}^2)$) be the corresponding saturate $(\varphi, \Gamma)$-submodule (resp. quotient) of $D$ (where $\phi^1:=\phi_1 \boxtimes \cdots \boxtimes \phi_{n-1}$, $\textbf{h}^1:=(\textbf{h}_1, \cdots, \textbf{h}_{n-1})$, and $\textbf{h}^2:=(\textbf{h}_2, \cdots, \textbf{h}_{n})$), and $\sF$, $\sG$ be the associated filtrations on $D$.  For $w\in S_{n-1}$, the  following diagram commutes (cf. (\ref{Eparatri})):
\begin{equation*}
	\begin{tikzcd}
	&	\Hom(T(K),E) \arrow[r, "\sim"]  &	\Hom(T_1(K),E) \times \Hom(K^{\times},E) \\
		& \overline{\Ext}^1_w(D,D) \arrow[u, "\kappa_w", "\sim"'] \arrow[r, "\sim"] \arrow[d, hook]   & \overline{\Ext}^1_w(D_1,D_1) \times \Hom(K^{\times}, E) \arrow[u, "\kappa_{w}", "\sim"']  \arrow[d, hook] \\
		&	\overline{\Ext}^1_{\sF}(D,D)  \arrow[r, "\sim", "\kappa_{\sF}"']  &  \overline{\Ext}^1(D_1,D_1) \times \Hom(K^{\times},E).
	\end{tikzcd}
\end{equation*}
By induction hypothesis, the map $\oplus_{w\in S_{n-1}}\ol{\Ext}^1_w(D_1,D_1) \twoheadrightarrow \ol{\Ext}^1(D_1,D_1)$ factors through the following map (defined similarly as in (\ref{ERT}))
\begin{equation*}
	t_{\phi^1,\textbf{h}^1}:  \oplus_{w\in S_{n-1}}\ol{\Ext}^1_w(D_1,D_1)
	\twoheadlongrightarrow \Ext^1_{\GL_{n-1}(K)}(\pi_{\alg}(\phi^1,\textbf{h}^1), \pi_1(\phi^1,\textbf{h}^1)).
\end{equation*}
This together with  Proposition \ref{PextpiFP} and (\ref{EkappaFp}) imply   $\oplus_{w\in S_{n-1}}\overline{\Ext}^1_w(D,D) \twoheadrightarrow \overline{\Ext}^1_{\sF}(D,D) $ ($ \hookrightarrow \overline{\Ext}^1(D,D)$) factors through 
\begin{equation}\label{EtFPD}t_{\sF,D}: \Ext^1_{\sF}(\pi_{\alg}(\phi, \textbf{h}), \pi_1(\phi, \textbf{h}))\twoheadlongrightarrow \ol{\Ext}^1_{\sF}(D,D) \hooklongrightarrow \ol{\Ext}^1(D,D).
	\end{equation} Let $S_{n-1}':=\{w\in S_n\ |\ w(n)=1\}$, that is a subset of $S_n$ of cardinality $(n-1)!$. By a similar discussion with $D_1$ replaced by $C_1$,  the map $\oplus_{w\in S_{n-1}'}\overline{\Ext}^1_w(D,D) \twoheadrightarrow \overline{\Ext}^1_{\sG}(D,D) \hookrightarrow \overline{\Ext}^1(D,D)$ factors through $t_{\sG,D}: \Ext^1_{\sG}(\pi_{\alg}(\phi, \textbf{h}), \pi_1(\phi, \textbf{h}))\twoheadrightarrow \ol{\Ext}^1_{\sG}(D,D)\hookrightarrow \ol{\Ext}^1(D,D)$. By (\ref{Eintg'}) and (\ref{Eint1}) (with $g$ replaced by $g'$), the following diagram commutes (see (\ref{Einjg'Fp}) for the injections from $\Ext^1_{g'}$)
\begin{equation*}
	\begin{tikzcd}
		\Ext^1_{g'}(\pi_{\mathrm{alg}}(\phi, \textbf{h}), \pi_1(\phi, \textbf{h})) 
		\arrow[r, hook]  
		\arrow[d, hook]  
		& \Ext^1_{\sF}(\pi_{\mathrm{alg}}(\phi, \textbf{h}), \pi_1(\phi, \textbf{h})) 
		\arrow[d, "t_{\sF,D}"] 
		\\
		\Ext^1_{\sG}(\pi_{\mathrm{alg}}(\phi, \textbf{h}), \pi_1(\phi, \textbf{h})) 
		\arrow[r, "t_{\sG,D}"] 
		& \overline{\Ext}^1(D,D).
	\end{tikzcd}
\end{equation*}
Hence by the second exact sequence in Proposition \ref{PFGfil}, the composition 
\begin{multline*}
	\Ext^1_{\sF}(\pi_{\alg}(\phi, \textbf{h}), \pi_1(\phi, \textbf{h})) \oplus \Ext^1_{\sG}(\pi_{\alg}(\phi, \textbf{h}), \pi_1(\phi, \textbf{h})) \\
\twoheadlongrightarrow\overline{\Ext}^1_{\sF}(D,D) \oplus \overline{\Ext}^1_{\sG}(D,D) \lra \overline{\Ext}^1(D,D)
\end{multline*}
factors though a map 
\begin{equation} \label{Etdfirst}t_D: \Ext^1_{\GL_n(K)}(\pi_{\alg}(\phi, \textbf{h}), \pi_1(\phi, \textbf{h})) \twoheadlongrightarrow \overline{\Ext}^1(D,D).\end{equation}
Next, we show $t_D$ satisfies the property in the theorem. By construction, the map
	$\oplus_{w\in S_{n-1} \cup S_{n-1}'} \overline{\Ext}^1_w(D,D) \ra \overline{\Ext}^1(D,D)$
factors through $t_D$. It suffices to show for the other $w\in S_n$,  $\overline{\Ext}^1_w(D,D)\ra \overline{\Ext}^1(D,D)$ also factors as 
\begin{equation}\label{Efactor1}
	\overline{\Ext}^1_w(D,D) \xrightarrow[\sim]{\kappa_w} \Hom(T(K),E)\xrightarrow{\zeta_w}  \Ext^1_{\GL_n}(\pi_{\alg}(\phi, \textbf{h}), \pi_1(\phi, \textbf{h}))  \xrightarrow{(\ref{Etdfirst})}\overline{\Ext}^1(D,D).
\end{equation} Suppose hence  $w(n)=i$ with $1<i<n$. We have 
\begin{equation}\label{Evschara}\Hom(T(K),E)\cong \oplus_{j=1}^{n-1} \Hom(Z_{j}(K),E) \oplus \Hom(Z(K),E),\end{equation} where $Z_{j}\subset T$ is the  centre of the Levi subgroup $L_j$ (containing $T$) of the  maximal parabolic subgroup $P_j$ (with $j\notin \sW_{P_j}$).  For any $j=1, \cdots, n-1$, $\kappa_w^{-1}(\Hom(Z_j(K),E))\subset \ol{\Ext}^1_{\sF_{P_j},g'}(D,D)$ (cf. Corollary \ref{Cint}), where $\sF_{P_j}$ is the $P_j$-filtration associated to the $B$-filtration $\sT_w$ (such that $\sT_w$ is compatible with $\sF_{P_j}$). 
Let $w_j$ be an element in the Weyl group of $L_{j}$ such that $w_j(i)=1$ or $w_j(i)=n$ (whose existence is clear).  By Corollary \ref{Cint} (2) and (\ref{EintP2}), we have a commutative diagram
\begin{equation*}
	\begin{tikzcd}
	&	\ol{\Ext}^1_{\sF_{P_j},g'}(D,D) \arrow[r, "\kappa_w", "\sim"'] 
	\arrow[d, equal] & \Hom_{P_j,g'}(T(K),E) \arrow[r, hook, "\zeta_w"] \arrow[d, "\sim", "w_j"']  &\Ext^1_{\GL_n}(\pi_{\alg}(\phi, \textbf{h}), \pi_1(\phi, \textbf{h}))  \arrow[d, equal]\\
	&	\ol{\Ext}^1_{\sF_{P_j},g'}(D,D) \arrow[r, "\kappa_{w_jw}", "\sim"'] &\Hom_{P_j,g'}(T(K),E) \arrow[r, hook, "\zeta_{w_jw}"] & \Ext^1_{\GL_n}(\pi_{\alg}(\phi, \textbf{h}), \pi_1(\phi, \textbf{h})).
	\end{tikzcd}
\end{equation*}
It is clear that $w_jw\in S_{n-1} \cup S_{n-1}'$, hence the map $\overline{\Ext}^1_{w_jw}(D,D) \ra \overline{\Ext}^1(D,D)$ is equal to $t_D \circ (\zeta_{w_jw} \circ \kappa_{w_jw})$.  In particular, its restriction to $	\ol{\Ext}^1_{\sF_{P_j},g'}(D,D)$ is equal to $t_D\circ (\zeta_{w_jw} \circ \kappa_{w_jw})=t_D \circ (\zeta_w \circ \kappa_w)$ by the above commutative diagram. As $\ol{\Ext}^1_w(D,D)$ is spanned by $\ol{\Ext}^1_{\sF_{P_j,g'}}(D,D)$ and $\Hom(Z(K),E)$ (e.g. using (\ref{Evschara})), we obtain the factorisation as in (\ref{Efactor1}). This concludes the proof.
\end{proof}
\begin{remark}\label{RtD}
(1) By  comparing dimensions (using Proposition \ref{PDef1}, Proposition \ref{Pextpi1} (1)), we have 
$\dim_E \Ker(t_D)=(2^n-\frac{n(n+1)}{2}-1)d_K$.

(2) The same argument holds with $\pi_1(\phi, \textbf{h})$ replaced by $\pi(\phi, \textbf{h})$ (with the same $t_D$ under the isomorphism (\ref{Epi1pi})).
\end{remark}
The following lemma is clear.
\begin{lemma}\label{Lzerointer}
For any $w\in S_n$, $\Ker(t_D) \cap \Ext^1_w(\pi_{\alg}(\phi, \textbf{h}), \pi_1(\phi, \textbf{h}))=0$.
\end{lemma}
Let $\pi_{\min}(D)$ (resp. $\pi_{\fss}(D)$) be the extension of $\Ker(t_D) \otimes_E\pi_{\alg}(\phi, \textbf{h})$ ($\cong \pi_{\alg}(\phi, \textbf{h})^{ \oplus (2^n-\frac{n(n+1)}{2}-1)d_K}$) by  $\pi_1(\phi, \textbf{h})$ (resp. $\pi(\phi, \textbf{h})$) associated to $\Ker(t_D)$ (cf. \S~\ref{S311}, \ see \ also \ Remark \ \ref{RtD} \ (2)). \ Note \ that \ as 
\ $\End_{\GL_n(K)}(\pi(\phi, \textbf{h}))\xrightarrow{\sim} \End_{\GL_n(K)}(\pi_1(\phi, \textbf{h}))\xrightarrow{\sim}\End_{\GL_n(K)}(\pi_{\alg}(\phi, \textbf{h})\cong E,$ either $\pi_{\min}(D)$ or $\pi_{\fss}(D)$ determines $\Ker(t_D)$. We have 
\begin{equation}\label{EpifsD}
\pi_{\fss}(D)\cong \pi_{\min}(D) \oplus_{\pi_1(\phi, \textbf{h})} \pi(\phi, \textbf{h}).
\end{equation}
In the sequel, we will mainly work with $\pi_{\min}(D)$, noting that most of the  statements generalize to $\pi_{\fss}(D)$ without effort. We have an exact sequence
\begin{multline}
0 \lra \Hom_{\GL_n(K)}(\pi_{\alg}(\phi, \textbf{h}), \Ker(t_D) \otimes_E\pi_{\alg}(\phi, \textbf{h}))\\  \lra \Ext^1_{\GL_n(K)}(\pi_{\alg}(\phi, \textbf{h}), \pi_1(\phi, \textbf{h})) \xlongrightarrow{f_D} \Ext^1_{\GL_n(K)}(\pi_{\alg}(\phi, \textbf{h}), \pi_{\min}(D)).
\end{multline}
By Lemma \ref{Lextps} (2), one sees the last map $f_D$ is surjective. For a $P$-filtration $\sF_P$ on $D$, we denote by $\Ext^1_{\sF_P}(\pi_{\alg}(\phi, \textbf{h}), \pi_{\min}(D))$ the image of $\Ext^1_{\sF_P}(\pi_{\alg}(\phi, \textbf{h}), \pi_1(\phi, \textbf{h}))$ under $f_D$, and write $\Ext^1_w$ for $\Ext^1_{\sT_w}$. Denote by $\Ext^1_g(\pi_{\alg}(\phi, \textbf{h}), \pi_{\min}(D))$ the image of $\Ext^1_g(\pi_{\alg}(\phi, \textbf{h}), \pi_1(\phi, \textbf{h}))$ under $f_D$. 
\begin{corollary}
	The map $t_D$ induces 
	$\Ext^1_g(\pi_{\alg}(\phi, \textbf{h}), \pi_{\min}(D))\xlongrightarrow{\sim}   \ol{\Ext}^1_g(D,D) $
	and $\Ext^1_w(\pi_{\alg}(\phi,\textbf{h}),\pi_{\min}(D))\xrightarrow{\sim} \ol{\Ext}^1_w(D,D)$ for all $w\in S_n$.
\end{corollary}
\begin{proof}
	By Lemma \ref{Lzerointer}, $\Ext^1_*(\pi_{\alg}(\phi, \textbf{h}), \pi_1(\phi, \textbf{h}))\xrightarrow{\sim}\Ext^1_*(\pi_{\alg}(\phi,\textbf{h}),\pi_{\min}(D))$ for $*\in \{w,g\}$. The corollary then follows from the definition of $t_D$,  (\ref{Ekappaw000}), (\ref{Eisomzetaw}) and Proposition \ref{Pextalg} (1).
\end{proof}The following corollary is a direct consequence of Theorem \ref{TtD}.
\begin{corollary}\label{CtD} Let \ $D\in \Phi\Gamma_{\nc}(\phi, \textbf{h})$. \ The \ representation \ $\pi_{\min}(D)$ \ is \ the \ unique \ extension \ of $\pi_{\alg}(\phi,\textbf{h})^{\oplus (2^n-\frac{n(n+1)}{2}-1)d_K}$ by $\pi_1(\phi, \textbf{h})$ satisfying the following properties:

(1)  $\soc_{\GL_n(K)}\pi_{\min}(D)\cong \pi_{\alg}(\phi, \textbf{h}),$ and
	$\soc_{\GL_n(K)}\big(\pi_{\min}(D)/ \pi_{\alg}(\phi, \textbf{h})\big)\cong  \soc_{\GL_n(K)}\big(\pi_1(\phi, \textbf{h})/ \pi_{\alg}(\phi, \textbf{h})\big).$

(2) There is a bijection
	\begin{equation*}
		t_D:\Ext^1_{\GL_{n}(K)}(\pi_{\alg}(\phi, \textbf{h}), \pi_{\min}(D))	\xlongrightarrow{\sim} 	\ol{\Ext}^1(D,D) 	
	\end{equation*}
	which is compatible with trianguline deformations, i.e. for $w\in S_n$, the composition $	\Hom(T(K),E) \xrightarrow[\sim]{\zeta_w} \Ext^1_w (\pi_{\alg}(\phi, \textbf{h}), \pi_{\min}(D)){\buildrel t_D \over \hookrightarrow} 	\ol{\Ext}^1(D,D) $ coincides with  $\Hom(T(K),E) \xrightarrow[\sim]{\kappa_w^{-1}} \ol{\Ext}^1_w(D,D)\hookrightarrow \ol{\Ext}^1(D,D)$. 

\end{corollary}

Let $\chi_D$ be the character $z^{|\lambda|}|\cdot|_K^{\frac{n(n-1)}{2}} \prod_{i=1}^n \phi_i$ of $K^{\times}$ with $|\lambda|=\sum_{\substack{i=1,\cdots, n\\ \sigma\in \Sigma_K}} \lambda_{i,\sigma}$. We have $\wedge^n D\cong \cR_{K,E}(\chi_D \varepsilon^{-\frac{n(n-1)}{2}})$. For an integral weight $\mu$ of $\ft_{\Sigma_K}$, let $\xi_{\mu}$ be the central character of $\text{U}(\gl_{n,\Sigma_K})$ acting on  $L(\mu)$.
\begin{proposition}
The representation $\pi_{\min}(D)$ has central character $\chi_D$ and infinitesimal character $\xi_{\lambda}$. 
\end{proposition}
\begin{proof}
We only prove the statement for the infinitesimal character,  the central character being similar. Let $\cZ_K$ be the centre of $\text{U}(\gl_{n,\Sigma_K})$. Recall we have the Harish-Chandra isomorphism $\HC: \cZ_K \xrightarrow{\sim} \text{U}(\ft_{\Sigma_K})^{\sW_{n,K}}$, where $\sW_{n,K}$ is the Weyl group of $\Res^K_{\Q_p} \GL_n$, isomorphic to $S_n^{d_K}$, and where we normalize the map such that a weight $\mu$ of $\ft_{\Sigma_K}$, seen as a character of $\text{U}(\ft_{\Sigma_K})^{\sW_{n,K}}$, corresponds to $\xi_{\mu-\theta^{[K:\Q_p]}}$ of $\cZ_K$ (recalling $\theta^{[K:\Q_p]}=(0,\cdots, 1-n)_{\sigma\in \Sigma_K}$). 
In particular, the weight $\textbf{h}$ corresponds to $\xi_{\lambda}$. Let $X_{\xi_{\lambda}}$ (resp. $X_{\textbf{h}}$) be the tangent space  of $\cZ_K$ (resp. $\text{U}(\ft_{\Sigma_K})$) at $\xi_{\lambda}$ (resp. at $\textbf{h}$), i.e. $X_{\xi_{\lambda}}=\{f:\cZ_K \twoheadrightarrow E[\epsilon]/\epsilon^2\ |\ f\equiv \xi_{\lambda} \pmod{\epsilon}\}$ and similarly for $X_{\textbf{h}}$. The map $\HC$ induces a bijection  $\HC: X_{\textbf{h}} \xrightarrow{\sim} X_{\xi_{\lambda}}$ (noting the injection $\text{U}(\ft_{\Sigma_K})^{\sW_{n,K}}\hookrightarrow \text{U}(\ft_{\Sigma_K})$ induces bijections on tangent spaces, e.g. by the explicit description of the invariants $\text{U}(\ft_{\Sigma_K})^{\sW_{n,K}}$ as a polynomial algebra).

For $\widetilde{D}\in \Ext^1(D,D)$, the Sen weights of $\widetilde{D}$ (over $E[\epsilon]/\epsilon^2$) have the form $(h_{i,\sigma}+a_{i,\sigma} \epsilon)_{\substack{\sigma\in \Sigma_K\\ i=1,\cdots, n}}$. We obtain hence an $E$-linear map $\Ext^1(D,D) \ra X_{\textbf{h}}$, $\widetilde{D} \mapsto (a_{i,\sigma})$. The map sends $\Ext^1_g(D,D)$ hence $\Ext^1_0(D,D)$ to zero, thus induces an $E$-linear map $d_{\Sen}:	\ol{\Ext}^1(D,D) \ra X_{\textbf{h}}$. For $\widetilde{\pi}\in  \Ext^1_{\GL_n(K)}(\pi_{\alg}(\phi, \textbf{h}), \pi_1(\phi, \textbf{h}))$ equipped with the natural $E[\epsilon]/\epsilon^2$-action, and $a\in \cZ_K$, the operator $(a-\xi_{\lambda}(a))$ (on $\widetilde{\pi}$) annihilates $\pi_1(\phi,\textbf{h})$ hence induces a $\GL_n(K)$-equivariant map $\widetilde{\pi}\twoheadrightarrow \pi_{\alg}(\phi, \textbf{h}) \ra \widetilde{\pi}$. As $\Hom(\pi_{\alg}(\phi, \textbf{h}), \pi_1(\phi, \textbf{h}))\cong E$,  this map is equal to $\alpha \epsilon$ for some $\alpha\in E$ (depending on $a$). We deduce $\cZ_K$ acts on $\widetilde{\pi}$ via a character over $E[\epsilon]/\epsilon^2$, which corresponds to an element in $X_{\xi_{\lambda}}$. We obtain hence  an $E$-linear map:
$	d_{\inf}: \Ext^1_{\GL_n(K)}(\pi_{\alg}(\phi, \textbf{h}), \pi_1(\phi, \textbf{h})) \ra X_{\xi_{\lambda}}$.
The proposition (for the infinitesimal character) will be a direct consequence of the commutativity of the diagram:
\begin{equation}\label{EHC}
	\begin{CD}
		\Ext^1_{\GL_n(K)}(\pi_{\alg}(\phi, \textbf{h}), \pi_1(\phi, \textbf{h})) @> d_{\inf}   >> 	X_{\xi_{\lambda}}	\\
		@V t_DV V @V  \HC^{-1} VV \\
		\ol{\Ext}^1(D,D) @> d_{\Sen}>>	X_{\textbf{h}}.
	\end{CD}
\end{equation}
By the construction of $t_D$, it suffices to show for all $w\in S_n$, the following diagram commutes
\begin{equation}\label{Einfcomm00}
	\begin{CD}
		\ol{\Ext}^1_w(D,D)@> d_{\Sen}>>	X_{\textbf{h}}\\
		@V \sim V V @V  \HC VV \\
		\Ext^1_{w}(\pi_{\alg}(\phi, \textbf{h}), \pi_1(\phi, \textbf{h})) @> d_{\inf}   >> 	X_{\xi_{\lambda}}.
	\end{CD}
\end{equation}
Let $\psi\in \Hom(T(K),E)$, and $\widetilde{D}\in \Ext^1_w(D,D)$ be a $(\varphi, \Gamma)$-module over $\cR_{K, E[\epsilon]/\epsilon^2}$ of trianguline parameter $\widetilde{\delta}:=w(\phi) z^{\textbf{h}}(1+\psi \epsilon)$. Let  $d\widetilde{\delta}: \text{U}(\ft_{\Sigma_K})\ra E[\epsilon]/\epsilon^2$ be the morphism induced by $\widetilde{\delta}$ by derivation. Then  $d_{\Sen}(\widetilde{D})=d \widetilde{\delta}$. By (\ref{Ekappaw}) (\ref{Ezetaw}) and Remark \ref{Rind1}, the image $\widetilde{\pi}$ of $\widetilde{D}$ under the left vertical map in (\ref{Einfcomm00}) satisfies $I_{B^-(K)}^{\GL_n(K)} (\widetilde{\delta} \eta z^{-\theta^{[K:\Q_p]}}) \hookrightarrow \widetilde{\pi}$. By \cite[(0.4)]{Em2}, we have $\widetilde{\delta}\eta \delta_Bz^{-\theta^{[K:\Q_p]}}\hookrightarrow J_B(\widetilde{\pi})\hookrightarrow  \widetilde{\pi}^{\fn}$, where $\fn$ is the Lie algebra of $N$. We see $\cZ_K$ acts on the image of the injection via $d\widetilde{\delta}\circ \HC$. Moreover, the composition  $\widetilde{\delta}\eta \delta_Bz^{-\theta^{[K:\Q_p]}}\hookrightarrow J_B(\widetilde{\pi}) \hookrightarrow  \widetilde{\pi} \twoheadrightarrow \pi_{\alg}(\phi, \textbf{h})$  is non-zero, since $\dim_E \Hom_{T(K)}\big(w(\phi)z^{\textbf{h}} \eta z^{-\theta^{[K:\Q_p]}}\delta_B,J_B(\pi_1(\phi, \textbf{h}))\big)\cong E$. We deduce the subrepresentation $\widetilde{\pi}[\cZ_K=d\widetilde{\delta}\circ \HC]$ strictly contains $\pi_1(\phi, \textbf{h})$ hence is equal to $\widetilde{\pi}$ itself. So $d_{\inf}(\widetilde{\pi})=d\widetilde{\delta}\circ \HC$ and (\ref{Einfcomm00}) commutes. This concludes the proof. 
\end{proof}
We next discuss the compatibility of $t_D$ (and $\pi_{\min}(D)$) with parabolic inductions. Let $P\supset B$ be a standard parabolic subgroup of $\GL_n$ with $L_P$ equal to $\diag(\GL_{n_1}, \cdots, \GL_{n_r})$. 
Let $\sF_P$ be a $P$-filtration of $D$, $M_i:=\gr_i \sF_P$, which is a $(\varphi, \Gamma)$-module of rank $n_i$, for $i=1, \cdots, r$.  Recall we have defined $\Ext^1_{\sF_P}$ for both $(\varphi, \Gamma)$-modules (cf. the discussion above Proposition \ref{Ppara1}), and $\GL_n(K)$-representations (cf. \S~\ref{S313}).
\begin{proposition}\label{PtDFp}
The map  $t_D$ restricts to  a surjection
\begin{equation*}
	t_{D,\sF_P}: 	\Ext^1_{\sF_P}(\pi_{\alg}(\phi, \textbf{h}), \pi_1(\phi, \textbf{h}))  \twoheadlongrightarrow \ol{\Ext}^1_{\sF_P}(D,D).
\end{equation*}Moreover, the following diagram commutes
\begin{equation}\label{EtDFp}
	\begin{CD}
		\prod_{i=1}^r \Ext^1_{\GL_{n_i}(K)}(\pi_{\alg}(\phi_{\sF_P,i}, \textbf{h}^i), \pi_1(\phi_{\sF_P,i},\textbf{h}^i)) @> (t_{M_i}) >> \prod_{i=1}^r \ol{\Ext}^1_{(\varphi, \Gamma)}(M_i,M_i) \\
		@V \sim V (\ref{EzetaFp}) V @V \sim V (\ref{EkappaFp}) V\\
		\Ext^1_{\sF_P}(\pi_{\alg}(\phi, \textbf{h}), \pi_1(\phi, \textbf{h})) @> t_{D,\sF_P} >> \ol{\Ext}^1_{\sF_P}(D,D).
	\end{CD}
\end{equation}
In particular, the parabolic induction (\ref{EzetaFp}) induces a natural isomorphism
\begin{equation}\label{Eparameterpara}
	\oplus_{i=1}^r \Ker(t_{M_i}) \xlongrightarrow{\sim} \Ker(t_{D,\sF_P})=\Ker(t_D)\cap 	\Ext^1_{\sF_P}(\pi_{\alg}(\phi, \textbf{h}), \pi_1(\phi, \textbf{h})).
\end{equation}
\end{proposition}
\begin{proof}By Corollary \ref{CkappaFp} (2), Proposition \ref{Pinffern},  $\ol{\Ext}^1_{\sF_P}(D,D)$ can be spanned by $\ol{\Ext}^1_w(D,D)$  for $\sT_w$ compatible with $\sF_P$.  Together with (\ref{EinffernFp}), the first part follows. The commutativity of the diagram follows from (\ref{Eparatri}) and (\ref{Eparaps}).
\end{proof}
\begin{remark}\label{RparatD}
Let $\pi_{\min,\sF_P}(D)\subset \pi_{\min}(D)$ be the extension of $\Ker(t_{D,\sF_P}) \otimes_E\pi_{\alg}(\phi, \textbf{h})\cong  \pi_{\alg}(\phi, \textbf{h})^{\oplus \sum_{i=1}^r(d_K(2^{n_i}-\frac{n_i(n_i+1)}{2}-1))}$ by $\pi_1(\phi, \textbf{h})$ associated to $\Ker(t_{D,\sF_P})$. By Proposition \ref{PtDFp}, $\pi_{\min, \sF_P}(D)$ is the maximal subrepresentation of $\pi_{\min}(D)$ which comes from the push-forward of extensions of $\pi_{\alg}(\phi, \textbf{h})$ by $\pi_{\sF_P}(\phi, \textbf{h})$ via $\pi_{\sF_P}(\phi ,\textbf{h})\hookrightarrow \pi_1(\phi, \textbf{h})$. We have   $I_{P^-(K)}^{\GL_n(K)}((\widehat{\boxtimes}_{i=1}^r \pi_{\min}(M_i)) \otimes_E \varepsilon^{-1} \circ \theta^P)\hookrightarrow \pi_{\min, \sF_P}(D)$.  Moreover, (\ref{EtDFp}) induces a commutative diagram
\begin{equation*}
	\begin{CD}
		\prod_{i=1}^r \Ext^1_{\GL_{n_i}(K)}(\pi_{\alg}(\phi_{\sF_P,i}, \textbf{h}^i), \pi_{\min}(M_i)) @> (t_{M_i}) > \sim > \prod_{i=1}^r \ol{\Ext}^1_{(\varphi, \Gamma)}(M_i,M_i) \\
		@V \sim VV @V \sim VV\\
		\Ext^1_{\sF_P}(\pi_{\alg}(\phi, \textbf{h}), \pi_{\min, \sF_P}(D)) @> t_{D,\sF_P} > \sim > \ol{\Ext}^1_{\sF_P}(D,D),
	\end{CD}
\end{equation*}
where $\Ext^1_{\sF_P}(\pi_{\alg}(\phi, \textbf{h}), \pi_{\min, \sF_P}(D))$ is the image of $\Ext^1_{\sF_P}(\pi_{\alg}(\phi, \textbf{h}), \pi_1(\phi, \textbf{h}))$ via the push-forward map, and where the left vertical map is obtained in a similar way as (\ref{EzetaFp}).

\end{remark}

Let $\sigma\in \Sigma_K$. By Proposition \ref{Pextpi1sigma} and  Corollary \ref{Cwtsigma}, (\ref{ERT})  restricts to a surjection
\begin{equation}\label{EDsigma000}\oplus_{w\in S_n}\ol{\Ext}^1_{\sigma, w}(D,D) \twoheadlongrightarrow \Ext^1_{\sigma}(\pi_{\alg}(\phi, \textbf{h}), \pi_1(\phi, \textbf{h})).
	\end{equation}
\begin{corollary} \label{CtDsigma} Let $D\in \Phi\Gamma_{\nc}(\phi, \textbf{h})$.

(1) The map (\ref{Einf2}), quotienting by $\Ext^1_0(D,D)$, factors through (\ref{EDsigma000}) and the restriction of $t_D$:
\begin{equation*}
	t_{D,\sigma}:	\Ext^1_{\sigma}(\pi_{\alg}(\phi, \textbf{h}), \pi_1(\phi, \textbf{h})) \twoheadlongrightarrow \ol{\Ext}^1_{\sigma}(D,D).
\end{equation*}

(2) Let $P$ be a standard parabolic subgroup and $\sF_P$ be a $P$-filtration on $D$. Let $t_{D,\sF_P, \sigma}$ be the restriction of  $t_{D, \sigma}$ to $\Ext^1_{\sigma, \sF_P}(\pi_{\alg}(\phi, \textbf{h}), \pi_1(\phi, \textbf{h}))$, we have a commutative diagram 
\begin{equation*}
	\begin{CD}
		\prod_{i=1}^r \Ext^1_{\sigma}(\pi_{\alg}(\phi_{\sF_P,i}, \textbf{h}^i), \pi_1(\phi_{\sF_P,i},\textbf{h}^i)) @> (t_{M_{i},\sigma}) >> \prod_{i=1}^r \ol{\Ext}^1_{\sigma}(M_i,M_i) \\
		@V \sim V (\ref{EzetaFpsigma}) V  @V \sim V (\ref{EkappaFP11})V\\
		\Ext^1_{\sigma,\sF_P}(\pi_{\alg}(\phi, \textbf{h}), \pi_1(\phi, \textbf{h})) @> t_{D,\sF_P,\sigma}>> \ol{\Ext}^1_{\sigma, \sF_P}(D,D).
	\end{CD}
\end{equation*}
In particular, (\ref{EzetaFpsigma}) induces $\oplus_{i=1}^r \Ker t_{M_{i,\sigma}} \xrightarrow{\sim} \Ker t_{D, \sF_P,\sigma} $.
\end{corollary}
\begin{proof}
(1) follows by Theorem \ref{TtD} and Corollary \ref{Cinf2}. (2) follows from   (\ref{EtDFp}).
\end{proof}
\begin{remark}\label{Rdimsigma}
(1)	 It is clear that  $\Ker(t_{D,\sigma})=\Ker(t_D) \cap 	\Ext^1_{\sigma}(\pi_{\alg}(\phi, \textbf{h}), \pi_1(\phi, \textbf{h}))$. By Lemma \ref{Lsigma1} and (\ref{Ext0dim}), $\dim_E\ol{\Ext}^1_{\sigma}(D,D)=n+\frac{n(n+1)}{2}$. Together with  Proposition \ref{Pextpi1sigma} (1), we then deduce $\dim_E \Ker(t_{D,\sigma})=2^n-\frac{n(n+1)}{2}-1$.

(2) Recall we have $\Ext^1_{\sigma}(\pi_{\alg}(\phi, \textbf{h}), \pi_{1,\sigma}(\phi, \textbf{h}))\xrightarrow{\sim} \Ext^1_{\sigma}(\pi_{\alg}(\phi, \textbf{h}), \pi_1(\phi, \textbf{h}))$ (see the discussion below Remark \ref{Rautosigma}). We  view hence $\Ker(t_{D,\sigma})$ as subspace of $\Ext^1_{\sigma}(\pi_{\alg}(\phi, \textbf{h}), \pi_{1,\sigma}(\phi, \textbf{h}))$. Set $\pi_{\min}(D)_{\sigma}$ to be the extension of $\Ker(t_{D,\sigma}) \otimes_E \pi_{\alg}(\phi, \textbf{h})$ by $\pi_{1,\sigma}(\phi, \textbf{h})$. Then $\pi_{\min}(D)_{\sigma}$ is just  the maximal $\ug_{\Sigma_K\setminus \{\sigma\}}$-algebraic subrepresentation of $\pi_{\min}(D)$. 

\end{remark}

\begin{corollary}\label{Cparadecom}
We have $\oplus_{\sigma\in \Sigma_K} \Ker(t_{D,\sigma})\xrightarrow{\sim} \Ker(t_D)$. Consequently, $\pi_{\min}(D) \cong \bigoplus_{\pi_{\alg}(\phi, \textbf{h})}^{\sigma\in \Sigma_K}\pi_{\min}(D)_{\sigma}$.
\end{corollary}
\begin{proof}
The second part follows from the first one (see also Remark \ref{Rdimsigma} (2)). Consider the induced maps
\begin{equation*}
\Ext^1_{\GL_n}(\pi_{\alg}(\phi, \textbf{h}), \pi_1(\phi, \textbf{h}))/\Ext^1_{g}(\pi_{\alg}(\phi, \textbf{h}), \pi_1(\phi, \textbf{h})) 
	{\buildrel \bar{t}_D \over \twoheadrightarrow} \ol{\Ext}^1(D,D)/\ol{\Ext}^1_g(D,D),
\end{equation*} 
\begin{equation*}
	\Ext^1_{\sigma}(\pi_{\alg}(\phi, \textbf{h}), \pi_1(\phi, \textbf{h}))/\Ext^1_{g}(\pi_{\alg}(\phi, \textbf{h}), \pi_1(\phi, \textbf{h})) 
	{\buildrel 	\bar{t}_{D,\sigma} 	\over \twoheadrightarrow} \ol{\Ext}_{\sigma}^1(D,D)/\ol{\Ext}^1_g(D,D).
\end{equation*} As $\Ext^1_g(\pi_{\alg}(\phi, \textbf{h}),\pi_1(\phi, \textbf{h}))\cap \Ker(t_D)=0$ (cf. Lemma \ref{Lzerointer}), we have isomorphisms $\Ker(t_D)\xrightarrow{\sim} \Ker \bar{t}_D$ and $\Ker(t_{D,\sigma}) \xrightarrow{\sim} \Ker \bar{t}_{D,\sigma}$. Using Proposition \ref{Pextpi1sigma} (3),  $\oplus_{\sigma\in \Sigma_K} \Ker \bar{t}_{D,\sigma}\ra \Ker \bar{t}_D$ is injective. We deduce the natural map 
$	\oplus_{\sigma\in \Sigma_K} \Ker(t_{D,\sigma})\ra  \Ker(t_D)$
is injective. As the both sides have  dimension $(2^n-\frac{n(n+1)}{2}-1)d_K$ by Remark \ref{RtD} (1) and Remark \ref{Rdimsigma} (1),  the map is actually bijective. 
\end{proof}

Let $D_{\sigma}=\fT_{\sigma}(D)$ (cf. (\ref{Ecow})),  
and consider (see Corollary \ref{Ccow2} (1) for the last isomorphism) $$t_{D_{\sigma}}:=\fT_{\sigma} \circ t_{D,\sigma}: \Ext^1_{\sigma}(\pi_{\alg}(\phi, \textbf{h}), \pi_1(\phi, \textbf{h})) \twoheadrightarrow \ol{\Ext}^1_{\sigma}(D,D) \xrightarrow[\sim]{\fT_{\sigma}} \ol{\Ext}^1_{\sigma}(D_{\sigma}, D_{\sigma}).$$
For a $P$-filtration $\sF_P$ on $D_{\sigma}$ (which corresponds to a $P$-filtration on $D$, still denoted by $\sF_P$), let $t_{D_{\sigma}, \sF_P}:=\fT_{\sigma} \circ t_{D, \sF_P, \sigma}$, which is equal to the restriction of $t_{D_{\sigma}}$ to $	\Ext^1_{\sigma,\sF_P}(\pi_{\alg}(\phi, \textbf{h}), \pi_{1,\sigma}(\phi, \textbf{h})) $. It is clear $\Ker t_{D_{\sigma}}=\Ker t_{D,\sigma}$, and $\Ker t_{D_{\sigma}, \sF_P}=\Ker t_{D,\sF_P, \sigma}$. The following corollary is clear. 
\begin{corollary}\label{CtDFpsigma}
(1) The surjective map $\oplus_w \ol{\Ext}^1_{\sigma, w}(D_{\sigma}, D_{\sigma}) \twoheadrightarrow \ol{\Ext}^1_{\sigma}(D_{\sigma}, D_{\sigma})$ (cf. Corollary \ref{Cinf2}) factors through $t_{D_{\sigma}}$ composed with the following composition (which is compatible with (\ref{EDsigma000}), by (\ref{Ecowg}))
\begin{multline*}
\oplus_{w\in S_n} \ol{\Ext}^1_{\sigma,w}(D_{\sigma}, D_{\sigma}) \xlongrightarrow[\sim]{(\kappa_w)} \oplus_{w\in S_n} \Hom_{\sigma}(T(K),E) \\ \xlongrightarrow{(\zeta_w)} \Ext^1_{\sigma}(\pi_{\alg}(\phi, \textbf{h}), \pi_{1,\sigma}(\phi, \textbf{h}))\twoheadlongrightarrow \Ext^1_{\sigma}(\pi_{\alg}(\phi, \textbf{h}), \pi_1(\phi, \textbf{h})).
\end{multline*}
Consequently, $\Ker t_{D_{\sigma}}$ depends only on $D_{\sigma}$.
	
(2) The statements in Corollary \ref{CtDsigma} (2) hold with $D$, $M_i$ replace by $D_{\sigma}$, $M_{i,\sigma}=\fT_{\sigma}(M_i)$. In particular, we have 
	\begin{equation}\label{Eparadecom}
		\oplus_{i=1}^r \Ker t_{M_{i,\sigma}} \xlongrightarrow{\sim} \Ker t_{D_{\sigma}, \sF_P}=\Ker t_{D_{\sigma}} \cap 	\Ext^1_{\sigma,\sF_P}(\pi_{\alg}(\phi, \textbf{h}), \pi_{1,\sigma}(\phi, \textbf{h})).
	\end{equation}
\end{corollary}
Similarly as in Corollary \ref{CtD} and  Remark \ref{RparatD}, we have 
\begin{corollary}
	We have natural isomorphisms (cf. Remark \ref{Rdimsigma} (2))
	\begin{equation*}
		t_{D_{\sigma}}: \Ext^1_{\sigma}(\pi_{\alg}(\phi, \textbf{h}), \pi_{\min}(D)_{\sigma}) \xlongrightarrow[\sim]{t_{D,\sigma}}\ol{\Ext}^1_{\sigma}(D,D)\xlongrightarrow[\sim]{\fT_{\sigma}}\ol{\Ext}^1_{\sigma}(D_{\sigma}, D_{\sigma}).
	\end{equation*}
Moreover, for a $P$-filtration $\sF_P$ on $D$ as in Corollary \ref{CtDsigma} (2), we have a commutative diagram (see Corollary \ref{Ccow2} for the right square):
\begin{equation}\label{EtDFPsigma2}\small
	\begin{tikzcd}
	&\prod_{i=1}^r \Ext^1_{\sigma}(\pi_{\alg}(\phi_{\sF_P,i}, \textbf{h}^i), \pi_{\min}(M_i)_{\sigma}) \arrow[r, "(t_{M_i,\sigma})", "\sim"']  \arrow[d, "\sim"'] & \prod_{i=1}^r \ol{\Ext}^1_{\sigma}(M_{i},M_{i})  \arrow[r, "\fT_{\sigma}", "\sim"'] \arrow[d, "\sim"']& \prod_{i=1}^r \ol{\Ext}^1_{\sigma}(M_{i,\sigma},M_{i,\sigma})  \arrow[d, "\sim"] \\
&\Ext^1_{\sigma,\sF_P}(\pi_{\alg}(\phi, \textbf{h}), \pi_{\min, \sF_P}(D)_{\sigma}) \arrow[r, "t_{D, \sF_P, \sigma}", "\sim"']  & \ol{\Ext}^1_{\sigma, \sF_P}(D,D) \arrow[r, "\fT_{\sigma}", "\sim"']  &  \ol{\Ext}^1_{\sigma, \sF_P}(D_{\sigma},D_{\sigma}) 
	\end{tikzcd}
\end{equation}
where $M_{i,\sigma}=\fT_{\sigma}(M_i)$, $\pi_{\min,\sF_P}(D)_{\sigma}\supset \pi_{1,\sigma}(\phi, \textbf{h})$ is the maximal $\ug_{\Sigma_K\setminus \{\sigma\}}$-algebraic subrepresentation of $\pi_{\min,\sF_P}(D)$ (Rk. \ref{RparatD}), and $\Ext^1_{\sigma, \sF_P}(\pi_{\alg}(\phi, \textbf{h}), \pi_{\min, \sF_P}(D)_{\sigma})$ is the image of $\Ext^1_{\sigma, \sF_P}(\pi_{\alg}(\phi, \textbf{h}), \pi_{1,\sigma}(\phi, \textbf{h}))$ via the natural push-forward map.
\end{corollary}
\begin{theorem} \label{Thodgeauto}Let $D$, $D' \in \Phi\Gamma_{\nc}(\phi, \textbf{h})$, and $\sigma \in \Sigma_K$.  Then $\pi_{\min}(D)_{\sigma}\cong\pi_{\min}(D')_{\sigma}$ if and only if $D_{\sigma}\cong D'_{\sigma}$. Consequently, $\pi_{\min}(D)\cong\pi_{\min}(D')$ if and only if $D_{\sigma} \cong D_{\sigma}'$ for all $\sigma\in \Sigma_K$. 
\end{theorem}
\begin{proof}The second part follows from the first part by Corollary \ref{Cparadecom} and Remark \ref{Rdimsigma} (2).  As $\End_{\GL_n(K)}(\pi_1(\phi, \textbf{h}))\xrightarrow{\sim} \End_{\GL_n(K)}(\pi_{\alg}(\phi, \textbf{h}))\cong E$, $\pi_{\min}(D)_{\sigma}\cong \pi_{\min}(D')_{\sigma}$ if and only if $\Ker(t_{D_{\sigma}})=\Ker(t_{D'_{\sigma}})$.  The ``only if" in the first part follows by Corollary \ref{CtDFpsigma} (1). 
We prove ``if" in the first statement by induction on $n$. The case where $n\leq 2$ is trivial. Indeed, in this case, $\pi_{\min}(D)_{\sigma}$ are all isomorphic, and $D_{\sigma}$ are all isomorphic as well, for $D\in \Phi\Gamma_{\nc}(\phi, \textbf{h})$. 
Suppose it holds for $n-1$.
Let $D_1$ (resp. $D_1'$) be the saturated $(\varphi, \Gamma)$-submodule of $D$ (resp. of $D'$) of rank $n-1$, and $C_1$ (resp. $C_1'$) be the quotient of $D$ (resp. of $D'$), both with the refinement $\phi^1=(\phi_1, \cdots, \phi_{n-1})$. Let $\sF$ (resp. $\sF'$, resp. $\sG$, resp. $\sG'$) be the filtration $D_1\subset D$ (resp. $D_1'\subset D'$, resp. $\cR_{K,E}(\phi_n z^{\textbf{h}_n})\subset C_1$, resp. $\cR_{K,E}(\phi_n z^{\textbf{h}_n})\subset D'$). As $\Ker(t_{D_{\sigma}})=\Ker(t_{D'_{\sigma}})$,  we have $\Ker(t_{D_{\sigma}, \sF})=\Ker(t_{D'_{\sigma}, \sF})$ and $\Ker(t_{D_{\sigma,\sG}})=\Ker(t_{D'_{\sigma,\sG}})$ by Corollary \ref{CtDFpsigma} (2). By  (\ref{Eparadecom}), we have $\Ker(t_{D_{1,\sigma}})=\Ker(t_{D_{1,\sigma}'})$ and $\Ker(t_{C_{1,\sigma}})=\Ker(t_{C_{1,\sigma}'})$, hence $D_{1,\sigma}\cong D_{1,\sigma}'$ and $C_{1,\sigma}\cong C_{1,\sigma}'$ by induction hypothesis.

Let $\pi:=\pi_{\min}(D)_{\sigma}\cong \pi_{\min}(D')_{\sigma}$. Let $\pi^-$ (resp. $\pi^+$) be the extension of $\Ker(t_{D_{\sigma},\sF})\otimes_E \pi_{\alg}(\phi, \textbf{h})$ (resp. $\Ker(t_{D_{\sigma},\sG})\otimes_E \pi_{\alg}(\phi, \textbf{h})$) by $\pi_{1,\sigma}(\phi, \textbf{h})$ (which stays unchanged if $D_{\sigma}$ is replaced by $D'_{\sigma}$). So $\pi^{\pm} \hookrightarrow \pi$. 
Let $\cL$ be the kernel of the following  natural (push-forward) map
\begin{equation}\label{Eparainf1}
	\ol{\Ext}^1_{\sigma}(\pi_{\alg}(\phi, \textbf{h}), \pi^-) \oplus \ol{\Ext}^1_{\sigma}(\pi_{\alg}(\phi, \textbf{h}), \pi^+) \twoheadlongrightarrow \ol{\Ext}^1_{\sigma}(\pi_{\alg}(\phi, \textbf{h}), \pi).
\end{equation}
We have a commutative diagram of exact sequences (see (\ref{Eparainf0}) for the bottom one)
\begin{equation}\label{EextRT}\small
	\begin{tikzcd}
&\cL \arrow[r, hook] \arrow[d] &\ol{\Ext}^1_{\sigma}(\pi_{\alg}(\phi, \textbf{h}), \pi^-) \oplus \ol{\Ext}^1_{\sigma}(\pi_{\alg}(\phi, \textbf{h}), \pi^+) \arrow[r, two heads] \arrow[d,  "\sim"] &\ol{\Ext}^1_{\sigma}(\pi_{\alg}(\phi, \textbf{h}), \pi) \arrow[d, " t_{D_{\sigma}}"', "\sim"]\\
&\cL(D_{\sigma}, D_{1,\sigma}, C_{1,\sigma}) \arrow[r, hook] &V(D_{1,\sigma}, C_{1,\sigma})_{\sigma} \arrow[r, two heads] &\ol{\Ext}^1_{\sigma}(D_{\sigma},D_{\sigma}) 
	\end{tikzcd}
\end{equation}
where the middle (bijective) map is induced by (\ref{EtDFPsigma2}). We then deduce $\cL \xrightarrow{\sim} \cL(D_{\sigma}, D_{1,\sigma}, C_{1,\sigma})$. Similarly, replacing $D_{\sigma}$ by $D'_{\sigma}$, we obtain $\cL\xrightarrow{\sim}  \cL(D_{\sigma}, D_{1,\sigma}, C_{1,\sigma})$. Note the middle map in (\ref{EextRT}) does not change when $D_{\sigma}$ is replaced by $D_{\sigma}'$ by the discussion in the first paragraph. Hence $ \cL(D_{\sigma}, D_{1,\sigma}, C_{1,\sigma})\cong  \cL(D_{\sigma}', D_{1,\sigma}, C_{1,\sigma})$ as subspace of $ V(D_{1,\sigma}, C_{1,\sigma})_{\sigma}$. But this implies $D_{\sigma}\cong D_{\sigma}'$ by Corollary \ref{Chodge1} (1). 
\end{proof}

\subsubsection{Universal extensions}\label{S322}
We give a reformulation of Theorem \ref{TtD} using  deformation rings of $(\varphi, \Gamma)$-modules, which will be useful in our proof of the local-global compatibility.

Let $D\in \Phi\Gamma_{\nc}(\phi, \textbf{h})$. Note  $\End_{(\varphi, \Gamma)}(D)\cong E$. Let $R_D$ be the universal deformation ring of deformations of $D$ over local Artinian $E$-algebras.  Let $R_{D,w}$ be the universal deformation ring of $\sT_w$-deformations of $D$ (i.e. the trianguline deformations of $D$ with respect to the refinement $w(\phi)$), and $R_{D,g}$ be the universal deformation ring of de Rham deformations or equivalently crystabelline deformations. All of these rings are formally smooth complete local Noetherian $E$-algebras (using the fact $\phi$ is generic). See \cite[\S~2.3.5, \S~2.5.3]{BCh} \cite[\S~2]{Na2} for detailed discussions. For a continuous character $\delta$ of $T(K)$, denote by $R_{\delta}$ the universal deformation ring of deformations of $\delta$ over local Artinian $E$-algebras, which is also formally smooth complete local Noetherian. If $\delta$ is locally algebraic, denote by $R_{\delta,g}$ the universal deformation ring of locally algebraic deformations of $\delta$. For a complete local Noetherian $E$-algebra $R$, we use $\fm_R$ to denote  its maximal ideal and  we will use $\fm$ for simplicity when it does not cause confusion. 

We have natural surjections $R_D \twoheadrightarrow R_{D,w} \twoheadrightarrow R_{D,g}$,  $R_{\delta} \twoheadrightarrow R_{\delta,g}$.
Moreover, we have natural isomorphisms of $E$-vector spaces for the tangent spaces $(\fm_{R_{D,*}}/\fm_{R_{D,*}}^2)^\vee\cong \Ext^1_{*}(D,D)$ for $*\in \{\emptyset, g, w\}$, $(\fm_{R_{\delta}}/\fm_{R_{\delta}}^2)^{\vee}\cong \Hom(T(K),E)$ and $(\fm_{R_{\delta,g}}/\fm_{R_{\delta,g}}^2)^{\vee}\cong \Hom_{\sm}(T(K),E)$. For $w\in S_n$, by Proposition \ref{PDef1} (2)(3), the map (\ref{Ekappaw}) $\kappa_w$ induces 
a commutative Cartesian diagram (of local Artinian $E$-algebras):
\begin{equation}\label{Ecartwg}
\begin{tikzcd}
&	R_{w(\phi)z^{\textbf{h}}}/\fm^2 \arrow[r, two heads] \arrow[d, hook] &R_{w(\phi)z^{\textbf{h}},g}/\fm^2 \arrow[d, hook] \\ 
&	R_{D,w}/\fm^2 \arrow[r, two heads]  &R_{D,g}/\fm^2.
\end{tikzcd}
\end{equation}
As in the proof of Proposition \ref{Pextalg} (2), let $\cH$ (resp. $\cH_i\cong \bG_m$) be the Bernstein centre over $E$ associated to the smooth representation $\pi_{\sm}(\phi)$ of $\GL_{n}(K)$ (resp. $\phi_i$ of $K^{\times}$) (cf. \cite[\S~3.13]{CEGGPS1}). For $w\in S_n$, there is a natural morphism $\cJ_w: \prod_{i=1}^n \Spec \cH_{w^{-1}(i)} \ra \Spec \cH$ (see the proof of  Proposition \ref{Pextalg} (2)), which, by \cite[Lem.~3.22]{CEGGPS1}, induces an isomorphism between completions at closed points.    The completion of $\prod_{i=1}^n \Spec \cH_{w^{-1}(i)}$ at the point $w(\phi)$ is naturally isomorphic to $R_{w(\phi), g}$.  Let $\widehat{\cH}_{\phi}$ be the completion of $\cH$ at $\pi_{\sm}(\phi)$. So $\cJ_w$ induces
$\cJ_w: \widehat{\cH}_{\phi}\xrightarrow{\sim} R_{w(\phi),g} \xrightarrow{\sim} R_{w(\phi)z^{\textbf{h}},g}$
where the second map is given by twisting $z^{-\textbf{h}}$. By Lemma \ref{Lint1} and Proposition \ref{Pextalg} (2), the composition
\begin{equation*}
A_0:=\widehat{\cH}_{\phi}/\fm^2 \xlongrightarrow[\sim]{\cJ_w} R_{w(\phi)z^{\textbf{h}},g}/\fm^2 \hooklongrightarrow R_{D,g}/\fm^2
\end{equation*}
is independent of the choice of $w$. We let $A_D:=	R_D/\fm^2 \times_{R_{D,g}/\fm^2} A_0$ and $A_{D,w}:=R_{D,w} /\fm^2 \times_{R_{D,g}/\fm_2} A_0$ ($\cong R_{w(\phi)z^{\textbf{h}}}/\fm^2$ by (\ref{Ecartwg})).  
The tangent space of $A_D$ (resp. $A_{D,w}$) is  naturally isomorphic to $\ol{\Ext}^1(D,D)$ (resp. $\ol{\Ext}^1_w(D,D)\cong  \Hom(T(K),E)$). We let $\cI_w$ be the kernel of $A_D \twoheadrightarrow A_{D,w}$. By Proposition \ref{Pinffern}, the natural morphism $A_D \ra \prod_{w} A_{D,w}$ is injective.

Let \ $\pi_1(\phi, \textbf{h})^{\univ}$ \ (resp. \ $\pi_1(\phi, \textbf{h})_w^{\univ}$) \ be \ the \ tautological \ extension \ of \ $\Ext^1_{\GL_n(K)}(\pi_{\alg}(\phi, \textbf{h}), \pi_1(\phi, \textbf{h})) \otimes_E\pi_{\alg}(\phi, \textbf{h})$ \big(resp. of $\Ext^1_{w}(\pi_{\alg}(\phi, \textbf{h}), \pi_1(\phi, \textbf{h})) \otimes_E\pi_{\alg}(\phi, \textbf{h})$\big) by $\pi_1(\phi, \textbf{h})$ (cf. \S~\ref{S311}). 
For $w\in S_n$, denote by $\delta_w:=w(\phi)z^{\textbf{h}} \varepsilon^{-1} \circ \theta$, and $\widetilde{\delta}_w^{\univ}$ the tautological extension of $\Ext^1_{T(K)}(\delta_w, \delta_w)\otimes_E \delta_w$ ($\cong \Hom(T(K),E) \otimes_E \delta_w$) by $\delta_w$. 
\begin{lemma}\label{Lunivw}
The induced representation $I_{B^-(K)}^{\GL_n(K)} \widetilde{\delta}_w^{\univ}$ is the universal extension of $\pi_{\alg}(\phi, \textbf{h})\otimes_E \Ext^1_{\GL_n(K)}(\pi_{\alg}(\phi, \textbf{h}), \PS_1(w(\phi), \textbf{h}))$ by $\PS_1(w(\phi), \textbf{h})$.
\end{lemma}
\begin{proof}
By Remark \ref{Rind1}, $I_{B^-(K)}^{\GL_n(K)} \widetilde{\delta}_w^{\univ}$ is an extension of $\pi_{\alg}(\phi, \textbf{h})^{\oplus (n(d_K+1))}$ by a certain subrepresentation $V$ of $\PS_1(w(\phi), \textbf{h})$. However, again by  Remark \ref{Rind1}, as $\widetilde{\delta}_w^{\univ}$ is universal, any extension in the image of (\ref{Eind1}) comes from an extension of $\pi_{\alg}(\phi, \textbf{h})$ by $V$ by push-forward via $V\hookrightarrow \PS_1(w(\phi), \textbf{h})$. As (\ref{Eind1}) is bijective (by Proposition \ref{Pextps} (1)), using the surjectivity of the last map in (\ref{Edivips}), we see $V$ has to be the entire $\PS_1(w(\phi), \textbf{h})$. 
Using again Proposition \ref{Pextps} (1), we see  $I_{B^-(K)}^{\GL_n(K)} \widetilde{\delta}_w^{\univ}$ is in fact the universal extension. 
\end{proof}
We have hence an isomorphism of $\GL_n(K)$-representations
\begin{equation}\label{Edeltauniv}
I_{B^-(K)}^{\GL_n(K)} \widetilde{\delta}_w^{\univ} \oplus_{\PS_1(w(\phi), \textbf{h})} \pi_1(\phi, \textbf{h})\xlongrightarrow{\sim } \pi_1(\phi, \textbf{h})_w^{\univ}.
\end{equation}
There is a natural action of $A_{D,w}\cong R_{w(\phi) z^{\textbf{h}}}/\fm^2$ on $\widetilde{\delta}_w^{\univ}$ where an element $x\in \fm_{R_{w(\phi) z^{\textbf{h}}}}/\fm_{R_{w(\phi) z^{\textbf{h}}}}^2\cong\Hom(T(K),E)^{\vee}$ acts via $x: \widetilde{\delta}_w^{\univ} \twoheadrightarrow\Hom(T(K),E)\otimes_E \delta_w \xrightarrow{x} \delta_w \hookrightarrow \widetilde{\delta}_w^{\univ}$.
Hence $I_{B^-(K)}^{\GL_n(K)} \widetilde{\delta}_w^{\univ} $ is equipped with an induced $R_{w(\phi)z^{\textbf{h}}}/\fm^2$-action. Similarly $\pi_1(\phi,\textbf{h})_w^{\univ}$ is equipped with an action of $A_{D,w}$ given by 
\begin{equation}\label{EactAw}
x:\pi_1(\phi, \textbf{h})_w^{\univ} \twoheadrightarrow \Ext^1_w(\pi_{\alg}(\phi,\textbf{h}), \pi_1(\phi, \textbf{h}))\otimes_E \pi_{\alg}(\phi, \textbf{h}) \xrightarrow{x} \pi_{\alg}(\phi, \textbf{h}) \hookrightarrow \pi_1(\phi, \textbf{h})_w^{\univ},
\end{equation}
for $x\in \fm_{A_{D,w}}/\fm_{A_{D,w}}^2\cong \Hom(T(K),E)^{\vee} \xrightarrow[\sim]{\zeta_w}\Ext^1_w(\pi_{\alg}(\phi,\textbf{h}), \pi_1(\phi, \textbf{h}))^{\vee}$. The injection $I_{B^-(K)}^{\GL_n(K)}\widetilde{\delta}_w^{\univ}\hookrightarrow \pi_1(\phi,\textbf{h})^{\univ}_w$ (induced by (\ref{Edeltauniv})) is $A_{D,w}$-equivariant.  

For a commutative $E$-algebra $A$ acting on an $E$-vector space $V$, and an ideal $I$ of $A$, we denote by $V[I]$ the subspace of $V$ annihilated by elements in $I$.
The following theorem is a reformulation of Theorem \ref{TtD}. 

\begin{theorem} \label{Tuniv}There is a unique $A_D$-action on $\pi_1(\phi, \textbf{h})^{\univ}$ such that for all $w\in S_n$, we have an $A_{D,w}\times \GL_n(K)$-equivariant injection
$\pi_1(\phi, \textbf{h})_w^{\univ} \hookrightarrow		\pi_1(\phi, \textbf{h})^{\univ}[\cI_w]$.
\end{theorem}
\begin{proof}
By Theorem \ref{TtD}, we define an $A_D$-action by letting $x\in \fm_{A_D}/\fm_{A_D}^2\cong \ol{\Ext}^1(D,D)^{\vee}\hookrightarrow \Ext^1_{\GL_n(K)}(\pi_{\alg}(\phi, \textbf{h}), \pi_1(\phi, \textbf{h}))^{\vee}$ act via
\begin{equation}\label{EAdaction}
	x: \pi_1(\phi, \textbf{h})^{\univ} \twoheadrightarrow \Ext^1(\pi_{\alg}(\phi,\textbf{h}), \pi_1(\phi, \textbf{h}))\otimes_E \pi_{\alg}(\phi, \textbf{h}) \xrightarrow{x} \pi_{\alg}(\phi, \textbf{h}) \hookrightarrow \pi_1(\phi, \textbf{h})_w^{\univ}.
\end{equation}
By definition, the restriction of $t_D$ to $\Ext^1_w(\pi_{\alg}(\phi, \textbf{h}), \pi_1(\phi, \textbf{h}))$ coincides with the composition $\Ext^1_w(\pi_{\alg}(\phi, \textbf{h}), \pi_1(\phi, \textbf{h})) \xrightarrow[\sim]{\zeta_w^{-1}}\Hom(T(K),E) \xrightarrow[\sim]{\kappa_w^{-1}} \ol{\Ext}^1_w(D,D)$. 
We deduce the action in (\ref{EAdaction}) is compatible with (\ref{EactAw}) when $x\in A_{D,w}$ hence satisfies the properties in the theorem. The uniqueness follows from the fact $\pi_1(\phi, \textbf{h})^{\univ}$ is generated by $\pi_1(\phi, \textbf{h})_w^{\univ}$ for $w\in S_n$.
\end{proof}
By the construction of $\pi_{\min}(D)$, we have: 
\begin{corollary}\label{Cpimin}
We have $\pi_{\min}(D)\cong \pi_1(\phi, \textbf{h})^{\univ}[\fm_{A_D}]$.
\end{corollary}

\section{Local-global compatibility}
\subsection{The patched setting}\label{S41}
Let $M_{\infty}$ be the patched module of \cite[\S~2]{CEGGPS1}. Then 
$\Pi_{\infty}:=\Hom_{\co_E}^{\cont}(M_{\infty}, E)$, equipped with the usual maximum norm,  is a unitary Banach representation of $\GL_n(K)$ (where $K$ is the field $F$ of \textit{loc. cit.}), which is equipped with an action of the patched Galois deformation ring $R_{\infty}\cong R_{\overline{\rho}}^{\square} \widehat{\otimes}_{\co_E} R_{\infty}^{\wp}$ (where $\wp$ is ``$\widetilde{\fp}$" and $\overline{\rho}$ is the local Galois representation $\overline{r}$ of \textit{loc. cit.}). We refer to \cite[\S~2.8]{CEGGPS1} for details. Let $\widehat{T}$ be the rigid space over $E$ parametrizing continuous characters of $T(K)$. Let
\begin{equation}\label{Eeigen}
\cE \hooklongrightarrow (\Spf R_{\overline{\rho}}^{\square})^{\rig} \times \widehat{T} \times (\Spf R_{\infty}^{\wp})^{\rig}
\end{equation}be the associated patched eigenvariety (see \cite[\S~4.1.2]{Ding7}, that is an easy variation of the patched eigenvariety introduced in \cite{BHS1}), $\cM$ be the natural coherent sheaf on $\cE$ such that there is a $T(K)\times R_{\infty}$-equivariant isomorphism 
$\Gamma(\cE, \cM)^{\vee} \cong J_B(\Pi_{\infty}^{R_{\infty}-\an})$ (see \cite[\S~3.1]{BHS1} for ``$R_{\infty}-\an$"). 
Recall a point $x=(\rho_{x,\wp}, \delta_x, \fm_x^{\wp})\in (\Spf R_{\overline{\rho}}^{\square})^{\rig} \times \widehat{T} \times (\Spf R_{\infty}^{\wp})^{\rig}$ lies in $\cE$ if and only if 
$\Hom_{T(K)}\big(\delta_x, J_B(\Pi_{\infty}^{R_{\infty}-\an})[\fm_x]\big)\neq 0$
where $\fm_x=(\rho_{x,\wp}, \fm_x^{\wp})$ is the associated maximal ideal of $R_{\infty}[1/p]$. 

Let  $X_{\tri}^{\square}(\overline{\rho})\hookrightarrow (\Spf R_{\overline{\rho}}^{\square})^{\rig} \times \widehat{T}$ be the trianguline variety \cite[\S~2.2]{BHS1}, and $\iota_p$ be the twisting map (see \S~\ref{S311} for $\delta_B$ and $\theta$)
\begin{equation*}
\iota_p: (\Spf R_{\overline{\rho}}^{\square})^{\rig} \times \widehat{T} \xlongrightarrow{\sim} (\Spf R_{\overline{\rho}}^{\square})^{\rig} \times \widehat{T}, \ (\rho_p, \chi ) \mapsto (\rho_p, \chi\delta_B (\varepsilon^{-1} \circ \theta)).
\end{equation*}Recall (\ref{Eeigen}) factors through an embedding (cf. \cite[Thm.~1.1]{BHS1}). 
\begin{equation}\label{Eeigvar1}
\cE \hooklongrightarrow \iota_p(X_{\tri}^{\square}(\overline{\rho})) \times (\Spf R_{\infty}^{\wp})^{\rig}, 
\end{equation} 
which identifies $\cE$ with a union of irreducible components of the latter.  Recall $\dim X_{\tri}^{\square}(\overline{\rho})=n^2+d_K\frac{n(n+1)}{2}$ (cf. \cite[Thm.~2.6]{BHS1}, where  ``$n^2$" comes from the framing).

Let $\rho: \Gal_{K}\ra \GL_n(E)$ be a continuous representation such that $\rho$ admits a $\Gal_K$-invariant $\co_E$-lattice whose  modulo $p$ reduction is equal to $\overline{\rho}$ and that $D:=D_{\rig}(\rho)\in \Phi\Gamma_{\nc}(\phi, \textbf{h})$ ($\phi$ generic, $\textbf{h}$ strictly dominant given as in \S~\ref{S22}). 
Let $\fm_{\rho}\subset R_{\overline{\rho}}^{\square}[1/p]$ be the maximal ideal associated to $\rho$.
\begin{lemma}\label{Lnocompa}
	For a continuous character $\delta: T(K) \ra E^{\times}$, $(\rho, \delta)\in X_{\tri}^{\square}(\overline{\rho})$ if and only if $\delta=w(\phi)z^{\textbf{h}}$ for some $w\in S_n$.
\end{lemma}
\begin{proof}
	  The ``if" part follows the very construction (cf. \cite[\S~2.2]{BHS1}). Indeed all these points lie in the space $U_{\tri}^{\square}(\overline{\rho})^{\mathrm{reg}}$ of \textit{loc. cit.} The ``only if" part follows from the fact that $D$ is non-critical hence does not admit companion points of non-dominant weight (e.g. see \cite[Thm.~4.2.3]{BHS3}, \cite[Cor.~6.4.12]{BD3}).
\end{proof} 
Suppose there exists a maximal ideal $\fm^{\wp}$ of $R_{\infty}^{\wp}[1/p]$ such that $\Pi_{\infty}[\fm_x]^{\lalg}\neq 0$ for $\fm_x=(\fm_{\rho}, \fm^{\wp})$, the corresponding maximal ideal of $R_{\infty}[1/p]$. By \cite[Thm.~4.35]{CEGGPS1}, we have $\Pi_{\infty}[\fm_x]^{\lalg} \cong \pi_{\alg}(\phi, \textbf{h})$. By taking Jacquet-Emerton modules, this implies $x_w:=(x_{w,\wp}, \fm^{\wp})=(\rho, \delta_{w}\delta_B=w(\phi)z^{\textbf{h}}\delta_B( \varepsilon^{-1} \circ \theta ), \fm^{\wp})\in \cE$ for all $w\in S_n$. By Lemma \ref{Lnocompa}, these give all the points on $\cE$ associated to $\fm_x$. By \cite[Thm.~2.6 (iii)]{BHS1}, 
$X_{\tri}^{\square}(\overline{\rho})$ is smooth at the points $\iota_p^{-1}(x_{w,\wp})$ and (\ref{Eeigvar1}) is a local isomorphism at $x_w$.  As $(\Spf R_{\infty}^{\wp})^{\rig}$ is also smooth at $\fm^{\wp}$ (e.g. see the proof of \cite[Cor.~4.4]{Ding6}), $\cE$ is smooth at all $x_{w}$.  By  \cite[Lem.~3.8]{BHS2} and \cite[Thm. 4.35]{CEGGPS1},  we see  $\cM$ is locally free of rank one at all $x_w$.  

Let $R_D^{\square}:=R_D \otimes_{R_{\rho}} R_{\rho}^{\square}\cong R_{\rho}^{\square}$  (where $R_{\rho}^{\square}$ is the framed universal deformation ring of $\rho$ of deformations over local Artinian $E$-algebras). Let $R_{D,w}^{\square}:=R_D^{\square} \otimes_{R_D} R_{D,w}$ for $w\in S_n$ and $R_{D,g}^{\square}:=R_D^{\square} \otimes_{R_D} R_{D,g}$. We have commutative Cartesian diagrams (see \S~\ref{S322}, in particular the discussion below (\ref{Ecartwg}) for $A_D$, $A_{D,w}$, $A_0$, where $\fm$ denotes the corresponding maximal ideals):
\begin{equation}\label{Ecartiwg1}
	\begin{tikzcd}
		&A_D \arrow[r, hook]  \arrow[d, two heads]  &R_{D}/\fm^2 \arrow[r, hook] \arrow[d, two heads] &R_{D}^{\square}/\fm^2 \arrow[d, two heads] \\
			&A_{D,w} \arrow[r, hook]  \arrow[d, two heads]  &R_{D,w}/\fm^2 \arrow[r, hook] \arrow[d, two heads] &R_{D,w}^{\square}/\fm^2 \arrow[d, two heads]  \\
				&A_0 \arrow[r, hook] &R_{D,g}/\fm^2 \arrow[r, hook]  &R_{D,g}^{\square}/\fm^2.
	\end{tikzcd}
\end{equation}
Let $\fa\supset \fm_{R_{D}^{\square}}^2$ be an ideal of $R_D^{\square}$ such that $\fa/\fm_{R_{D}^{\square}}^2 \oplus \fm_{A_D}/\fm_{A_D}^2 \xrightarrow{\sim} \fm_{R_{D}^{\square}}/\fm_{R_{D}^{\square}}^2$ (noting $ \fm_{A_D}/\fm_{A_D}^2\hookrightarrow \fm_{R_{D}^{\square}}/\fm_{R_{D}^{\square}}^2$). The composition $A_D \hookrightarrow R_{D}^{\square}/\fm_{R_{D}^{\square}}^2 \twoheadrightarrow R_D^{\square}/\fa$ is hence an isomorphism. We use $\fa$ to denote its image in $R_{D,w}^{\square}$ and $R_{D,g}^{\square}$. By (\ref{Ecartiwg1}), $\fa/\fm_{R_{D,w}^{\square}}^2 \oplus \fm_{A_{D,w}}/\fm_{A_{D,w}}^2 \xrightarrow{\sim} \fm_{R_{D,w}^{\square}}/\fm_{R_{D,w}^{\square}}^2$ and  $\fa/\fm_{R_{D,g}^{\square}}^2 \oplus \fm_{A_{0}}/\fm_{A_{0}}^2 \xrightarrow{\sim} \fm_{R_{D,g}^{\square}}/\fm_{R_{D,g}^{\square}}^2$.  Moreover, the compositions $A_{D,w} \hookrightarrow R_{D,w}^{\square}/\fm^2 \twoheadrightarrow R_{D,w}^{\square}/\fa$ for $w\in S_n$, and $A_0 \hookrightarrow R_{D,g}^{\square}/\fm^2 \twoheadrightarrow R_{D,g}^{\square}/\fa$ are all isomorphisms.

Recall the completion of $R_{\overline{\rho}}^{\square}[1/p]$ at $\rho$ is naturally isomorphic to $R_{\rho}^{\square}\cong R_{D}^{\square}$ (cf. \cite{Kis09}). Let $\fa\subset R_{\overline{\rho}}^{\square}[1/p]$ denote the preimage of $\fa\subset R_{D}^{\square}$. Let $\fa_x:=(\fa, \fm^{\wp})\subset R_{\infty}[1/p]$. So $\Pi_{\infty}^{R_{\infty}-\an}[\fa_x]$ is equipped with a natural $R_{D}^{\square}/\fa\cong A_{D}$-action.
\begin{lemma}\label{Lfax}
	(1) We have $\Pi_{\infty}^{R_{\infty}-\an}[\fm_x]\xrightarrow{\sim}\Pi_{\infty}^{R_{\infty}-\an}[\fa_x][\fm_{A_D}]$.
	
\noindent	(2) $\Hom_{\GL_n}(\pi_{\alg}(\phi, \textbf{h}), \Pi_{\infty}^{R_{\infty}-\an}[\fa_x])\cong \Hom_{\GL_n}(\pi_{\alg}(\phi, \textbf{h}), \Pi_{\infty}^{R_{\infty}-\an}[\fm_x])\cong E$.
\end{lemma}
\begin{proof}By definition, $\fa+\fm_{A_D}=\fm_{R_D^{\square}}$ hence $\fa_x+\fm_{A_D}=\fm_x$. (1) follows. By \cite[Lem.~4.17]{CEGGPS1}, the $R_D^{\square}$-action on $\Pi_{\infty}^{R_{\infty}-\an}[\fa_x]^{\lalg}$ factors through $R_{D,g}^{\square}/\fa$.  Let $\sigma_{\sm}$ be a smooth irreducible representation of $\GL_n(\co_K)$ over $E$ associated to the Bernstein component of $\pi_{\sm}(\phi)$. We have  $\cH\cong\End_{\GL_n(K)}(\cind_{\GL_n(\co_K)}^{\GL_n(K)}\sigma_{\sm})$ (recalling $\cH$ is the Bernstein centre associated to $\pi_{\sm}(\phi)$). Consider 
	\begin{multline*}
		M:=\Hom_{\GL_n(K)}\big((\cind_{\GL_n(\co_K)}^{\GL_n(K)} \sigma_{\sm}) \otimes_E L(\lambda), \Pi_{\infty}^{R_{\infty}-\an}[\fa_x]^{\lalg}\big)\\
		\cong \Hom_{\GL_n(K)}\big(\cind_{\GL_n(\co_K)}^{\GL_n(K)} \sigma_{\sm}, (\Pi_{\infty}^{R_{\infty}-\an}[\fa_x]^{\lalg}\otimes_E L(\lambda)^{\vee})_{\sm}\big)
	\end{multline*}
	which is naturally an $\cH$-module (with $\cH$ acting on $\cind_{\GL_n(\co_K)}^{\GL_n(K)} \sigma_{\sm}$). Hence $\Hom_{\GL_n(K)}(\pi_{\alg}(\phi, \textbf{h}), \Pi_{\infty}^{R_{\infty}-\an}[\fa_x]^{\lalg})\cong M[\fm_{A_0}]$. By \cite[Thm.~4.19]{CEGGPS1}, this $\cH$-action on $M$ coincides with the one induced by $\cH \ra \widehat{\cH}_{\phi} \twoheadrightarrow A_0 \xrightarrow{\sim} R_{D,g}^{\square}/\fa$ acting on $\Pi_{\infty}^{R_{\infty}-\an}[\fa_x]^{\lalg}$ . This implies 
$M[\fm_{A_0}]\cong \Hom_{\GL_n(K)}\big((\cind_{\GL_n(\co_K)}^{\GL_n(K)} \sigma_{\sm}) \otimes_E L(\lambda), \Pi_{\infty}^{R_{\infty}-\an}[\fa_x]^{\lalg}[\fm_{A_{0}}]\big)$.
	However, as ideals of $R_{D,g}^{\square}/\fm^2$, we have $\fm_{A_0}+\fa=\fm$ hence $ \Pi_{\infty}^{R_{\infty}-\an}[\fa_x]^{\lalg}[\fm_{A_0}]\cong \Pi_{\infty}^{R_{\infty}-\an}[\fm_x]^{\lalg}\cong \pi_{\alg}(\phi, \textbf{h})$ (by \cite[Thm.~4.3.5]{CEGGPS1}). So $M[\fm_{A_0}]\cong E$, proving (2).
\end{proof}
For $w\in S_n$, let $U=U_{\wp} \times U^{\wp}\subset \iota_p(X_{\tri}^{\square}(\overline{\rho}) )\times (\Spf R_{\infty}^{\wp})^{\rig}$ be a smooth affinoid neighbourhood of $x_w$ such that $x_{w'}\notin U$ for $w'\neq w$.   Let $\fm_{x_{w,\wp}}$ be the maximal ideal of $\co(U_{\wp})$ at $x_{w,\wp}$. As $x_{w,\wp}$ is the only point in $U_{\wp}$ associated to $\rho$ (by assumption and Lemma \ref{Lnocompa}), $\fm_{x_{w,\wp}}\subset \co(U_{\wp})$ is the closed ideal generated by $\fm_{\rho}\subset R_{\overline{\rho}}^{\square}[1/p]$. Let $\fa\supset \fm_{x_{w,\wp}}^2$ be the closed ideal generated by $\fa \subset R_{\overline{\rho}}^{\square}[1/p]$. Consider $\cM_{\tilde{x}_w}:=\cM(U)/(\fa+\fm^{\wp})$. By definition, we have  a $T(K) \times R_{\infty}$-equivariant map 
\begin{equation}\label{Etangent1}
	\cM_{\tilde{x}_w}^{\vee} \hooklongrightarrow J_B(\Pi_{\infty}^{R_{\infty}-\an})[\fa_x]\cong J_B(\Pi_{\infty}^{R_{\infty}-\an}[\fa_x]).
\end{equation}
 Recall as $D$ is non-critical, the completion of $X_{\tri}^{\square}(\overline{\rho})$ at $\iota_p^{-1}(x_{w,\wp})$ is naturally isomorphic to $R_{D,w}^{\square}$. As $\cM$ is locally free of rank one at $x_w$, we see $\cM_{\tilde{x}_w}\cong R_{D,w}^{\square}/\fa$. In particular, $\dim_E \cM_{\tilde{x}_w}=	1+(n+nd_K)$. The $T(K)$-action on $\cM_{\tilde{x}_w}$ is encoded in the $A_{D,w}$-action. Indeed, $T(K)$ acts on the $A_{D,w}$-module $\cM_{\tilde{x}_w}$ via the composition:
  \begin{equation}\label{Et(K)act}T(K) \lra R_{\delta_w \delta_B}/\fm^2 \xlongrightarrow{\sim}R_{w(\phi)z^{\textbf{h}}}/\fm^2 \cong A_{D,w}
 \end{equation}
where the first map is induced by the universal deformation, and the second is induced by twisting $\delta_B (\varepsilon^{-1} \circ \theta)$ (which corresponds to the twist in $\iota_p$).  
We equip $A_{D,w}$ with the $T(K)$-action as in (\ref{Et(K)act}). Then the $T(K)$-representation $A_{D,w}^{\vee}$ is just isomorphic to the universal extension $\widetilde{\delta}_w^{\univ} \delta_B$ (see the discussion below (\ref{Edeltauniv})). In summary, we have  a $T(K) \times A_{D,w}$-equivariant isomorphism $\cM_{\tilde{x}_w}^{\vee}\cong \widetilde{\delta}_w^{\univ} \delta_B$.


\begin{lemma}\label{Lbala}
The map (\ref{Etangent1}) is balanced in the sense of \cite[Def.~0.8]{Em2}, hence (by \cite[Thm.~0.13]{Em2}) induces a $\GL_n(K) \times R_{\infty}$-equivariant injection 
$$\iota_w: 	I_{B^-(K)}^{\GL_n(K)}\widetilde{\delta}_w^{\univ}\hooklongrightarrow \Pi_{\infty}^{R_{\infty}-\an}[\fa_x],$$ where the $R_{\infty}$-action on $I_{B^-(K)}^{\GL_n(K)}\widetilde{\delta}_w^{\univ}=I_{B^-(K)}^{\GL_n(K)}((\widetilde{\delta}_w^{\univ}\delta_B) \delta_B^{-1})$ is induced by  $R_{\infty} \ra (R_{D,w}^{\square}/\fa) \otimes_E (R_{\infty}^{\wp}[1/p]/\fm^{\wp}) \xrightarrow{\sim} A_{D,w}$ acting on $\widetilde{\delta}_w^{\univ}\delta_B$.
\end{lemma}
\begin{proof}
The lemma follows by the same argument as in \cite[Lem.~4.11]{Ding7}, using Lemma \ref{Lnocompa}.
\end{proof}

\begin{lemma}\label{Clg1}
We have $(\Ima \iota_w)[\fm_x]\cong \PS_1(w(\phi), \textbf{h})$. 
\end{lemma}
\begin{proof}By Lemma \ref{Lunivw}, we have 
	\begin{equation}\label{EinjPS1}\PS_1(w(\phi), \textbf{h})\hooklongrightarrow 	I_{B^-(K)}^{\GL_n(K)}\widetilde{\delta}_w^{\univ}{\buildrel {\iota_w} \over \hooklongrightarrow} \Pi_{\infty}^{R_{\infty}-\an}[\fa_x].
	\end{equation} By Lemma \ref{Lfax}, it sends $\pi_{\alg}(\phi, \textbf{h})$ to $\Pi_{\infty}^{R_{\infty}-\an}[\fm_x]$. By Lemma \ref{Lnocompa} and Lemma \ref{Ljacsi} (1),  $\Hom_{\GL_n(K)}(\PS_1(w(\phi), \textbf{h}), \Pi_{\infty}^{R_{\infty}-\an}[\fm_x])=0$. For any $\alpha\in \fm_x$,  the induced ($\GL_n(K)$-equivariant) map $\Pi_{\infty}^{R_{\infty}-\an}[\fa_x]\xrightarrow{\alpha} \Pi_{\infty}^{R_{\infty}-\an}[\fa_x]$ composed with (\ref{EinjPS1}) factors through $\PS_1(w(\phi), \textbf{h}) \ra \Pi_{\infty}^{R_{\infty}-\an}[\fm_x]$ (noting $\fa_x \supset \fm_x^2$), hence has to be zero. So (\ref{EinjPS1}) has image in $\Pi_{\infty}^{R_{\infty}-\an}[\fm_x]$. 

By \cite[Lem.~4.16]{BHS2} (which directly generalizes to the crystabelline case), we have (where $\{-\}$ denotes the generalized eigenspace):
\begin{equation}\label{Esemi11}J_B(\Pi_{\infty}^{R_{\infty}-\an}[\fm_x])[T(K)=\delta_w \delta_B]\xlongrightarrow{\sim} J_B(\Pi_{\infty}^{R_{\infty}-\an}[\fm_x])\{T(K)=\delta_w \delta_B\}\end{equation}
which is hence one dimensional. Indeed, consider the tangent map of $U_{\wp}\ra (\Spf R_{\overline{\rho}}^{\square})^{\rig} \times \widehat{T}$ at the point $x_{w,\wp}$:
\begin{equation*}
	(\kappa_1, \kappa_2): T_{x_{w,\wp}} U_{\wp} \lra T_{\rho}(\Spf R_{\overline{\rho}}^{\square})^{\rig} \times T_{\delta_w\delta_B} \widehat{T}\lra \Ext^1(D,D) \times \Hom(T(K),E)
\end{equation*}(where $T_x X$ denotes the tangent space of a rigid analytic space $X$ at a closed point $x$).  For $v\in T_{x_{w,\wp}} U_{\wp} $, let $\widetilde{D}_v$ be the deformation of $D$ over $\cR_{K, E[\epsilon]/\epsilon^2}$ associated to $\kappa_1(v)$. As $D$ is generic non-critical, the global triangulation theory  (e.g. using \cite[Thm.~5.3]{Bergd17} and an induction argument) implies $\widetilde{D}_v$ is trianguline of parameter $\delta_w (1+\kappa_2(v) \epsilon)$. In particular, if $\kappa_1(v)=0$, then $\kappa_2(v)$ must be zero as well. As $\cM$ is locally free at $x_w$, we deduce $J_B(\Pi_{\infty}^{R_{\infty}-\an}[\fm_x])[\fm_{\delta_w\delta_B}]\xrightarrow{\sim} J_B(\Pi_{\infty}^{R_{\infty}-\an}[\fm_x])[\fm_{\delta_w\delta_B}^2]$, where $\fm_{\delta_w\delta_B}$ denotes the maximal ideal of $\co(\widehat{T})$ at the point $\delta_w\delta_B$. But this implies $J_B(\Pi_{\infty}^{R_{\infty}-\an}[\fm_x])[\fm_{\delta_w\delta_B}]\xrightarrow{\sim} J_B(\Pi_{\infty}^{R_{\infty}-\an}[\fm_x])[\fm_{\delta_w\delta_B}^k]$ for all $k\geq 1$. (\ref{Esemi11}) follows.

Now if the injection $\PS_1(w(\phi),\textbf{h})\hookrightarrow (\Ima \iota_w)[\fm_x]$ is not surjective, by Proposition \ref{Pextps} (1) and Remark \ref{Rind1}, there exists $\psi\in \Hom(T(K),E)$ such that $I_{B^-(K)}^{\GL_n(K)}(w(\phi) \eta z^{\lambda}(1+\psi \epsilon))\hookrightarrow(\Ima \iota_w)[\fm_x]$. Hence $$\delta_w \delta_B(1+\psi \epsilon)\hooklongrightarrow J_B(\Pi_{\infty}^{R_{\infty}-\an}[\fm_x])\{T(K)=\delta_w \delta_B\},$$ a contradiction. The lemma follows.
\end{proof}
Let $\widetilde{\pi}$ be the closed subrepresentation of $\Pi_{\infty}^{R_{\infty}-\an}[\fa_x]$ generated by $\Ima \iota_w$ for all $w\in S_n$. It is clear that $\widetilde{\pi}$ is stabilized by $R_{\infty}$. In particular, $\widetilde{\pi}$ has an induced $A_D$-action via $A_D \xrightarrow{\sim} R_{D}^{\square}/\fa\xrightarrow{\sim} R_{\infty}[1/p]/\fa_x$.
\begin{theorem}\label{Tlg}
We have a $\GL_n(K) \times A_D$-equivariant  isomorphism  $\pi_1(\phi, \textbf{h})^{\univ} \cong \widetilde{\pi}$ (see Theorem \ref{Tuniv} for the $A_D$-action on $\pi_1(\phi, \textbf{h})^{\univ}$). 
\end{theorem}
\begin{proof}
We first show $\widetilde{\pi} \cong \pi_1(\phi, \textbf{h})^{\univ}$ as $\GL_n(K)$-representation. It is clear that the irreducible constituents of $\widetilde{\pi}$ are given by $\pi_{\alg}(\phi, \textbf{h})$, and all $\sC(I,s_{i,\sigma})$ (with multiplicities no less than one, cf. (\ref{ECIsigma})). By Lemma \ref{Lnocompa} and Lemma \ref{Ljacsi} (1), $\Hom_{\GL_n(K)}(\sC(I,s_{i,\sigma}), \Pi_{\infty}^{R_{\infty}-\an}[\fa_x])=0$. Together with the fact $\widetilde{\pi}^{\lalg}\cong \pi_{\alg}(\phi, \textbf{h})$ by Lemma \ref{Lfax} (2), we see $\soc_{\GL_n(K)} \widetilde{\pi}\cong  \pi_{\alg}(\phi, \textbf{h})$.  It is also clear from the definition that all $\sC(I,s_{i,\sigma})$ lie in the socle of $\widetilde{\pi}/ \pi_{\alg}(\phi, \textbf{h})$, hence all have multiplicity one by Lemma \ref{Lextps} (1). These together with $I_{B^-(K)}^{\GL_n(K)} \widetilde{\delta}_w^{\univ} \subset \widetilde{\pi}$ for all $w$, imply $\widetilde{\pi}\cong  \pi_1(\phi, \textbf{h})^{\univ}$.

We equip $ \pi_1(\phi, \textbf{h})^{\univ}$ with an $A_D$-action induced by the  $A_D$-action on $\widetilde{\pi}$ (induced from $R_{\infty}$). By Lemma \ref{Lbala}, the composition $I_{B^-(K)}^{\GL_n(K)}\widetilde{\delta}_w^{\univ} \hookrightarrow \widetilde{\pi}\cong  \pi_1(\phi, \textbf{h})^{\univ}$ is then $A_D$-equivariant hence factors through an $A_{D,w}$-equivariant injection $I_{B^-(K)}^{\GL_n(K)} \widetilde{\delta}_w^{\univ} \hookrightarrow  \pi_1(\phi, \textbf{h})^{\univ}[\cI_w]$ (recalling $\cI_w$ is the kernel of $A_D \twoheadrightarrow A_{D,w}$). 
By Theorem \ref{Tuniv}, we see this (global) $A_D$-action coincides with the one given there. This concludes the proof.
\end{proof}
Together with Corollary \ref{Cpimin}, Lemma \ref{Lfax} (1), we get 
\begin{corollary}
	We have  a $\GL_n(K)$-equivariant injection 
	\begin{equation}
		\label{EinjDcru} \pi_{\min}(D) \cong\pi_1(\phi, \textbf{h})^{\univ}[\fm_{A_D}]\cong \widetilde{\pi}[\fm_{A_D}] \hookrightarrow  \Pi_{\infty}^{R_{\infty}-\an}[\fa_x][\fm_D] \cong \Pi_{\infty}^{R_{\infty}-\an}[\fm_x].
	\end{equation}
\end{corollary}
\begin{remark}\label{Rpifs2}
 By \cite[Thm.~5.12]{BH2}, $\pi_{\alg}(\phi, \textbf{h}) \hookrightarrow \Pi_{\infty}^{R_{\infty}-\an}[\fm_x]$ uniquely extends to $\pi(\phi, \textbf{h})\hookrightarrow \Pi_{\infty}^{R_{\infty}-\an}[\fm_x]$. Using (\ref{EpifsD}), we deduce from (\ref{EinjDcru}) an injection
$\pi_{\fss}(D) \hookrightarrow\Pi_{\infty}^{R_{\infty}-\an}[\fm_x]$.
Remark that $\pi_{\fss}(D)$ should still be far from the entire $\Pi_{\infty}^{R_{\infty}-\an}[\fm_x]$. 
\end{remark}

\begin{corollary}\label{Cmaxmin}The representation $\pi_{\min}(D)$ is the maximal subrepresentation of $\Pi_{\infty}^{R_{\infty}-\an}[\fm_x]$ given by extensions of $\pi_{\alg}(\phi, \textbf{h})$ by $\pi_1(\phi, \textbf{h})$. 
\end{corollary}
\begin{proof}
Let $V$ be an extension of $\pi_{\alg}(\phi, \textbf{h})$ by $\pi_1(\phi, \textbf{h})$ contained in $\Pi_{\infty}^{R_{\infty}-\an}[\fm_x]\subset  \Pi_{\infty}^{R_{\infty}-\an}[\fa_x]$. If $V$ is not a subrepresentation of $\widetilde{\pi}$, then $V\oplus_{\pi_1(\phi, \textbf{h})} \widetilde{\pi}\hookrightarrow \Pi_{\infty}^{R_{\infty}-\an}[\fa_x]$. As $\widetilde{\pi}$ is isomorphic to the universal extension of $\pi_1(\phi, \textbf{h})$ by $\pi_{\alg}(\phi, \textbf{h})$, this implies $\dim_E \Hom_{\GL_n(K)}(\pi_{\alg}(\phi, \textbf{h}), \Pi_{\infty}^{R_{\infty}-\an}[\fa_x])\geq 2$ contradicting Lemma \ref{Lfax} (2). So $V\subset \widetilde{\pi} \cap \Pi_{\infty}^{R_{\infty}-\an}[\fm_x]= \widetilde{\pi}[\fm_x]=\widetilde{\pi}[\fm_{A_D}]\cong \pi_{\min}(D)$.
\end{proof}
By  \cite[Thm.~5.3.3]{BHS3} (for the crystalline case) and \cite[Thm.~1.3]{BD3} (for the crystabelline non-crystalline case), the information that $D$ is non-critical can be determined by $\Pi_{\infty}^{R_{\infty}-\an}[\fm_x]$. By Corollary \ref{Cmaxmin}, Theorem \ref{Thodgeauto}, we then obtain
\begin{corollary}Keep the situation, then  $\Pi_{\infty}^{R_{\infty}-\an}[\fm_x]$ determines $\{D_{\sigma}\}_{\sigma\in \Sigma_K}$ for $D=D_{\rig}(\rho)$. In particular, when $K=\Q_p$, it determines $\rho$. 
\end{corollary}
\subsection{Some other cases}\label{S42}
We discuss the local-global compatibility in the space of $p$-adic automorphic representations for certain definite unitary groups (with fewer global hypotheses than \S \ref{S41}). 

\subsubsection{Some formal results}We first discuss some corollaries of the results in \S~\ref{S32}.  Let $D\in \Phi\Gamma_{\nc}(\phi, \textbf{h})$  and $\Ext^1_U(D,D)$ be a certain subspace of $\Ext^1(D,D)$. For $w\in S_n$, set $\Ext^1_{U,w}(D,D):=\Ext^1_U(D,D) \cap \Ext^1_w(D,D)$. We assume the following hypotheses.
\begin{hypothesis}\label{Hglo1}
(1)  $\Ext^1_U(D,D) \cap \Ext^1_g(D,D)=0$.

(2) For $w\in S_n$, $\dim_E\Ext^1_{U,w}(D,D)=nd_K$. 
\end{hypothesis} 
\begin{corollary}\label{CdimU}
	The natural map $\oplus_{w\in S_n} \Ext^1_{U,w}(D,D) \ra \Ext^1_U(D,D)$ is surjective, and  $\dim_E \Ext^1_U(D,D)=\frac{n(n+1)}{2} d_K$.
\end{corollary}
\begin{proof}
We have a commutative diagram
	\begin{equation*}
		\begin{tikzcd}
		&	\oplus_{w\in S_n} \Ext^1_{U,w}(D,D) \arrow[r] \arrow[d,"\sim"] & \Ext^1_U(D,D)  \arrow[d, hook]\\
	&	\oplus_{w\in S_n} \Ext^1_w(D,D)/\Ext^1_g(D,D) \arrow[r, two heads] &\Ext^1(D,D)/\Ext^1_g(D,D)
		\end{tikzcd}
	\end{equation*}
	where the vertical maps are injective  by Hypothesis \ref{Hglo1} hence the left one is bijective by comparing dimensions (cf. Proposition \ref{PDef1} (1)), the surjectivity of the bottom map follows from Proposition \ref{Pinffern}. We deduce the top and right maps are  also surjective. The second part follows then by Proposition \ref{PDef1} (1).
\end{proof}
Denote by $\Ext^1_{U,w}(\pi_{\alg}(\phi, \textbf{h}), \pi_1(\phi, \textbf{h}))$ the image of the composition (see (\ref{Ekappaw000}) and (\ref{Eisomzetaw})): 
$\Ext^1_{U,w}(D,D) \hookrightarrow \ol{\Ext}^1_w(D,D) \xrightarrow[\sim]{\zeta_w\circ \kappa_w} \Ext^1_{w}(\pi_{\alg}(\phi, \textbf{h}), \pi_1(\phi, \textbf{h}))$, 
where the injectivity of the first map follows from  Hypothesis \ref{Hglo1} (1) (recalling $\Ext^1_0(D,D)\subset \Ext^1_g(D,D)$). Denote by $\Ext^1_U(\pi_{\alg}(\phi, \textbf{h}), \pi_1(\phi, \textbf{h}))$ the subspace of $\Ext^1_{\GL_n(K)}(\pi_{\alg}(\phi, \textbf{h}), \pi_1(\phi, \textbf{h}))$ generated by $\Ext^1_{U,w}(\pi_{\alg}(\phi, \textbf{h}), \pi_1(\phi, \textbf{h}))$ for all $w\in S_n$. 

\begin{corollary}
(1) \ The \ map \   
$	t_{\phi, \textbf{h}}: \oplus_{w\in S_n} \Ext^1_{U,w}(\pi_{\alg}(\phi, \textbf{h}), \pi_1(\phi, \textbf{h})) \cong \oplus_{w\in S_n} \Ext^1_{U,w}(D,D) \twoheadrightarrow \Ext^1_U(D,D)
$
(uniquely) factors through a surjection $$	t_{D,U}: \Ext^1_U(\pi_{\alg}(\phi, \textbf{h}), \pi_1(\phi, \textbf{h})) \twoheadlongrightarrow \Ext^1_U(D,D).$$ 

(2) We have $\Ker t_{D,U}\xrightarrow{\sim} \Ker t_D$ (cf. (\ref{EtD000})). 
\end{corollary}
\begin{proof}
We have a commutative diagram
\begin{equation}\label{EdiagU}\small
	\begin{tikzcd}&	\oplus_{w} \Ext^1_{U,w}(D,D) \arrow[r, hook] \arrow[d, two heads] &\oplus_{w} \ol{\Ext}^1_w(D,D) \arrow[r, "\sim"', "(\zeta_w\circ \kappa_w)"] \arrow[d, two heads] & \oplus_w \Ext^1_w(\pi_{\alg}(\phi, \textbf{h}), \pi_1(\phi, \textbf{h})) \arrow[d, two heads] \\
	&	\Ext^1_{U}(D,D) \arrow[r, hook] &\ol{\Ext}^1(D,D) &\Ext^1_{\GL_n(K)}(\pi_{\alg}(\phi, \textbf{h}), \pi_1(\phi, \textbf{h})).\arrow[l, "t_D"', "(\ref{EtD000})"] 
	\end{tikzcd}
\end{equation}
By (\ref{EdiagU}), $t_{D,U}:=t_D|_{\Ext^1_U(\pi_{\alg}(\phi, \textbf{h}), \pi_1(\phi, \textbf{h}))}$ satisfies the property in (1).
We have 
\begin{equation}\label{EKertDU}\Ker t_{D,U}=\Ker t_D \cap \Ext^1_U(\pi_{\alg}(\phi, \textbf{h}), \pi_1(\phi, \textbf{h}))\subset \Ker t_D.
	\end{equation}  
Denote \ by \ $\Ext^1_U(\pi_{\alg}(\phi, \textbf{h}), \PS_1(w(\phi), \textbf{h}))$ \ the \ image \ of \ the \ composition $\Ext^1_{U,w}(D,D)  {\buildrel {\kappa_w} \over \hookrightarrow} \Hom(T(K),E) \xrightarrow[\sim]{(\ref{Eind1})} \Ext^1_{\GL_n(K)}(\pi_{\alg}(\phi, \textbf{h}), \PS_1(w(\phi), \textbf{h}))$, which has dimension $nd_K$ by Hypothesis \ref{Hglo1} (2). By (\ref{Edivips}), we have  an exact sequence
\begin{equation}\label{Edivglob}\small
	0 \lra W \lra \Ext^1_U(\pi_{\alg}(\phi, \textbf{h}), \PS_1(w(\phi), \textbf{h})) \lra  \oplus_{\substack{i=1,\cdots, n-1\\ \sigma\in \Sigma_K}} \Ext^1(\pi_{\alg}(\phi, \textbf{h}), \sC(w(\phi), s_{i,\sigma})),
\end{equation}
where $W$ is a subspace of $\Ext^1_{\GL_n(K)}(\pi_{\alg}(\phi, \textbf{h}), \pi_{\alg}(\phi, \textbf{h}))$. By Hypothesis \ref{Hglo1} (1) ((\ref{Ekappaw000}) and Proposition \ref{Pextalg} (1)), $W\cap \Ext^1_{\lalg}(\pi_{\alg}(\phi, \textbf{h}), \pi_{\alg}(\phi, \textbf{h}))=0$. So $\dim_E W\leq (n+d_K)-n=d_K$.  By comparing dimensions, the last map in (\ref{Edivglob}) must be  surjective and $\dim_E W=d_K$. Similarly by (\ref{Edivi}), $\Ext^1_U(\pi_{\alg}(\phi, \textbf{h}), \pi_1(\phi, \textbf{h}))$ lies in an exact sequence 
\begin{equation}\label{Edivglob2}\small
	0 \lra W' \lra \Ext^1_U(\pi_{\alg}(\phi, \textbf{h}), \pi_1(\phi, \textbf{h})) \lra \oplus_{\substack{i=1,\cdots, n-1, \sigma\in \Sigma_K\\ I \subset \{1,\cdots, n-1\}, \# I =i}} \Ext^1_{\GL_n(K)}(\pi_{\alg}(\phi, \textbf{h}), \sC(I,s_{i,\sigma})),
\end{equation}
where $W'\supset W$ is a subspace of $\Ext^1_{\GL_n(K)}(\pi_{\alg}(\phi, \textbf{h}), \pi_{\alg}(\phi, \textbf{h}))$. By the surjectivity of the last map in (\ref{Edivglob}) and varying $w$ (see also the proof of Proposition \ref{Pextpi1}), the last map of (\ref{Edivglob2}) is surjective as well. Hence $\dim_E \Ext^1_U(\pi_{\alg}(\phi, \textbf{h}), \pi_1(\phi, \textbf{h}))\geq d_K+ (2^n-2)d_K$. As $\dim_E \Ext^1_U(D,D)=\frac{n(n+1)}{2}d_K$ by Corollary \ref{CdimU},  we see $\dim_E  \Ker t_{D,U}\geq (2^n-\frac{n(n+1)}{2}-1) d_K=\dim_E \Ker t_D$. By (\ref{EKertDU}), (2) follows.
\end{proof}

We set  $\pi_1(\phi, \textbf{h})^{\univ}_U$ (resp. $\pi_1(\phi, \textbf{h})^{\univ}_{U,w}$) to be the tautological extension of  
$\Ext^1_U(\pi_{\alg}(\phi, \textbf{h}), \pi_1(\phi, \textbf{h})) \otimes_E \pi_{\alg}(\phi,\textbf{h}) \text{ \big(resp. }\Ext^1_{U,w}(\pi_{\alg}(\phi, \textbf{h}), \pi_1(\phi, \textbf{h})) \otimes_E \pi_{\alg}(\phi,\textbf{h})\text{\big)}$
by $\pi_1(\phi, \textbf{h})$ (cf. \S~\ref{S311}). Let $A_{D,U}$ (resp. $A_{D,U,w}$)  be the quotient of $R_D/\fm^2$ (resp. $R_{D,w}/\fm^2$) associated to $\Ext^1_U(D,D)$ (resp. $\Ext^1_{U,w}(D,D)$). Let $A_{U,w}$ be the quotient of $R_{w(\phi)z^{\textbf{h}}}/\fm^2$ associated to the image $\Ext^1_U(w(\phi) z^{\textbf{h}}, w(\phi) z^{\textbf{h}})$ of $\kappa_w: \Ext^1_{U,w}(D,D) \hookrightarrow \Ext^1_{T(K)}(w(\phi) z^{\textbf{h}}, w(\phi) z^{\textbf{h}})\big(\cong \Hom(T(K),E)\big)$. The map $\kappa_w$ then induces  a natural isomorphism of Artinian $E$-algebras: 
\begin{equation}\label{Ewtring}A_{U,w}\xlongrightarrow{\sim} A_{D,U,w}.
\end{equation} 
Recall $\delta_w=w(\phi)z^{\textbf{h}}(\varepsilon^{-1} \circ \theta)$. We equip $A_{U,w}$ with the $T(K)$-action via
\begin{equation*}
	T(K) \lra R_{\delta_w}/\fm^2 \xlongrightarrow{\sim} R_{w(\phi)z^{\textbf{h}}}/\fm^2 \twoheadlongrightarrow A_{U,w}
\end{equation*}
where the middle isomorphism is induced by twisting by $\varepsilon^{-1} \circ \theta$. The $T(K)$-representation $A_{U,w}^{\vee}$ is  isomorphic to the tautological extension $\widetilde{\delta}_{U,w}^{\univ}$ of $\Ext^1_U(\delta_w, \delta_w) \otimes_E \delta_w$ by $\delta_w$, where $\Ext^1_U(\delta_w,\delta_w)$ consists of characters $\widetilde{\delta}_w$ over $E[\epsilon]/\epsilon^2$ such that $\widetilde{\delta}_w (\varepsilon \circ \theta) \in \Ext^1_U(w(\phi)z^{\textbf{h}}, w(\phi) z^{\textbf{h}})$. Reciprocally, the $T(K)$-representation $\widetilde{\delta}_{U,w}^{\univ}$ is equipped with a natural $A_{U,w}$-action (hence an $A_{D,U,w}$-action via (\ref{Ewtring})) as in the discussion below (\ref{Edeltauniv}) (identifying the tangent space of $A_{U,w}$ with a subspace of that of $R_{\delta_w}$).  Note the natural map $E[T(K)] \ra R_{\delta_w}/\fm^2$ is surjective. Thus the action of $T(K)$ and $A_{U,w}$ actually determine each other. 

Consider 	$I_{B^-(K)}^{\GL_n(K)} \widetilde{\delta}_{U,w}^{\univ}$. By similar arguments as in the proof of Lemma \ref{Lunivw} and (the surjectivity of the last map in) (\ref{Edivglob}),  $I_{B^-(K)}^{\GL_n(K)} \widetilde{\delta}_{U,w}^{\univ}$ is isomorphic to the universal extension of $\pi_{\alg}(\phi, \textbf{h}) \otimes_E \Ext^1_U(\pi_{\alg}(\phi, \textbf{h}), \PS_1(w(\phi), \textbf{h}))$ by $\PS_1(w(\phi),\textbf{h})$. Moreover, similarly as in the discussion below (\ref{Edeltauniv}), we have a $\GL_n(K)\times A_{D,U,w}$-equivariant injection
$I_{B^-(K)}^{\GL_n(K)} \widetilde{\delta}_{U,w}^{\univ} \hookrightarrow \pi_1(\phi, \textbf{h})^{\univ}_{U,w}$,
where the $A_{D,U,w}$-action on the left hand side is induced by its action on $\widetilde{\delta}_{U,w}^{\univ}$ as discussed in the precedent paragraph and the $A_{D,U,w}$-action on the right hand side is given in a similar way as in (\ref{EactAw}) (using also (\ref{Ewtring})).
The following corollary follows by similar arguments as in Theorem \ref{Tuniv} and Corollary \ref{Cpimin}.
\begin{corollary}\label{CpiDUuni}
(1) There is a unique $A_{D,U}$-action on $\pi_1(\phi, \textbf{h})^{\univ}_U$ such that for all $w\in S_n$, there is a $\GL_n(K)\times A_{D,U,w}$-equivariant injection  $\pi_1(\phi, \textbf{h})^{\univ}_{U,w}\hookrightarrow \pi_1(\phi, \textbf{h})^{\univ}_U[\cI_w]$.

(2) We have $\pi_1(\phi, \textbf{h})_U^{\univ}[\fm_{A_{D,U}}]\cong \pi_{\min}(D)$.
\end{corollary}

\subsubsection{Local-global compatibility}\label{S422}
We prove a local-global compatibility result in a non-patched setting. We briefly introduce  the setup and some notation.

Let $F/F^+$ be a CM extension and $G/F^+$ be a unitary group attached to the quadratic extension $F/F^+$ (e.g. see \cite[\S~6.2.1]{BCh}) such that $G\times_{F^+} F\cong \GL_n$ ($n\geq 2$) and $G(F^+\otimes_{\Q} \bR)$ is compact. For a finite place $v$ of $F^+$ which is split in $F$ and $\widetilde{v}$ a place of $F$ dividing $v$, we have isomorphisms $\iota_{\widetilde{v}}: G(F^+_v)\xrightarrow{\sim} G(F_{\widetilde{v}})\xrightarrow{\sim} \GL_n(F_{\widetilde{v}})$. We let $S_p$ denote the set of places of $F^+$ dividing $p$ and we assume that each place in $S_p$ is split in $F$. For each $v\in S_p$, we fix a place $\widetilde{v}$ of $F$ dividing $v$. 

We fix a place $\wp$ of $F^+$ above $p$, and set $K:=F^+_\wp=F_{\widetilde{\wp}}$.  We have thus an isomorphism $G(F^+_\wp)\xrightarrow{\sim} \GL_n(K)$. For each $v\in S_p$, $v\neq \wp$, let $\xi_v$ be a dominant weight of $\Res_{\Q_p}^{F^+_v} \GL_n$, and $\tau_v: I_{F_v^+} \ra \GL_n(E)$ be an inertial type. To $\tau_v$, one can associate a smooth irreducible representation $\sigma(\tau_v)$ of $\GL_n(\co_{F_v^+})$ over $E$ (see for example \cite[Thm~3.7]{CEGGPS1}). Let $W_{\xi,\tau}$ be a $\GL_n(\co_{F_v^+})$-invariant $\co_E$-lattice of the locally algebraic representation $\sigma(\tau_v) \otimes_E L(\xi_v)$ (see also \cite[\S~2.3]{CEGGPS1}). 

Let $U^{\wp}=U^{p} U^{\wp}_p=\prod_{v\nmid p} U_v \times \prod_{v\in S_p \setminus \{\wp\}} U_v$ be a sufficiently small (cf. \cite{CHT}) compact open subgroup of $G(\bA_{F^+}^{\infty, \wp})$ with $\iota_{\widetilde{v}}(U_v)= \GL_n(\co_{F_v^+})$ for $v\in S_p\setminus \{\wp\}$.  We also assume that $U_v$ is hyperspecial if $v$ is inert in $F$. Let $S$ be the union of $S_p$ and of the places $v\notin S_p$ such that $U_v$ is not hyperspecial. 

For  $k\in \Z_{\geq 1}$ and a compact open subgroup $U_{\wp}$ of $G(\co_{F^+_{\wp}})$, consider the $\co_E/\varpi_E^k$-module 
$S_{\xi,\tau}(U^{\wp} U_{\wp}, \co_E/\varpi_E^k)
=\{f: G(F^+)\backslash G(\bA_{F^+}^{\infty}) \ra W_{\xi,\tau}/\varpi_E^k\ |\ f(gu)=u^{-1} f(g), \ \forall g\in G(\bA_{F^+}^{\infty}), u\in U^{\wp}U_{\wp}\}$
where $U^{\wp}U_{\wp}$ acts on $W_{\xi,\tau}/\varpi_E^k$ via  $U^{\wp}U_{\wp} \twoheadrightarrow \prod_{v\in S_p \setminus \{\wp\}} U_v$.  Put 
\[\widehat{S}_{\xi,\tau}(U^{\wp}, \co_E):=\varprojlim_k  S_{\xi,\tau}(U^{\wp}, \co_E/\varpi_E^k):=\varprojlim_k \varinjlim_{U_{\wp}} S_{\xi,\tau}(U^{\wp}U_{\wp}, \co_E/\varpi_E^k),\] 
and $\widehat{S}_{\xi,\tau}(U^{\wp},E):=\widehat{S}_{\xi,\tau}(U^{\wp}, \co_E) \otimes_{\co_E} E$. Then  $\widehat{S}_{\xi,\tau}(U^{\wp},E)$ is an admissible unitary Banach representation of $\GL_n(K)$. 
Recall that 
$\widehat{S}_{\xi,\tau}(U^{\wp},E)$ is equipped  with a natural action of $\bT(U^{\wp})$ commuting with $\GL_n(K)$, where $\bT(U^{\wp})$ is the polynomial $\co_E$-algebra generated by Hecke operators: $T_{\widetilde{v}}^{(j)}=\bigg[U_v \iota_{\widetilde{v}}^{-1}\begin{pmatrix}
	1_{n-j} & 0 \\ 0 &\varpi_{\widetilde{v}}1_j
\end{pmatrix} U_v\bigg]$,  for $v\notin S$ which splits to $\widetilde{v}\widetilde{v}^c$ in $F$ and $j=1, \cdots, n$, where $\varpi_{\widetilde{v}}$ is a uniformizer of $F_{\widetilde{v}}$.

Using Emerton's method \cite[(2.3)]{Em04}, one can  construct an eigenvariety $\cE(U^{\wp})$ from $J_B(\widehat{S}_{\xi,\tau}(U^{\wp},E)^{\Q_p-\an})$. There is a natural morphism of rigid spaces $\kappa: \cE(U^{\wp}) \ra \widehat{T}$. The strong dual $J_B(\widehat{S}_{\xi, \tau}(U^{\wp},E)^{\Q_p-\an})^{\vee}$ gives rise to a coherent sheaf $\cM(U^{\wp})$ over $\cE(U^{\wp})$.  An $E$-point of $\cE(U^{\wp})$ can be parametrized by $(\delta, \omega)$ where $\delta$ is a continuous character of $T(K)$, and $\omega$ is a morphism of $E$-algebras $\bT(U^{\wp})\twoheadrightarrow E$ which corresponds to a maximal ideal $\fm_{\omega}$ of $\bT(U^{\wp})$. Moreover, $(\delta, \omega)\in \cE(U^{\wp})$ if and only if $\Hom_{T(K)}(\delta, J_B(\widehat{S}_{\xi, \tau}(U^{\wp},E)^{\Q_p-\an}[\fm_{\omega}]))\neq 0$. Recall a point $(\delta, \omega)\in \cE(U^{\wp})$ is called \textit{classical} if $\Hom_{T(K)}(\delta, J_B(\widehat{S}_{\xi, \tau}(U^{\wp},E)^{\lalg}[\fm_{\omega}]))\neq 0$. In fact, by \cite[(6.3)]{BD1}, such points are associated to classical automorphic representations. Note also if $(\delta, \omega)$ is classical, then $\delta$ is locally algebraic hence has the form $\delta_{\sm} \delta_{\alg}$, where $\delta_{\alg}$ is an algebraic character of $T(K)$. We call a classical point $(\delta, \omega)$ generic if $\delta_{\sm} \delta_B^{-1} |\cdot|_K\circ \theta=:(\phi_i')$ is generic, i.e. $\phi_i'(\phi_j')^{-1}\neq 1, |\cdot|_K$ for $i\neq j$.

The following proposition is well-known.
\begin{proposition}\label{Peigen}
	(1) $\cE(U^{\wp})$ is equidimensional of dimension $nd_K$.
	
	(2) The coherent sheaf $\cM(U^{\wp})$ is Cohen-Macaulay over $\cE(U^{\wp})$.

	(3) $\cE(U^{\wp})$ is reduced.
\end{proposition}
\begin{proof}By \cite[Lem.~6.1]{BD1}, for a compact open subgroup $H$ of $\GL_n(\co_K)$, we have $\widehat{S}_{\xi,\tau}(U^{\wp},E)|_H\cong \cC(H,E)^{\oplus s}$ for some $s\geq 1$ (where $\cC(H,E)$ denotes the space of continuous functions on $H$).
	(1) (2) then  follows verbatim from \cite[Lem.~3.10, Prop.~3.11, Cor.~3.12]{BHS1}\cite[Lem.~3.8]{BHS2}, applying \cite[\S~5.2]{BHS1} to $\Pi:=\widehat{S}_{\xi,\tau}(U^{\wp},E)^{\Q_p-\an}$. (3) follows by the same argument as in \cite[Prop. 3.9]{Che05} (see also \cite[Cor.~3.20]{BHS1}).
\end{proof}
Let $F^{S}$ be the maximal algebraic extension of $F$ unramified outside the places dividing those in $S$, and $\Gal_{F,S}:=\Gal(F^S/F)$. Let $\rho: \Gal_{F,S}	\ra \GL_n(E)$ be a continuous representation satisfying $\rho^c\cong \rho^{\vee} \otimes_E \varepsilon^{1-n}$ where $\rho^c(g):=\rho(cgc)$ for $g\in \Gal_{F,S}$ with $c$ being the complex conjugation. To $\rho$, one naturally associates a maximal ideal $\fm_{\rho}$ of $\bT(U^{\wp})$ generated by $((-1)^j (\#k_{\widetilde{v}})^{j(j-1)/2}T_{\widetilde{v}}^{(j)}-a_{\widetilde{v}}^{(j)})$, where $k_{\widetilde{v}}$ is the residue field of $F_{\widetilde{v}}$, and the characteristic polynomial of $\rho(\Frob_{\widetilde{v}})$ (for a geometric Frobenius $\Frob_{\widetilde{v}}$ at $\widetilde{v}$) is given by $X^n+a_{\widetilde{v}}^{(1)} X^{n-1}+\cdots+a_{\widetilde{v}}^{(n-1)} X+a_{\widetilde{v}}^{(n)}$. Let $\omega_{\rho}$ denote the morphism $\bT(U^{\wp})\twoheadrightarrow  \bT(U^{\wp})/\fm_{\rho}\cong E$.  Assume $\widehat{S}_{\xi,\tau}(U^{\wp},E)^{\lalg}[\fm_{\rho}]\neq 0$ and $D:=D_{\rig}(\rho_{\widetilde{\wp}})\in \Phi\Gamma_{\nc}(\phi, \textbf{h})$ (with $\phi$ generic,  and $\textbf{h}$ strictly dominant, cf. \S~\ref{S2.1}), where $\rho_w:=\rho|_{\Gal_{F_w}}$ for a place $w$ of $F$. There exists hence $r\in \Z_{\geq 1}$ such that (e.g. by \cite[Thm.~1.1]{BLGGTII} and \cite[(6.3)]{BD1})
\begin{equation}\label{Elalg2}
	\pi_{\alg}(\phi, \textbf{h})^{\oplus r} \xlongrightarrow{\sim} \widehat{S}_{\xi,\tau}(U^{\wp},E)[\fm_{\rho}]^{\lalg}.
\end{equation}
Taking Jacquet-Emerton modules, we see $z_w:=(\delta_w\delta_B, \omega_{\rho})\in \cE(U^{\wp})$ for $w\in S_n$ (with $\delta_w=w(\phi) z^{\textbf{h}}(\varepsilon^{-1}\circ \theta)$). Moreover, similarly as in Lemma \ref{Lnocompa}, using the global triangulation theory (cf. \cite{KPX}\cite{Liu}), $(\delta,  \omega_{\rho})\in \cE(U^{\wp})$ if and only if $\delta=\delta_w\delta_B$ for some $w\in S_n$ (cf. \cite[Prop.~9.2]{Br13II}).

Let $\Ext^1_U(\rho,\rho)$ be the subspace of $\Ext^1_{\Gal_{F,S}}(\rho,\rho)$ consisting of $\widetilde{\rho}$ such that $\widetilde{\rho} ^c \cong \widetilde{\rho}^{\vee} \otimes_E \varepsilon^{1-n}$.
For $v\in S_p$, we have a natural map $\Ext^1_{U}(\rho, \rho) \ra \Ext^1_{\Gal_{F_{\widetilde{v}}}}(\rho_{\widetilde{v}}, \rho_{\widetilde{v}})$. Set $$\Ext^1_{g, S_p\setminus\{\wp\}}(\rho,\rho):=\Ker\big[\Ext^1_{U}(\rho, \rho) \ra \prod_{v\in S_p\setminus \{\wp\}} \Ext^1_{\Gal_{F_{\widetilde{v}}}}(\rho_{\widetilde{v}}, \rho_{\widetilde{v}})/\Ext^1_{g}(\rho_{\widetilde{v}}, \rho_{\widetilde{v}})\big]. $$
Let $\Ext^1_U(D,D)$ be the image of the subspace $\Ext^1_{g, S_p\setminus\{\wp\}}(\rho,\rho)$ via $\Ext^1_{U}(\rho, \rho) \ra \Ext^1_{\Gal_{F_{\widetilde{\wp}}}}(\rho_{\widetilde{\wp}}, \rho_{\widetilde{\wp}})\xrightarrow{\sim} \Ext^1(D,D)$. 

 We make the following vanishing hypothesis on the adjoint Selmer group:
\begin{hypothesis}\label{Hselmer}
	Suppose the composition  $\Ext^1_{g, S_p\setminus\{\wp\}}(\rho,\rho) \ra \Ext^1(D,D) \ra \Ext^1(D,D)/\Ext^1_g(D,D)$ is injective. In particular, $\Ext^1_U(D,D)\cap \Ext^1_g(D,D)=0$.
\end{hypothesis} 
 \begin{remark}
 	Hypothesis \ref{Hselmer} is known to hold in many cases, see \cite[Thm.~A]{All16} \cite[Cor.~4.12]{BHS1} \cite[Thm.~A]{NT}.
 \end{remark}
Let $R_{\rho,U}$ be the universal deformation ring of deformations $\rho_A$ of the $\Gal_{F,S}$-representation $\rho$ over local Artinian $E$-algebras $A$ satisfying $\rho_A^c \cong \rho_A^{\vee}\otimes_E \varepsilon^{1-n}$. Note $R_{\rho,U}$ exists as $\End_{\Gal_{F,S}}(\rho)\hookrightarrow \End(D)\cong E$.  Let $\fa_{\rho}\supset \fm_{R_{\rho,U}}^2$ (resp. $\fa_D \supset \fm_{R_D}^2$) be the ideal associated to $\Ext^1_{g,S_p\setminus\{\wp\}}(\rho,\rho)$ (resp. $\Ext^1_U(D,D)$). By Hypothesis \ref{Hselmer}, $\Ext^1_{g, S_p\setminus\{\wp\}}(\rho,\rho)\xrightarrow{\sim} \Ext^1_U(D,D)$. The natural morphism $R_D \ra R_{\rho,U}$ induces hence an isomorphism (of local Artinian $E$-algebras) $A_{D,U}=R_D/\fa_D \xrightarrow{\sim} R_{\rho,U}/\fa_{\rho}$. Let $\widetilde{\rho}_{R_{\rho,U}/\fa_{\rho}}$ be the universal deformation of $\rho$ over $R_{\rho,U}/\fa_{\rho}$. We have a natural morphism $\bT(U^{\wp}) \ra  R_{\rho,U}/\fa_{\rho}$ sending $T_{\widetilde{v}}^{(j)}$ to $(-1)^j (\#k_{\widetilde{v}})^{-j(j-1)/2}\widetilde{a}_{\widetilde{v}}^{(j)}$ where $\widetilde{a}_{\widetilde{v}}^{(j)}\in R_{\rho,U}/\fa_{\rho}$, $j=1, \cdots, n$,  satisfy that the characteristic polynomial of $\widetilde{\rho}_{R_{\rho,U}/\fa_{\rho}}(\Frob_{\widetilde{v}})$ is given by $X^n+\widetilde{a}_{\widetilde{v}}^{(1)} X^{n-1}+\cdots+\widetilde{a}_{\widetilde{v}}^{(n-1)} X+\widetilde{a}_{\widetilde{v}}^{(n)}$. Let $\fa_T$ be its kernel. The induced morphism $\bT(U^{\wp})/\fa_T \ra R_{\rho,U}/\fa_{\rho}$ is an isomorphism. Indeed, it suffices to show the morphism sends $\fm_{\rho}$ onto $\fm_{R_{\rho,U}}/\fa_{\rho}$. Consider the universal representation $\widetilde{\rho}$ of $\Gal_{F,S}$ over $R_{\rho,U}/\fm_{\rho}$. By the definition of the morphism, we deduce the associated pseudo-character $\tr(\widetilde{\rho})$ takes values  in $E$. This implies $\tr(\widetilde{\rho})$ is a trivial deformation of $\tr(\rho)$. As $\rho$ is absolutely irreducible, deforming $\rho$ is equivalent to deforming $\tr(\rho)$ (cf. \cite[Thm.~1]{Nys96}). We deduce $\widetilde{\rho}$ is a trivial deformation of $\rho$ hence $\fm_{\rho}(R_{\rho,U}/\fa_{\rho})=\fm_{R_{\rho,U}}/\fa_{\rho}$.
\begin{proposition}\label{Psmooth}
	Assume Hypothesis \ref{Hselmer} and $\rho$ is absolutely irreducible. Then $\cE(U^{\wp})$ is smooth at $z_w$ for all $w\in S_n$. Moreover, $\Ext^1_U(D,D)$ satisfies Hypothesis \ref{Hglo1} (1) (2). 
\end{proposition}
\begin{proof}There is a natural  family of $\Gal_{F,S}$-representations on $\cE(U^{\wp})$. We quickly recall some of its properties that we need.
	Let $X\subset \cE(U^{\wp})$ be a sufficiently small affinoid neighbourhood of $x_w$ such that $x_{w'}\notin X$ for $w'\neq w$, and that the generic classical points are Zariski dense in $X$. Recall to each classical point $z\in X$, we can associate an $n$-dimensional continuous representation $\rho_z$ of $\Gal_{F,S}$, hence an $n$-dimension pseudo-character of $\Gal_{F,S}$. By \cite[Prop.~7.1.1]{Che04}, there is a pseudo-character $\cT_X: \Gal_{F,S}\ra \co(X)$ interpolating those associated to the classical points. By \cite[Lem.~5.5]{Bergd13} and the assumption $\rho$ is absolutely irreducible, shrinking $X$ if needed, $\cT_X$ gives rise to a continuous representation $\rho_{X}: \Gal_{F,S} \ra \GL_n(\co(X))$ satisfying that for all points $z\in X$, $\rho_z=z^* \rho_X$ is absolutely irreducible. Let $\delta_{X}=(\delta_{X,1},\cdots, \delta_{X,n}): T(K) \ra \co(X)^{\times}$ be the natural character associated to $\kappa$. Let $\cR_{K,X}$ be the relative Robba ring over $\co(X)$ for $K$ (cf. \cite[Def.~2.2.2]{KPX}). By shrinking $X$ if needed, $\rho_X$ has the following properties:
	\begin{enumerate}
		\item[(i)] $\rho_{X}^c \cong \rho_{X} \otimes_E \varepsilon^{1-n}$. 
		\item[(ii)] For $v\in S_p\setminus \{\wp\}$, $\rho_{X, \widetilde{v}}$ is de Rham of Hodge-Tate weights $\xi_v-\theta^{F_{v}^+}$.
		\item[(iii)] The $(\varphi, \Gamma)$-module $D_{\rig}(\rho_{X,\widetilde{\wp}})$ over $\cR_{K,X}$ is isomorphic to a successive extension of the rank one $(\varphi, \Gamma)$-modules  $\cR_{K,X}(\delta_{X,i}|\cdot|_K^{2i-(n+1)}\varepsilon^{1-i})$ (cf. \cite[Def.~6.2.1]{KPX}) for $i=1,\cdots, n$. 
	\end{enumerate} 
	(i) follows easily from the fact that for all the classical points, $\rho_z^c \cong \rho_z \otimes_E \varepsilon^{1-n}$. For a place $v\in S_p\setminus \{\wp\}$, $\rho_{z,\widetilde{v}}$ is de Rham (of inertial type $\tau_v$) of Hodge-Tate weights $\xi_v-\theta^{F_{v}^+}$ for all classical points $z$. (ii)  follows then by \cite[Thm.~B]{BC08}. Finally, as $D:=D_{\rig}(\rho_{x_w})$ is non-critical and ($\varphi$-)generic, by \cite[Thm.~5.3]{Bergd17} (and an easy induction argument), (iii) follows (by shrinking $X$ if needed).

		Let $T_{z_w}$ be the tangent space of $\cE(U^{\wp})$ at the point $z_w$. By (i), we have a map $T_{z_w} \ra \Ext^1_U(\rho, \rho)$, sending $t: \Spec E[\epsilon]/\epsilon^2 \ra X$ to $t^* \rho_{X}$.  Denote by $\kappa_{z_w}: T_{z_w} \ra \Ext^1_{T(K)}(\delta_w\delta_B, \delta_w\delta_B)$ the tangent map of $\cE(U^{\wp}) \ra \widehat{T}$ at $z_w$. 

	\noindent \textbf{Claim.} The induced map $f_{z_w}: T_{z_w} \ra \Ext^1_U(\rho, \rho)$ is injective and has image in $\Ext^1_{g, S_p\setminus \{\wp\}}(\rho,\rho)$. 
	
	Let $\nu$ be in the kernel. So $\nu^* \rho_{X}\cong \rho \oplus \rho$. This implies the composition $\bT(U^{\wp}) \ra \co(X)\xrightarrow{\nu} E[\epsilon]/\epsilon^2$ factors through $\fm_{\rho}$. However, by (iii),  $\kappa_{z_w}(\nu)$ (as a character of $T(K)$ over $E[\epsilon]/\epsilon^2$) is a trianguline parameter of the $(\varphi, \Gamma)$-module $D_{\rig}(\nu^*\rho_{X})$ over $\cR_{K,E[\epsilon]/\epsilon^2}$. Hence $\kappa_{z_w}(\nu)=0$. But by the construction of $\cE(U^{\wp})$, $\bT(U^{\wp})\otimes_E E[T(K)]$ is dense in $\co(X)$, hence the map $T_{z_w} \xrightarrow{(f_{z_w}, \kappa_{z_w})} \Ext^1_U(\rho, \rho) \times \Ext^1_{T(K)}(\delta_w\delta_B, \delta_w\delta_B)$ is injective.  We deduce $\nu$ is zero. The second part of the claim follows from (ii).

	By (iii), the composition  $f_D: T_{z_w} \ra \Ext^1_U(\rho,\rho)\ra \Ext^1(D,D)$ has image in $\Ext^1_w(D,D)$ hence (by the claim) in $\Ext^1_{U,w}(D,D)$. Together with Hypothesis \ref{Hselmer}, we deduce $\dim_E T_{z_w} \leq \dim \Ext^1_{U,w}(D,D)\leq \dim_E \Ext^1_w(D,D)/\Ext^1_g(D,D)=nd_K$. As $\dim \cE(U^{\wp})=nd_K$, we see $z_w$ is a smooth point and $\dim_E \Ext^1_{U,w}(D,D)=nd_K$. This finishes the proof. 
\end{proof}

By Proposition \ref{Peigen} (2), Proposition \ref{Psmooth} and (\ref{Elalg2}), $\cM$ is locally free of rank $r$ at each $z_w$ for $w\in S_n$. Let $X$ be a sufficiently small smooth affinoid neighbourhood of $z_w$, and $\fm_{z_w}\subset \co(X)$ be  the maximal ideal associated to $z_w$. We use the notation of \S~\ref{S42}: $A_{D,U}$, $A_{D,U,w}$, $A_{U,w}$ etc. By the proof of Proposition \ref{Psmooth}, the composition $A_{D,U,w}  \hookrightarrow A_{D,U}=R_{D}/\fa_D \xrightarrow{\sim} R_{\rho,U}/\fa_{\rho} \twoheadrightarrow \co(X)/\fm_{z_w}^2$ is an isomorphism (where the last map is obtained by sending an element in the tangent space of $X$ at $z_w$ to the associated $\Gal_{F,S}$-representation). The map $A_{U,w}\xrightarrow{\sim} A_{D,U,w} \xrightarrow{\sim} \co(X)/\fm_{z_w}^2$ coincides with the natural map induced by $\kappa$, and the map $\bT(U^{\wp})/\fa_T\xrightarrow{\sim}R_{\rho,U}/\fa_{\rho}\twoheadrightarrow \co(X)/\fm_{z_w}^2$ coincides with the one induced by $\bT(U^{\wp}) \ra \co(X)$.
We deduce a
$T(K)\times A_{D,w,U}$-equivariant injection
\begin{equation}\label{Ejac2}
\widetilde{\delta}_{U,w}^{\univ,\oplus r}\cong 	(\cM/\fm_{z_w}^2)^{\vee} \hooklongrightarrow J_B(\widehat{S}_{\xi,\tau}(U^{\wp},E)_{\overline{\rho}}^{\Q_p-\an})[\fa_T][\cI_w]\{T(K)=\delta_w \delta_B\},
\end{equation}
where the $A_{D,w,U}$-action on the left hand side is given as in the discussion below (\ref{Ewtring}) and it acts on the right hand side via $A_{D,w,U} \cong  A_{D,U}/\cI_w\twoheadleftarrow A_{D,U}\cong R_{D,U}/\fa_{\rho}\cong \bT(U^{\wp})/\fa_T$. Note  as in the discussion below (\ref{Ewtring}), the action of  $A_{D,w,U}$ and $T(K)$ determine each other. 
Similarly in Lemma \ref{Lbala}, the map (\ref{Ejac2}) is balanced and induces (by \cite[Thm.~0.13]{Em2}) a $\GL_n(K) \times R_{D,w,U}$-equivariant injection
\begin{equation}\label{Eiotawg}
\iota_w: 	(I_{B^-(K)}^{\GL_n(K)} 	\widetilde{\delta}_{U,w}^{\univ})^{\oplus r}  \hooklongrightarrow \widehat{S}_{\xi,\tau}(U^{\wp},E)_{\overline{\rho}}^{\Q_p-\an}[\fa_T].
\end{equation}
Let $\widetilde{\pi}$ be the closed subrepresentation of $\widehat{S}_{\xi,\tau}(U^{\wp},E)_{\overline{\rho}}^{\Q_p-\an}[\fa_T]$ generated by $\Ima \iota_w$ for all $w$. Note $\widetilde{\pi}$ inherits from $\widehat{S}_{\xi,\tau}(U^{\wp},E)_{\overline{\rho}}^{\Q_p-\an}[\fa_T]$ a (global) $A_{D,U}$ ($\cong \bT(U^{\wp})/\fa_T$)-action. 
\begin{theorem}\label{TlocgloU}Suppose Hypothesis \ref{Hselmer} and $\rho$ is absolutely irreducible. 
We have a $\GL_n(K)\times A_{D,U}$-equivariant isomorphism $\widetilde{\pi}\cong \pi_1(\phi, \textbf{h})_U^{\univ, \oplus r}$ (cf. Corollary \ref{CpiDUuni}). Consequently, we have $\pi_{\min}(D)^{\oplus r} \hookrightarrow \widehat{S}_{\xi,\tau}(U^{\wp},E)^{\Q_p-\an}[\fm_{\rho}]$.
\end{theorem}
\begin{proof}
We first show $\widetilde{\pi}\cong \pi_1(\phi, \textbf{h})_U^{\univ, \oplus r}$ as $\GL_n(K)$-representation. By (\ref{Eiotawg}) and similar arguments as in  the proof of Corollary \ref{Clg1} (or using the same argument as in the proof of \cite[Thm.~5.12]{BH2}), the injection $\pi_{\alg}(\phi, \textbf{h})^{\oplus r} \hookrightarrow\widehat{S}_{\xi,\tau}(U^{\wp},E)^{\Q_p-\an}[\fm_{\rho}]$ extends uniquely to an injection $\pi_1(\phi, \textbf{h})^{\oplus r} \hookrightarrow \widehat{S}_{\xi,\tau}(U^{\wp},E)^{\Q_p-\an}[\fm_{\rho}]$. Note that  $\Ima \iota_w \cap \pi_1(\phi, \textbf{h})^{\oplus r} \cong \PS_1(w(\phi), \textbf{h})^{\oplus r}$ (by the same argument as in the proof of Corollary \ref{Clg1}). As in the proof of Theorem \ref{Tlg}, $\widetilde{\pi}$ is isomorphic to an extension of certain copies of $\pi_{\alg}(\phi, \textbf{h})$ by $\pi_1(\phi, \textbf{h})^{\oplus r}$. Using (\ref{Elalg2}) (\ref{Eiotawg}) (and the structure of $\pi_1(\phi, \textbf{h})_U^{\univ}$), it is not difficult to see $\widetilde{\pi}$ has to be isomorphic to $\pi_1(\phi, \textbf{h})_U^{\univ, \oplus r}$. 

For the compatibility of the $A_{D,U}$-action, it suffices to show any injection
$\iota: \pi_1(\phi, \textbf{h})_U^{\univ} \hookrightarrow\widehat{S}_{\xi,\tau}(U^{\wp},E)^{\Q_p-\an}[\fa_T]$ (extending $\pi_1(\phi, \textbf{h})\hookrightarrow\widehat{S}_{\xi,\tau}(U^{\wp},E)^{\Q_p-\an}[\fm_{\rho}]$) is $A_{D,U}$-equivariant. As $\pi_1(\phi, \textbf{h})_U^{\univ}$ is generated by $I_{B^-(K)}^{\GL_n(K)} \widetilde{\delta}_{U,w}^{\univ}$, it suffices to show the restriction of $\iota$ to $I_{B^-(K)}^{\GL_n(K)} \widetilde{\delta}_{U,w}^{\univ}$ is $A_{D,U}$-equivariant. The restriction of $\iota$ to $I_{B^-(K)}^{\GL_n(K)} \widetilde{\delta}_{U,w}^{\univ}$ corresponds to a unique $T(K)$-equivariant injection 
\begin{equation}\label{Eiotaglob}
	\widetilde{\delta}_{U,w}^{\univ}\delta_B \hooklongrightarrow J_B(\widehat{S}_{\xi,\tau}(U^{\wp},E)^{\Q_p-\an}[\fa_{\rho}]).
\end{equation}
Its  image is clearly contained in the image of (\ref{Ejac2}) (by the definition of $\cM$). However, any $T(K)$-equivariant injection $\widetilde{\delta}_{U,w}^{\univ} \delta_B\hookrightarrow  (\widetilde{\delta}_{U,w}^{\univ} \delta_B)^{\oplus r}$ has to be $A_{D,U,w}$-equivariant (by the discussion below (\ref{Ewtring})), so is (\ref{Eiotaglob}). Thus $\iota|_{I_{B^-(K)}^{\GL_n(K)} \widetilde{\delta}_{U,w}^{\univ}}$ is $A_{D,U,w}$-equivariant for all $w$, so $\iota$ is  $A_{D,U}$-equivariant. The second part follows from Corollary \ref{CpiDUuni} (2).
\end{proof}
\begin{remark}By the same argument as in the proof of \cite[Thm.~5.12]{BH2}, the injection $\pi_{\alg}(\phi, \textbf{h})^{\oplus r} \hookrightarrow \widehat{S}_{\xi,\tau}(U^{\wp},E)^{\Q_p-\an}[\fm_{\rho}]$ uniquely extends to $\pi(\phi, \textbf{h})^{\oplus r}\hookrightarrow \widehat{S}_{\xi,\tau}(U^{\wp},E)^{\Q_p-\an}[\fm_{\rho}]$. \ Similarly \ as \ in \ Remark \ \ref{Rpifs2}, \ we \ see \ the \ injection $\pi_{\min}(D)^{\oplus r} \hookrightarrow \widehat{S}_{\xi,\tau}(U^{\wp},E)^{\Q_p-\an}[\fm_{\rho}]$ (in Theorem \ref{TlocgloU}) extends uniquely to $\pi_{\fss}(D)^{\oplus r} \hookrightarrow \widehat{S}_{\xi,\tau}(U^{\wp},E)^{\Q_p-\an}[\fm_{\rho}]$.
\end{remark}
By similar arguments as in Corollary \ref{Cmaxmin} (replacing Lemma \ref{Lfax} (2) by (\ref{Elalg2})), we have:
\begin{corollary}\label{Cmaxmin1}
The representation $\pi_{\min}(D)^{\oplus r}$ is the maximal subrepresentation of $\widehat{S}_{\xi,\tau}(U^{\wp},E)^{\Q_p-\an}[\fm_{\rho}]$, which is  generated by extensions of $\pi_{\alg}(\phi, \textbf{h})$ by $\pi_1(\phi, \textbf{h})$.
\end{corollary}
The information that $D$ is non-critical can be read out from $\widehat{S}_{\xi,\tau}(U^{\wp},E)[\fm_{\rho}]$ by \cite[Thm.~9.3]{Br13II}. Together with Corollary \ref{Cmaxmin1} and Theorem \ref{Thodgeauto}, we get:
\begin{corollary}
For $D'\in \Phi\Gamma_{\nc}(\phi, \textbf{h})$, $\pi_{\min}(D')\hookrightarrow \widehat{S}_{\xi,\tau}(U^{\wp},E)[\fm_{\rho}]$ if and only if $D'_{\sigma}\cong D_{\sigma}$ for all $\sigma\in \Sigma_K$. In particular, when $K=\Q_p$, the $\GL_n(\Q_p)$-representation $\widehat{S}_{\xi,\tau}(U^{\wp},E)^{\Q_p-\an}[\fm_{\rho}]$ determines $\rho_{\widetilde{\wp}}$.
\end{corollary}

\end{document}